\documentclass{article}
\usepackage[utf8]{inputenc}
\pdfoutput=1 

\title{Trees with many leaves in tournaments}
\author{}
\usepackage{preamble,pictures,hyperref,comment}

\newcommand\itref[1]{\textit{\ref{#1}}}

\newcounter{propcounter}

\newcommand{\AB}[1]{\textcolor{blue}{ $\blacktriangleright$\ {\sf AB: #1}\  $\blacktriangleleft$ }}

\title{\vspace{-0.7cm}Trees with many leaves in tournaments}
\author{Alistair\ Benford\thanks{School of Mathematics, University of Birmingham,
Birmingham,
B15 2TT,
UK.
a.s.benford@bham.ac.uk}
\ and\ Richard\ Montgomery\thanks{Mathematics Institute, University of Warwick, Coventry, CV4 7AL, UK. 
richard.montgomery@warwick.ac.uk. Supported by the European Research Council (ERC) under the European Union Horizon 2020 research and innovation programme (grant agreement No.\ 947978) and the Leverhulme Trust (grant agreement No.\ PLP-2020-183).}}
\date{}

\begin{document}

\maketitle

\begin{abstract}
Sumner's universal tournament conjecture states that every $(2n-2)$-vertex tournament should contain a copy of every $n$-vertex oriented tree. If we know the number of leaves of an oriented tree, or its maximum degree, can we guarantee a copy of the tree with fewer vertices in the tournament? Due to work initiated by H\"aggkvist and Thomason (for number of leaves) and K\"uhn, Mycroft and Osthus (for maximum degree), it is known that improvements can be made over Sumner's conjecture in some cases, and indeed sometimes an $(n+o(n))$-vertex tournament may be sufficient.

In this paper, we give new results on these problems.
Specifically, we show

\begin{enumerate}[label = \roman{enumi})]
\item for every $\alpha>0$, there exists $n_0\in\mathbb{N}$ such that, whenever $n\geqslant n_0$, every $((1+\alpha)n+k)$-vertex tournament contains a copy of every $n$-vertex oriented tree with $k$ leaves, and

\item for every $\alpha>0$, there exists $c>0$ and $n_0\in\mathbb{N}$ such that, whenever $n\geqslant n_0$, every $(1+\alpha)n$-vertex tournament contains a copy of every $n$-vertex oriented tree with maximum degree $\Delta(T)\leqslant cn$.
\end{enumerate}
\noindent Our first result gives an asymptotic form of a conjecture by Havet and Thomass\'e, while the second improves a result of Mycroft and Naia which applies to trees with polylogarithmic maximum degree.
\end{abstract}

\section{Introduction}\label{sect:introduction}

When the edges of a complete graph are oriented in any manner, giving a tournament, which oriented trees must appear within its edges? The study of this question has been motivated by Sumner's universal tournament conjecture from 1971, which states that every $(2n-2)$-vertex tournament should contain a copy of every $n$-vertex oriented tree (see, e.g.,~\cite{reid1983embedding}). The extremal examples showing this conjecture would be tight, the $n$-vertex stars whose root vertex has in- or out-degree 0, also maximise the number of leaves and the maximum degree among the $n$-vertex oriented trees. In this paper, we consider whether fewer vertices are required in the tournament for trees with fewer leaves or a lower maximum degree.

The first major step towards Sumner's conjecture was taken by H\"aggkvist and Thomason~\cite{HAE-THO} in 1991, who showed that $O(n)$ vertices in a tournament are sufficient to find a copy of any $n$-vertex oriented tree. The constant implicit in this result has been improved in the intervening years (see~\cite{havet2000median,HAV,El_S}), with the best current bound applicable for all $n$ by Dross and Havet~\cite{DRO-HAV}, who showed that any $\big\lceil\frac{21}{8}n-\frac{47}{16}\big\rceil$-vertex tournament contains a copy of any $n$-vertex oriented tree. Significantly, however, Sumner's conjecture has been proved exactly for all sufficiently large $n$, by K\"uhn, Mycroft and Osthus~\cite{KUE-MYC-OST-2}, so that the conjecture remains open for only finitely many oriented trees.

That trees with fewer leaves generally require fewer vertices in the tournament was also first demonstrated by H\"aggkvist and Thomason~\cite{HAE-THO} in 1991. That is, they showed that there is some smallest $g(k)$ such that every $(n+g(k))$-vertex tournament contains a copy of every $n$-vertex oriented tree with $k$ leaves. Due to Thomason~\cite{THO}, it is known that every $(n+1)$-vertex tournament contains a copy of every $n$-vertex oriented path, and so $g(2)=1$ (noting that we must have $g(k)\geqslant k-1$ as a consequence of the previous example of an $n$-vertex star). Motivated in part by this result, Havet and Thomass\'e~\cite{HAV2} generalised Sumner's conjecture, suggesting that $g(k)=k-1$, that is, that every $(n+k-1)$-vertex tournament should contain a copy of every $n$-vertex oriented tree with $k$ leaves.

Improving the initial bound given by H\"aggkvist and Thomason~\cite{HAE-THO}, which was exponential in $k^3$, Dross and Havet~\cite{DRO-HAV} showed that $g(k)=O(k^2)$, before the current authors showed that $g(k)=O(k)$~\cite{BEN-MON}. That is, it is now known that  every $(n+O(k))$-vertex tournament contains a copy of every $n$-vertex oriented tree with $k$ leaves. Further evidence towards the conjecture of Havet and Thomass\'e is that $(n+k-1)$-vertices in the tournament are known to be sufficient if $n$ is much larger than $k$~\cite{BEN-MON} or if the tree is an arborescence~\cite{DRO-HAV} (that is, if the tree either has all paths branching outwards, or all paths branching inwards, from some designated root vertex). In addition, a result of Mycroft and Naia~\cite{MYC-NAI} shows that almost every $n$-vertex oriented tree is contained in every tournament of size $n$, indicating that Havet and Thomass\'e's conjecture is true for all but at most an asymptotically small proportion of $n$-vertex trees.

The largest gap between these results and the conjecture of Havet and Thomass\'e occurs whenever $k=\Omega(n)$. In this paper, we reduce the required number of vertices in the tournament  to $n+k+o(n)$, giving an asymptotic form of the Havet-Thomass\'e conjecture, as follows.
\begin{theorem}\label{thm:n+k+an}
Let $\alpha>0$. There exists $n_0\in\mathbb{N}$ such that for any $n\geqslant n_0$, if $G$ is a $((1+\alpha)n+k)$-vertex tournament and $T$ is an $n$-vertex oriented tree with $k$ leaves, then $G$ contains a copy of $T$.
\end{theorem}

We turn now to consider whether oriented trees with low maximum degree are guaranteed to appear in tournaments with fewer vertices than the extremal cases for Sumner's conjecture. This appears more difficult than the study of oriented trees based on the number of leaves. Indeed, it is not known whether there is a function $h(\Delta)$ such that any $(n+h(\Delta))$-vertex tournament contains a copy of every oriented tree with maximum degree at most $\Delta$, despite this question being raised by K\"uhn, Mycroft and Osthus~\cite{KUE-MYC-OST}. Mycroft and Naia~\cite{MYC-NAI} asked whether $h(\Delta)=2\Delta-4$ is sufficient as long as $n$ is much larger than $\Delta$, while recalling extremal examples due to Allen and Cooley that demonstrate this would be tight for each $\Delta$ (see also~\cite{KUE-MYC-OST}). When the maximum degree of the tree is sufficiently tightly bounded, it is known that few additional vertices are required in the tournament. Specifically, K\"uhn, Mycroft and Osthus~\cite{KUE-MYC-OST} proved that, if $\Delta$ is a fixed constant, then every $(1+o(1))n$-vertex tournament contains a copy of every $n$-vertex oriented tree with maximum degree $\Delta(T)\leqslant\Delta$ (however many leaves it has). Mycroft and Naia~\cite{MYC-NAI} later showed that the same conclusion holds even if the bound on $\Delta(T)$ is relaxed to one polylogarithmic in $n$. Here, we will relax the bound on $\Delta(T)$ much further still, showing that a degree bound linear in $n$ is sufficient, as follows.

\begin{theorem}\label{thm:n+an}
Let $\alpha>0$. There exists $c>0$ and $n_0\in\mathbb{N}$ such that for any $n\geqslant n_0$, if $G$ is a $(1+\alpha)n$-vertex tournament and $T$ is an $n$-vertex oriented tree with $\Delta(T)\leqslant cn$, then $G$ contains a copy of $T$.
\end{theorem}

We prove Theorem~\ref{thm:n+k+an} and Theorem~\ref{thm:n+an} under a common framework, using regularity methods and random homomorphisms to reduce these theorems to critical cases amenable to a more direct study. Though our methods apply to all the trees covered by these theorems, the results discussed above imply that the critical case for consideration are those trees with $\Omega(n)$ leaves. The most difficult cases arise when these leaves are all close to each other within the tree, connected via some smaller core tree (see Figure~\ref{fig:example-trees}). The challenge is then to be able to distribute these leaves around the tournament despite their location being quite tightly restricted by the location of the vertices in the core tree. For Theorem~\ref{thm:n+an}, the maximum degree condition will imply the core tree cannot be too small, and we exploit this to distribute the vertices of the core tree around the tournament. Here, the key novelty in our methods is the identification of the small core tree in the most challenging cases, and its embedding around the tournament.

For Theorem~\ref{thm:n+k+an}, we will be able to contract this small core tree in the most challenging cases to a single vertex without increasing the number of leaves. As the core tree is small, if we can find a copy of this contracted tree then we will be able to recover the original tree using suitable regularity techniques. The critical case will then be trees which have one very high degree vertex, whose removal results in components of at most constant size. Further simplification will allow us to assume that this high degree vertex has either in-degree or out-degree 0. This simplification focuses in on the hardest cases in our proof. To embed a tree $T$ with one high out-degree vertex $x$ with in-degree 0 into a tournament $G$, a natural approach is to place $x$ at the vertex of $G$ with highest out-degree, maximising the attachment possibilities for the components of $T-\{x\}$. Often, this is a good strategy (indeed, this approach will always succeed for the first tree depicted in Figure~\ref{fig:example-trees}), but when many vertices of $T$ are reached from $x$ by travelling along a path beginning with an forwards edge followed by a backwards edge (such as for the second tree depicted in Figure~\ref{fig:example-trees}), this may fail. Key to our proof is to use the failure in these cases to identify structural properties of the tournament, and thus a better location for the high out-degree vertex. This is the most significant novelty in our proof of Theorem~\ref{thm:n+k+an}, and enables the most difficult trees to be found in tournaments.

\begin{figure}
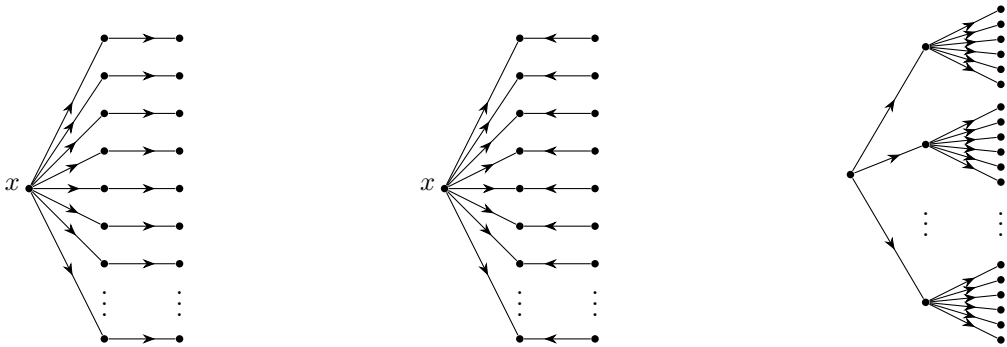

    \centering
    \begin{minipage}[b]{0.33\textwidth}
        \DiagramIntroExampleA
    \end{minipage}\hfill
    \begin{minipage}[b]{0.33\textwidth}
        \DiagramIntroExampleB
    \end{minipage}
    \begin{minipage}[b]{0.33\textwidth}
        \DiagramDoubleOutStars
    \end{minipage}
    \caption{Examples of oriented trees with many leaves close to one another. While the first two trees have $\Delta(T)\approx n/2$, the third tree may be realised with $\Delta(T)\leqslant cn$ for a small constant $c$, making it a case of particular interest for Theorem~\ref{thm:n+an}.}\label{fig:example-trees}
\end{figure}

In Section~\ref{sect:preliminaries}, we will state our notation before giving a more detailed outline of the proof, and of this paper.

\section{Preliminaries}\label{sect:preliminaries}
\subsection{Notation}\label{sect:notation}

For a directed graph (digraph) $G$, we use $V(G)$ to denote the vertex set of $G$ and $E(G)$ to denote the edge set of $G$. We write $|G|=|V(G)|$ for the order of $G$. Each element of $E(G)$ is an ordered pair $(u,v)$ (which we write as $uv$, or $u\rightarrow_G v$), where $u,v\in V(G)$. If $uv\in E(G)$, then we say that $v$ is an \emph{out-neighbour} of $u$, and that $u$ is an \emph{in-neighbour} of $v$. Given $v\in V(G)$, the \emph{out-neighbourhood} of $v$, written $N_G^+(v)$, is the set of out-neighbours of $v$ in $V(G)$, and the \emph{in-neighbourhood} of $v$, written $N_G^-(v)$ is the set of in-neighbours of $v$ in $V(G)$. Throughout, we use $+$ and $-$ interchangeably with `out' and `in' respectively. For $X,Y\subseteq V(G)$ and $\diamond\in\{+,-\}$, we write $N_G^\diamond(X)=(\cup_{v\in X}N_G^\diamond(v))\setminus X$ and $N_G^\diamond(X,Y)=N_G^\diamond(X)\cap Y$. For each $\diamond\in\{+,-\}$, the \emph{$\diamond$-degree} of $v$ in $G$ is $d_G^\diamond(v)=|N_G^\diamond(v)|$, and for $X,Y\subseteq V(G)$ we also write $d_G^\diamond(X,Y)=|N_G^\diamond(X,Y)|$. For a vertex $v$, we also define its neighbourhood to be $N_G(v)=N_G^+(v)\cup N_G^-(v)$ and its degree to be $d_G(v)=|N_G(v)|$, and similarly define $N_G(X)=N_G^+(X)\cup N_G^-(X)$ for a set $X\subseteq V(G)$. We denote by $G[X]$ the induced sub-digraph of $G$ with vertex set $X$ and let $G-X=G[V(G)\setminus X]$. Subscripts are omitted wherever they are clear from context, as are rounding signs wherever they are not crucial.

An \emph{oriented graph} is a digraph with at most one edge between any pair of vertices. A \emph{tournament} $G$ is an oriented graph whose underlying graph is a complete graph, i.e., for each $u,v\in V(G)$ with $u\neq v$, exactly one of $uv$ or $vu$ is in $E(G)$. An \emph{oriented tree} (respectively, \emph{oriented path}) is an oriented graph whose underlying graph is a tree (respectively, path). The \emph{maximum degree} of an oriented tree $T$ is the maximum degree of its underlying tree, and denoted $\Delta(T)$. A \emph{directed path} from $v_0$ to $v_\ell$ is an oriented path of the form $v_0\rightarrow v_1\rightarrow \ldots\rightarrow v_\ell$. The \emph{length} of a path $P$ is $|P|-1$. If $G,H$ are digraphs, a \emph{homomorphism} $\phi$ from $G$ to $H$ is a function $\phi:V(G)\to V(H)$ such that $\phi(u)\phi(v)\in E(H)$ whenever $uv\in E(G)$. We sometimes write $\phi:G\to H$ to denote a homomorphism from $G$ to $H$, and refer to $\phi$ as an \emph{embedding} of $G$ into $H$.

Having proved that, for example, a result holds for $\diamond=+$, we will occasionally deduce the same result for $\diamond=-$ by \emph{directional duality}. That is, reversing all the relevant orientations and applying the result with $\diamond=+$ implies, after reversing the edges again, the result with $\diamond=-$. Where the symbol $\pm$ appears in a formula, we mean the formula holds for both $+$ and $-$ in place of $\pm$. For a set $X$ and a function $f:X\to\mathbb{R}$, if $A\subseteq X$ we will often write $f(A)$ to mean $\sum_{x\in A}f(x)$ and $f(x_1,\ldots,x_k)$ to mean $f(\{x_1,\ldots,x_k\})$. For an event $E$ depending on the parameter $n$, we will say that $E$ holds \emph{with high probability} if $\mathbb{P}(E)\to1$ as $n\to\infty$. We also use standard hierarchy notation. That is, for $a,b\in(0,1]$, we write $a\ll b$ to mean that there is a non-decreasing function $f:(0,1]\to(0,1]$ such that the subsequent statement holds whenever $a\leqslant f(b)$.

\subsection{Proof outline}\label{sect:outline}
We will now sketch the proofs for both Theorem~\ref{thm:n+k+an} and Theorem~\ref{thm:n+an} together. In the introduction, we discussed how, for both results, we need to take particular care with trees which contain some small core subtree that restricts the distribution of the other vertices in the tree around the tournament. Therefore, we will identify a small core in any tree $T$, from which $T$ can be recovered by first appending a collection of constant-sized trees, then connecting components by constant-length paths, and then iteratively attaching a small number of additional leaves. This decomposition is independent of the directions of the edges of $T$, and so we state it for non-oriented trees. More precisely, given any tree $T$, we find $T_0\subseteq T_1\subseteq T_2\subseteq T_3\subseteq T_4=T$ (shown in Figure~\ref{fig:tree-decomposition}), such that
\begin{enumerate}[nosep,label=\roman{enumi})]
    \item\label{mini-tree-decomp-i} $T_0$ is small, 
    \item\label{mini-tree-decomp-ii} $T_1$ is formed by adding constant-sized trees, each attached with an edge to some vertex of $T_0$,
    \item\label{mini-tree-decomp-iii} $T_2$ is formed by adding (unattached) constant-sized trees to $T_1$,
    \item\label{mini-tree-decomp-iv} $T_3$ is formed by adding long but constant-sized paths connecting the components of $T_2$, and
    \item\label{mini-tree-decomp-v} $T_4$ is formed by attaching constant-sized trees to $T_3$, so that few vertices are added in total.
\end{enumerate}
Having found such a decomposition, we need a strategy for embedding these pieces. We note first that the vertices added in \ref{mini-tree-decomp-iii} and \ref{mini-tree-decomp-v} above pose little trouble given the spare vertices in our tournament. Indeed, within, say, any $\alpha n/2$ vertices in a tournament, any oriented tree with up to $\alpha n/6$ vertices can be found using known results (see Theorem~\ref{thm:any_linear_bound}). This allows the new constant-sized trees in \ref{mini-tree-decomp-iii} to be found greedily. For \ref{mini-tree-decomp-v}, we note that setting aside a small random subset of the spare vertices preserves some in- and out-neighbours for almost all the remaining vertices in the tournament (see Proposition~\ref{prop:random-subset}). Carrying out the embedding for \ref{mini-tree-decomp-i}--\ref{mini-tree-decomp-iv} within the tournament induced on these \emph{good} remaining vertices, will then allow us to extend the embedding greedily to cover the final vertices in \ref{mini-tree-decomp-v} (see Corollary~\ref{cor:tree-extension}).

Thus, our focus is on how to embed the vertices in $T_0$ so that this can be extended to an embedding of $T_1$, and how to embed the paths at \ref{mini-tree-decomp-iv}. In certain tournaments the paths at \ref{mini-tree-decomp-iv} can also be embedded straightforwardly by reserving a random set of vertices for this purpose. Where this is not possible, by removing a small set of vertices from the tournament, we will be able to partition the vertices into a sequence of linearly-sized sets, with all edges between the sets directed forwards along the sequence. This partition allows us to divide the tree naturally into pieces, which can then be found separately along the sequence of sets. We note that this is a streamlined version of a decomposition due to K\"uhn, Mycroft and Osthus~\cite{KUE-MYC-OST,KUE-MYC-OST-2}.

Let us assume then that the tournament is sufficiently well connected that the paths at \ref{mini-tree-decomp-iv} can be embedded within a reserved random set of vertices. We need then to embed $T_0$ so that the vertices of $T_1-V(T_0)$ can be distributed throughout the tournament. To do this we will use the regularity lemma for digraphs, so that we may assign vertices to clusters before embedding them. The challenge is to identify some good clusters for $T_0$, for which we can assign the vertices in $V(T_1)\setminus V(T_0)$ across the other regularity clusters. The whole of $T_1$ can then be embedded using relatively standard regularity techniques, in combination with the result that any oriented tree with $\ell$ vertices can be found in tournaments with only $O(\ell)$ vertices.

Embedding the core $T_0$ of the tree and extending it to cover $T_1$ is the only part where the proofs of Theorem~\ref{thm:n+k+an} and Theorem~\ref{thm:n+an} that differ. For Theorem~\ref{thm:n+k+an}, the core tree can always be embedded within a single regularity cluster, which will allow us to reduce the problem to embedding trees where the core is a single vertex, $x$ say (which may have very high degree), and further reduction will allow us to assume that all components of $T-x$ are attached to $x$ by out-edges. Similar to the discussion in the introduction, here it would be natural to try embedding $x$ to a cluster with as many out-edges as possible in a suitable `reduced digraph' (see Section~\ref{sect:regularity}). This may fail, but we try this anyway, essentially embedding as much of $T_1$ as possible. If the embedding fails, it will be due to certain structural properties of the tournament which will allow us to move the embedding of $x$, along with some of the embedded vertices, to complete the embedding. This part of the proof, with its division into a number of detailed subcases, is the most technical aspect of our proof, but solves the key problem and allows the proof of Theorem~\ref{thm:n+k+an}.

Fortunately, embedding $T_0$ and extending this to cover $T_1$ is less involved for Theorem~\ref{thm:n+an}. The maximum degree condition in this case ensures that $T_0$ necessarily contains at least a large constant number of vertices (for example, for the third tree shown in Figure~\ref{fig:example-trees}, we may identify $T_0$ with the star consisting of all non-leaf vertices). Thus, it is possible to distribute the vertices of $T_0$ across several regularity clusters if required for the even distribution of $V(T_1)\setminus V(T_0)$ throughout the tournament. For this, we identify a particular caterpillar-like structure which spans most of the clusters in the regularity graph (see Section~\ref{sect:caterpillar}).

Each aspect of the proof is discussed in more detail before it is carried out. In Section~\ref{sect:tree_decomposition}, we define precisely our tree decomposition and show that such a decomposition can always be found. In Section~\ref{sect:tree-embedding-results}, we state some established tree embedding results. We then recall and discuss the regularity lemma is Section~\ref{sect:regularity}, before covering some probabilistic results in Section~\ref{sect:probabilistic-results}. In Section~\ref{sect:T_1-unbounded}, we embed the core tree $T_0$ and extend the embedding to cover $T_1$ for Theorem~\ref{thm:n+k+an} (see Theorem~\ref{thm:n+k+an-partial}), deferring the most technical parts to Section~\ref{sect:technical} where we prove a key intermediate result, Theorem~\ref{thm:extending-distillation}. In Section~\ref{sect:T_1-bounded}, we embed the core tree $T_0$ and extend the embedding to cover $T_1$ for Theorem~\ref{thm:n+an} (see Theorem~\ref{thm:n+an-partial}). These embeddings of $T_0$ extended to $T_1$ allow us then to prove both Theorem~\ref{thm:n+k+an} and~\ref{thm:n+an} in Section~\ref{sect:final-proofs}. We then finish with the deferred proof of Theorem~\ref{thm:extending-distillation} in Section~\ref{sect:technical}.

\subsection{Tree decomposition}\label{sect:tree_decomposition}

We now give the tree decomposition discussed in the proof outline (see Figure~\ref{fig:tree-decomposition}).

\begin{lemma}\label{lm:tree_decomposition}
Let $1/n\ll1/m\ll\eta,1/q$ with $q\geqslant 2$. Then, any $n$-vertex tree $T$ contains forests $T_0\subseteq T_1\subseteq T_2\subseteq T_3\subseteq T_4=T$, such that $T_3$ is a tree, and the following properties hold.
    \stepcounter{propcounter}
    \begin{enumerate}[label =\emph{\textbf{\Alph{propcounter}\arabic{enumi}}}]
		\item\label{tree1} $|T_0|\leqslant\eta n$.
		\item\label{tree2} $T_1$ is formed from $T_0$ by the vertex-disjoint addition of trees $S_v$, $v\in V(T_0)$, so that, for each $v\in V(T_0)$, $S_v-v$ is a forest with each component tree having size at most $m$.
		\item\label{tree3} $T_2$ is the vertex-disjoint union of $T_1$ and a forest with each component tree having size at most $m$.
		\item\label{tree4} $T_3$ is formed from $T_2$ by connecting components by paths of length $q$.
		\item\label{tree5} $|V(T_4)\setminus V(T_2)|\leqslant \eta n$.
	\end{enumerate}
\end{lemma}

\begin{figure}
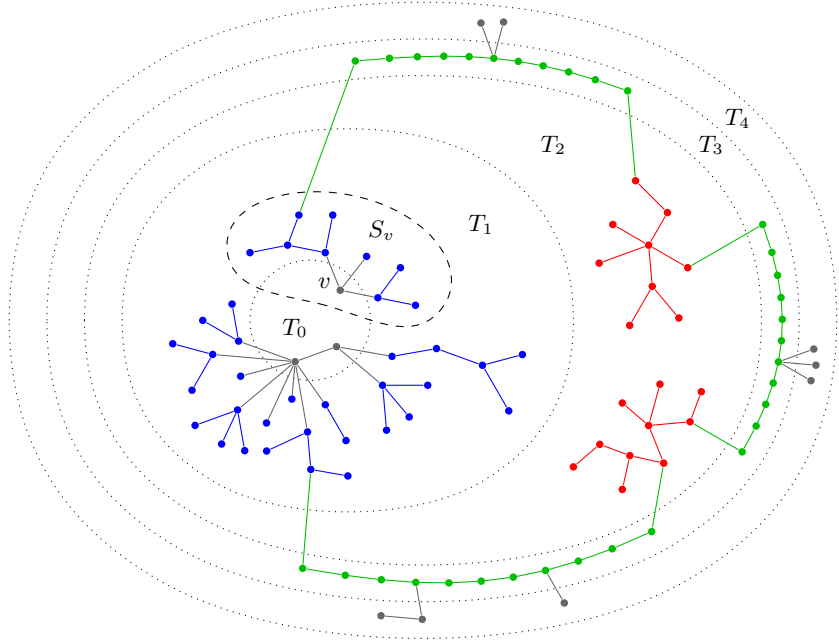

\DiagramTreeDecomposition
\vspace{-0.75cm}
\caption{A simplified example of the tree decomposition $T_0\subseteq T_1\subseteq T_2\subseteq T_3\subseteq T_4=T$ described by Lemma~\ref{lm:tree_decomposition}. In this illustration, the forest $T_0$ consists of a single edge together with an isolated vertex, and there are three paths, each with length $q=12$, connecting components of $T_2$ to form the tree $T_3$.}\label{fig:tree-decomposition}
\end{figure}

This is a modified version of a result of Kathapurkar and the second author (see~\cite{KAT-MON}); though its proof is very similar, we include it for completeness (first stating a structural result on which the proof relies).
\begin{lemma}[{\cite[Lemma~4.1]{MON-POK-SUD}}]\label{lm:removing-paths-from-trees}
Let $\ell,k\geqslant 2$. Suppose $T$ is a tree with at most $\ell$ leaves. Then, there is some $r$ and some vertex-disjoint bare paths $P_1,\ldots,P_r$ with length $k$ such that $|T-P_1-\ldots-P_r|\leqslant 6k\ell+2|T|/(k+1)$. 
\end{lemma}

\begin{proof}[Proof of Lemma~\ref{lm:tree_decomposition}]
Choose $\epsilon>0$ and $k\in\mathbb{N}$ such that $1/m\ll\epsilon\ll1/k\ll\eta,1/q$. Fix an arbitrary vertex $t\in V(T)$. We start by finding a subtree $T'$ of $T$ which includes $t$ and has few leaves, and is such that $T-V(T')$ is a forest of components each having size at most $m$. We do this by including in $T'$ every vertex which appears on the path in $T$ from $t$ to many other vertices. That is, for each $v\in V(T)$, let $w(v)$ be the number of vertices $u\in V(T)$ whose path from $t$ to $u$ includes $v$ (in particular, $v$ is such a vertex). Let $T'$ be the subgraph of $T$ induced on all the vertices $v\in V(T)$ with $w(v)\geqslant m+1$.

For each $v\in V(T')$, let $S_v$ be the tree containing $v$ in $T-(V(T')\setminus\{v\})$. Note that $S_v-v$ is a forest with each component tree having size at most $m$. Indeed, suppose $T''$ is a component of $S_v-v$, and let $v'$ be the neighbour of $v$ in $T''$. Since every path from a vertex $u\in V(T'')$ to $t$ in $T$ goes through $v'$ (and then $v$), we have that $m\geqslant w(v')\geqslant |T''|$ (and, in fact, the final inequality is an equality). Observe further that, for any leaf $v$ of $T'$, $|S_v-v|=w(v)-1\geqslant m$, and, therefore, $T'$ can have at most $n/m\leqslant \epsilon n$ leaves.

We say a subpath $P$ of $T'$ is a bare path if all of the internal vertices $v$ of $P$ have $d_{T'}(v)=2$, and we denote by $T'-P$ the graph formed from $T'$ by removing all the edges and internal vertices of $P$. Using Lemma~\ref{lm:removing-paths-from-trees}, find in $T'$ vertex-disjoint bare paths $P_1,\ldots,P_r$ with length $k$ such that
\begin{equation}
    |T'-P_1-\ldots-P_r|\leqslant6k\cdot\epsilon n+2n/(k+1)\leqslant\eta n/4
\end{equation}
Note that $r\leqslant n/k$. For each path $P_i$, if possible, find within $P_i$ a path $P_i'$ with length at least $k-2\eta^3k$, such that, letting $X_i$, $Y_i$ be the subpaths of $P_i'$ induced by the first and last $q-1$ vertices of $P_i'$, the following hold.
\begin{enumerate}[label=(\roman{enumi})]
    \item $\sum_{v\in V(X_i)}|S_v|,\sum_{v\in V(Y_i)}|S_v|\leqslant\eta k/4$.
    \item Letting $Q_i$ be the component of $T-X_i-Y_i$ containing $P_i'-X_i-Y_i$, we have $|Q_i|\leqslant m$.
\end{enumerate}

Say, with relabelling, these paths are $P_1',\ldots,P_{r'}'$. Let $T_0=T'-P_1'-\ldots-P_{r'}'$. We will show that $|T_0|\leqslant\eta n$. Consider first the number of paths $P_i$ which do not have length $q-2$ subpaths $X_i,Y_i$, each contained within $\eta^3 k$ of each end of $P_i$, and for which $\sum_{v\in V(X_i)}|S_v|,\sum_{v\in V(Y_i)}|S_v|\leqslant\eta k/4$. There are at most $n/(\lfloor\eta^3k/(q-1)\rfloor(\eta k/4))\leqslant\eta n/4k$ such paths. Of the remaining paths, at most $n/m$ may fail to produce a $P_i'$ due to having $|Q_i|>m$. Thus, we have $r'\geqslant r-\eta n/4k-n/m\geqslant r-\eta n/2k$.

Note that, for each $i\in[r']$, $|V(P_i)\setminus V(P_i')|\leqslant2\eta^3k$. Therefore
\[
|T_0|\leqslant|T'-P_1-\ldots-P_r|+k(r-r')+r'(2\eta^3 k)\leqslant\eta n/4 + k(\eta n/2k)+r(2\eta^3 k)\leqslant\eta n,
\]
and hence \itref{tree1} holds. Let $T_1=T[\cup_{v\in V(T_0)}V(S_v)]$. Recall that for each $v\in V(T')$, $S_v-v$ is a forest with each component tree having size at most $m$. Therefore, \itref{tree2} holds. Let $T_2=T_1\cup(\cup_{i\in[r']}Q_i)$, and note that \itref{tree3} holds. Note that
\[
|V(T)\setminus V(T_2)|=\sum_{i\in[r']}\sum_{v\in V(X_i)\cup V(Y_i)}|S_v|\leqslant 2r(\eta k/4)\leqslant\eta n,
\]
and hence \itref{tree5} holds. Let $T_3=T[V(T_2)\cup(\cup_{i\in[r']}(V(X_i)\cup V(Y_i)))]$ and note that \itref{tree4} holds. Finally, the only vertices missing from $T_3$ are those in $S_v-v$ for each $v\in \cup_{i\in[r']}(V(X_i)\cup V(Y_i))$, and hence $T_3$ is a tree. 
\end{proof}

\subsection{Tree embedding results}\label{sect:tree-embedding-results}

We will often embed small parts of a tree into a subset of a tournament with many spare vertices. To do this we could use any result embedding an $n$-vertex tree into a tournament with $O(n)$ vertices, but for convenience we will use the following result of El Sahili~\cite{El_S}.

\begin{theorem}[{\cite[Corollary~2]{El_S}}]\label{thm:any_linear_bound}
For each $n\geq 2$, every $(3n-3)$-vertex tournament contains a copy of every $n$-vertex oriented tree.
\end{theorem}

The following corollary shows how Theorem~\ref{thm:any_linear_bound} can be used to extend a partial copy of a tree to a full copy, provided each vertex in the partial copy has sufficient out- and in-degree to the remaining vertices in the tournament.

\begin{corollary}\label{cor:tree-extension}
Let $G$ be a tournament with disjoint subsets $U,V\subseteq V(G)$. Let $T$ be a tree, and suppose $T'\subseteq T$ is a subtree such that there is a copy $S'$ of $T'$ in $G[V]$. If $d_G^\pm(v,U)\geqslant3|V(T)\setminus V(T')|$ for every $v\in V$, then $S'$ can be extended to a copy of $T$ in $G$, with $T-V(T')$ copied to $U$.
\end{corollary}
\begin{proof}
Label the components of $T-V(T')$ as $T_1,\ldots,T_r$, and take the largest $s\leqslant r$ such $S'$ can be extended to a copy $S$ of $T[V(T')\cup(\cup_{i\in[s]}V(T_i))]$. Suppose that $s<r$. Then, if $\diamond\in\{+,-\}$ is such that $T_{s+1}$ is attached to $T'$ by a $\diamond$-neighbour, and $v\in V(S')$ is the copy of the attachment point, then
\begin{equation*}
    d_G^\diamond(v,U\setminus V(S))\geqslant3|V(T)\setminus V(T')|-|T_1|-\ldots-|T_s|\geqslant 3|T_{s+1}|,
\end{equation*}
and so, by Theorem~\ref{thm:any_linear_bound}, $N_G^\diamond(v,U\setminus V(S))$ contains a copy of $T_{s+1}$, contradicting the maximality of $s$. Thus, $S$ is a copy of $T$ in $G$.
\end{proof}

We will also need to embed a tree into a subset of a tournament with a number of spare vertices depending on the number of leaves of the tree. Any such bound would suffice, but we will use the following result.

\begin{theorem}[{\cite[Theorem~1.1]{BEN-MON}}]\label{thm:n+Ck}
There is some $C>0$ such that every $(n+Ck)$-vertex tournament contains a copy of every $n$-vertex oriented tree with $k$ leaves.
\end{theorem}

\subsection{Regularity}\label{sect:regularity}
Our embeddings use the regularity lemma for digraphs, by now a well-established tool in the study of tournaments (see, for example, \cite{KUE-MYC-OST,KUE-MYC-OST-2,MYC-NAI}). As with the regularity lemma for graphs, this partitions most of the vertices of a tournament into clusters so that edges behave pseudorandomly between most pairs of clusters. We will now recall the notation needed to state the regularity lemma for digraphs.

Let $G$ be a digraph. For disjoint subsets $A,B\subseteq V(G)$, define the \emph{directed density} from $A$ to $B$ to be
\begin{equation*}
    d(A,B)=\frac{|E(A,B)|}{|A||B|},
\end{equation*}
where $E(A,B)$ denotes the set of edges of $G$ directed from $A$ towards $B$. Note that, if $G$ is tournament, then $d(B,A)=1-d(A,B)$. We say that $(A,B)$ forms an \emph{$\epsilon$-regular pair} if, for every $X\subseteq A$ such that $|X|\geqslant\epsilon|A|$ and every $Y\subseteq B$ such that $|Y|\geqslant\epsilon|B|$, we have $|{d}(X,Y)-d(A,B)|\leqslant\epsilon$. Note that, for tournaments, $|d(X,Y)-d(A,B)|\leqslant\epsilon$ if and only if $|d(Y,X)-d(B,A)|\leqslant\epsilon$. We say that $(A,B)$ forms an $\epsilon$-regular pair of density at least $\mu$ if, in addition to forming an $\epsilon$-regular pair, we also have $d(A,B)\geqslant \mu$.

We will use the following directed version of Szemer\'edi's regularity lemma proved by Alon and Shapira~\cite{ALO_SHA}.

\begin{theorem}[Regularity lemma for digraphs]\label{thm:digraph_regularity}
Let $1/r_2\ll 1/r_1\ll\epsilon$. Every digraph on a vertex set $V$ of order at least $r_1$ partitions as $V=V_0\cup V_1\cup\ldots\cup V_r$, with $r_1\leqslant r\leqslant r_2$, satisfying the following.
\stepcounter{propcounter}
\begin{enumerate}[label =\emph{\textbf{\Alph{propcounter}\arabic{enumi}}}]
    \item\label{R1} $|V_0|\leqslant\epsilon|V|$.
    \item\label{R2} $|V_1|=\ldots=|V_r|$.
    \item\label{R3} All but at most $\epsilon r^2$ pairs $(V_i,V_j)$ with $1\leqslant i<j\leqslant r$ are $\epsilon$-regular.
\end{enumerate}
\end{theorem}

We now state the definition of an \emph{$\epsilon$-regular partition}. For convenience, we use a slightly different definition of an $\epsilon$-regular partition of a tournament than is directly produced by Theorem~\ref{thm:digraph_regularity}, but which is gained through the removal of few clusters (see  Corollary~\ref{cor:regular_subtournament}).

\begin{defn}\label{defn:regular-partition}
An \emph{$\epsilon$-regular partition} of a tournament $G$ is a partition $V(G)=V_1\cup\ldots\cup V_r$ with $|V_1|=\ldots=|V_r|$ such that, for each fixed $i\in[r]$, $(V_i,V_j)$ forms an $\epsilon$-regular pair for all but at most $\sqrt{\epsilon}r$ many $j\in[r]$.
\end{defn}

\begin{corollary}\label{cor:regular_subtournament}
Let $\alpha>\beta>0$ and $1/n\ll1/r_2\ll1/r_1\ll\epsilon\ll\beta$. Let $G$ be a $(1+\alpha)n$-vertex tournament. Then, there is a subtournament $G'\subseteq G$ with $|G'|\geqslant(1+\alpha-\beta)n$, and an $\epsilon$-regular partition $V(G')=V_1\cup\ldots\cup V_r$ with $r_1\leqslant r\leqslant r_2$.
\end{corollary}

\begin{proof}
Given a tournament $G$, using Theorem~\ref{thm:digraph_regularity}, take a partition $V(G)=V_0\cup V_1\cup\ldots\cup V_{\bar{r}}$, with $2r_1\leqslant \bar{r}\leqslant r_2$, satisfying \itref{R1}-\itref{R3}. By reordering, we may suppose there is some $r$ with $0\leqslant r\leqslant \bar{r}$ such that, for each fixed $i\in[\bar{r}]$, $(V_i,V_j)$ forms an $\epsilon$-regular pair with at least $(1-\sqrt{\epsilon}/2)\bar{r}$ many $j\in[\bar{r}]$ if and only if $i\in[r]$. By \itref{R3}, we find $(\bar{r}-r)\sqrt{\epsilon}\bar{r}/2\leqslant \epsilon \bar{r}^2$, and hence $r\geqslant(1-2\sqrt{\epsilon})\bar{r}\geqslant r_1$. Let $G'=G[V_1\cup\ldots\cup V_r]$. The desired properties for $G'$ then follow by noting that $|V(G)\setminus V(G')|\leqslant |V_0|+\frac{\bar{r}-r}{\bar{r}}|G|\leqslant(\epsilon+2\sqrt{\epsilon})|G|\leqslant\beta n$, and that $\sqrt{\epsilon}\bar{r}/2\leqslant \sqrt{\epsilon} r$.
\end{proof}

We will use the following simple proposition on vertex degrees in $\eps$-regular partitions.

\begin{proposition}\label{prop:outneighbourhood_to_many_clusters}
Let $\epsilon,\mu>0$ and $r,m\in\mathbb{N}$. Suppose $G$ is a tournament with disjoint subsets $V_0,V_1,\ldots,V_r\subseteq V(G)$ of size $|V_0|=|V_1|=\ldots=|V_r|=m$, such that $(V_0,V_i)$ is an $\epsilon$-regular pair of density at least $\mu$ for each $1\leqslant i\leqslant r$. Fix a subset $U\subseteq\cup_{i\in[r]} V_i$. Then, all but at most $\epsilon m$ vertices of $V_0$ have at least $(\mu-\epsilon)(|U|-\epsilon rm)$ out-neighbours in $U$.
\end{proposition}
\begin{proof}
Let $W$ be the set of vertices of $V_0$ which have fewer than $(\mu-\epsilon)(|U|-\epsilon rm)$ out-neighbours in $U$, and suppose that $|W|\geqslant\epsilon m$. Then, for each $i\in[r]$, because $(V_0,V_i)$ is an $\epsilon$-regular pair of density at least $\mu$, there are at least $|W|(\mu-\epsilon)(|U\cap V_i|-\epsilon m)$ edges directed from $W$ to $U\cap V_i$, noting that this is trivial if $|U\cap V_i|\leqslant\epsilon m$. Therefore, there are at least $|W|(\mu-\epsilon)(|U|-\epsilon rm)$ edges directed from $W$ to $U$. However, from the definition of $W$, the number of edges from $W$ to $U$ is less than $|W|(\mu-\epsilon)(|U|-\epsilon rm)$, a contradiction.
\end{proof}

Our proofs will often allocate the vertices of a tree to the clusters of a regularity partition, before applying variations of standard regularity methods to embed these vertices so that they are (mostly) embedded to their assigned cluster. For this we will use, in part, the following simple proposition, which embeds a tree from an assignment in this way, provided that the tree is small and also that the vertices of the tree are not assigned to too many different clusters.

\begin{proposition}\label{prop:branch-embed}
Let $1/m\ll\epsilon\ll\beta,\mu,1/\ell$. Suppose $G$ is a tournament with subsets $V_1,\ldots,V_\ell\subseteq V(G)$ of size $|V_1|=\ldots=|V_\ell|=m$, and, for $j\in[\ell]$, let $U_j\subseteq V_j$ have size $|U_j|\geq\beta m$. Let $T$ be an oriented tree with $|T|\leqslant\epsilon m$, and suppose $\varphi:V(T)\to[\ell]$ is such that if $uv\in E(T)$ and $\varphi(u)\neq\varphi(v)$, then $(V_{\varphi(u)},V_{\varphi(v)})$ is an $\epsilon$-regular pair of density at least $\mu$. Then, there is an embedding $\psi:T\to G$ with $\psi(v)\in U_{\varphi(v)}$ for each $v\in V(T)$.
\end{proposition}
\begin{proof}
Let $V(T)=Y_1\cup\ldots\cup Y_r$ be a partition such that
\stepcounter{propcounter}
\begin{enumerate}[label =\textbf{\Alph{propcounter}\arabic{enumi}}]
\item For each $i\in [r]$, $T[Y_i]$ is a connected component of $T[\phi^{-1}(j)]$ for some $j\in[\ell]$.
\item For each $i\in [r]$, $T[Y_1\cup \ldots\cup Y_i]$ is a tree.
\end{enumerate}
Let $s\in\{0\}\cup[r]$ be maximal such that, if $T_s=T[Y_1\cup\ldots\cup Y_s]$, then there is an embedding $\psi:T_s\to G$ with $\psi(v)\in V_{\varphi(v)}$ for every $v\in V(T_s)$, and, for every $v\in V(T_s)$ and $j\in[\ell]$ for which $(V_{\varphi(v)},V_j)$ is an $\epsilon$-regular pair, we have $d_G^+(\psi(v),U_j)\geqslant (d(V_{\varphi(v)},V_j)-\epsilon)\beta m$ and $d_G^-(\psi(v),U_j)\geqslant (d(V_j,V_{\varphi(v)})-\epsilon)\beta m$. Suppose, for contradiction, that $s<r$. If $s=0$, then let $y\in Y_1$ be arbitrary and set $Z_{1}=U_{\varphi(y)}$. If instead we have $s>0$, then let $x\in V(T_s)$, $y\in Y_{s+1}$ and $\diamond\in\{+,-\}$ be such that $y\in N_T^\diamond(x)$, and set $Z_{s+1}=N_G^\diamond(\psi(x),U_{\varphi(y)})$. In either case, we find that $|Z_{s+1}|\geqslant\beta\mu m/2$ and $Z_{s+1}\subseteq U_{\varphi(y)}$. For each $j\in[\ell]$ such that $(V_{\varphi(y)},V_j)$ is an $\epsilon$-regular pair, all but at most $\epsilon m$ vertices $z$ of $Z_{s+1}$ satisfy $d_G^+(z,U_j)\geqslant (d(V_{\varphi(y)},V_j)-\epsilon)\beta m$ and all but at most $\epsilon m$ vertices $z$ of $Z_{s+1}$ satisfy $d_G^-(z,U_j)\geqslant (d(V_j,V_{\varphi(y)})-\epsilon)\beta m$. Therefore, as $\epsilon\ll\beta,\mu,1/\ell$, there is a subset $Z_{s+1}'\subseteq Z_{s+1}\setminus\psi(V(T_s))$ with $|Z_{s+1}'|\geqslant\beta\mu m/4$, such that, for every $z\in Z_{s+1}'$ and $j\in[\ell]$ for which $(V_{\varphi(y)},V_j)$ is an $\epsilon$-regular pair, we have $d_G^+(z,U_j)\geqslant (d(V_{\varphi(y)},V_j)-\epsilon)\beta m$ and $d_G^-(z,U_j)\geqslant (d(V_j,V_{\varphi(y)})-\epsilon)\beta m$. But then, by Theorem~\ref{thm:any_linear_bound}, there is a copy of $T[Y_{s+1}]$ in $G[Z_{s+1}']$, and so we can extend $\psi$ to cover $Y_{s+1}$, a contradiction to the maximality of $s$.
\end{proof}

As is common, given an \emph{$\epsilon$-regular partition} $V_1\cup\ldots\cup V_r$ of a tournament $G$, we will consider the reduced digraph $R$ for the partition which has $V(R)=[r]$, and $ij\in E(R)$ exactly when $(V_i,V_j)$ is an $\epsilon$-regular pair with density comfortably larger than $\epsilon$. We will sometimes delete edges arbitrarily from $R$ so that there is at most 1 edge between any pair of vertices. As a small proportion of pairs of clusters in an $\epsilon$-regular partition may not form an $\epsilon$-regular pair, this will not necessarily result in a tournament. For this, we define an \emph{$\epsilon$-almost tournament}, as follows.

\begin{defn}
An \emph{$\epsilon$-almost tournament} $R$ is a digraph with 
at most one edge between any pair of vertices, and in which, for each $v\in V(R)$, there are at most $\epsilon |R|$ vertices $u\in V(R)$ with $vu\notin E(R)$ and $uv\notin E(R)$.
\end{defn}

We will use the following simple property of $\epsilon$-almost tournaments, which shows they each have some vertex with a good number of both in- and out-neighbours.

\begin{proposition}\label{prop:max_digraph_split}
Let $R$ be an $\epsilon$-almost tournament on $r$ vertices. Then, there exists a vertex $v\in V(R)$ such that $d_R^+(v),d_R^-(v)\geqslant\frac{r-1}{4}-\epsilon r$.
\end{proposition}
\begin{proof}
Let $H$ be any tournament with $V(H)=V(R)$ such that $R\subseteq H$, and let $m=\frac{r-1}{4}$. Any set of $2m+1$ vertices in $H$ contains a vertex with out-degree at least $m$ and a vertex with in-degree at least $m$. Therefore, all but at most $2m$ vertices of $H$ have out-degree at least $m$, and all but at most $2m$ vertices of $H$ have in-degree at least $m$. Therefore, as $n>4m$, there is some $v\in V(H)$ with $d_H^+(v),d_H^-(v)\geqslant\frac{r-1}{4}$. Then, $v\in V(R)$ satisfies $d_R^+(v),d_R^-(v)\geqslant\frac{r-1}{4}-\epsilon r$.
\end{proof}

\subsection{Probabilistic results}\label{sect:probabilistic-results}
Parts of our embeddings will be random, or use some reserved random set. To analyse these parts, we will use the following probabilistic bounds (see, for example, \cite{intro-random-graphs}). 
The first is a Chernoff bound, and the second is Hoeffding's inequality.

\begin{lemma}\label{lm:chernoff} If $X$ is a binomial variable with standard parameters~$n$ and $p$, denoted $X=\mathrm{Bin}(n,p)$, and $\epsilon$ satisfies $0<\epsilon\leq 3/2$, then
\[
\mathbb{P}(|X-\mathbb{E}X|\geq \epsilon \mathbb{E} X)\leq 2\exp\left(-\epsilon^2\mathbb{E} X/3\right).\hfill\qedhere
\]
\end{lemma}

\begin{theorem}\label{thm:hoeffding-inequality}
Let $X_1,\ldots,X_n$ be independent random variables with $X_i$ bounded by the interval $[a_i,b_i]$ for $i\in[n]$. Let $X=\sum_{i\in[n]}X_i$. Then, for any $t>0$, we have
\begin{equation*}
    \mathbb{P}(|X-\mathbb{E}X|\geqslant t)\leqslant2\exp{\left(-\frac{2t^2}{\sum_{i\in[n]}(b_i-a_i)^2}\right)}.
\end{equation*}
\end{theorem}

It will often be convenient for most of the vertices to have large in- and out-degree into a reserved random set, for which we use the following result.

\begin{proposition}\label{prop:random-subset}
Fix $p>0$. Let $G$ be a tournament with $n\leqslant|G|\leqslant3n$. Let $U\subseteq V(G)$ be a random subset, with elements from $V(G)$ chosen independently at random with probability $p$. Let $V'$ be the set of vertices $v\in V(G)\setminus U$ for which $d^\pm(v,U)\geqslant p^2n$. Then, with high probability, $pn/2\leqslant|U|\leqslant4pn$, and $|V(G)\setminus V'|\leqslant12pn$.
\end{proposition}
\begin{proof}
By Lemma~\ref{lm:chernoff} and the fact that $pn\leqslant\E|U|\leqslant3pn$, we have $pn/2\leqslant|U|\leqslant4pn$ with high probability. If $v\in V(G)$ is such that $d_G^\pm(v)\geqslant 2pn$, then, by setting $\epsilon=1/2$ in Lemma~\ref{lm:chernoff}, the probability that $d^\pm(v,U)\geqslant p^2n$ fails for $v$ is at most $4\exp{(-p^2n/6)}$. Any set of $4pn+1$ vertices in $G$ contains a vertex with out-degree at least $2pn$ and a vertex with in-degree at least $2pn$. So at most $4pn$ vertices $v$ of $G$ have $d_G^+(v)<2pn$ and at most $4pn$ vertices of $G$ have $d_G^-(v)<2pn$. Therefore, the probability that $|V(G)\setminus V'|\leqslant|U|+8pn$ fails is at most $12n\exp{(-p^2n/6)}$. So $U$ satisfies both $pn/2\leqslant|U|\leqslant4pn$ and $|V(G)\setminus V'|\leqslant12pn$ with high probability.
\end{proof}

\section{Theorem~\ref{thm:n+k+an}: embedding the core and attached small  trees}\label{sect:T_1-unbounded}
In this section, following the proof outline in Section~\ref{sect:outline}, we embed $T_0$ and $T_1$ for Theorem~\ref{thm:n+k+an}. In the embeddings we may assume that $T_0$ is connected (i.e., that it is a tree not just a forest). Indeed, if $T_0$ is not a tree, then we may add arbitrary edges between its components to make it so. Note that if arbitrary edges were to be added in this way, then $T_1$ would be a tree due to \itref{tree2}. Thus, our embedding of $T_0$ and $T_1$ will follow by identifying $T_1$ with $T$ in the following theorem.

\begin{theorem}\label{thm:n+k+an-partial}
Let $1/n\ll\eta\ll\bar{\alpha}<1$. Suppose $T$ is an $n$-vertex $k$-leaf oriented tree with a subtree $T_0\subseteq T$, such that $|T_0|\leqslant\eta n$ and every component of $T-V(T_0)$ has size at most $\eta n$. Then, any $((1+\bar{\alpha})n+k)$-vertex tournament contains a copy of $T$.
\end{theorem}

We require, for Theorem~\ref{thm:n+k+an-partial}, the components of $T-V(T_0)$ to be bounded above by $\eta n$, for some appropriately small $\eta$. This is not required for its application, where these components will have constant size (see \emph{\ref{tree2}} earlier), but this small linear bound follows  at no additional cost. As discussed in Section~\ref{sect:outline}, to embed the core $T_0$ and extend this embedding to $T$, we will first allocate the vertices of the tree to regularity clusters. This allocation requires care beyond that in previous embeddings of trees in tournaments (see~\cite{KUE-MYC-OST,KUE-MYC-OST-2,MYC-NAI}) as the large degree of some vertices in the tree require edges not just to have sufficient density for regularity embedding techniques to be effective (i.e., $\eps\ll\mu$ in Proposition~\ref{prop:branch-embed}), but sufficient density for potentially linearly many neighbours of a vertex to be embedded within the same regularity cluster. For this, we find it convenient to consider an $\epsilon$-regular partition of clusters $V_1\cup\ldots\cup V_r$ as a weighted complete looped digraph $D$ with vertex set $[r]$ and edge weights $d(e)\in[0,1]$, $e\in E(D)$, indicating the edge density between $\epsilon$-regular pairs of clusters. We call the sets of edge weights we typically encounter \emph{$\epsilon$-complete}, as follows.

\begin{defn}
Given a complete looped digraph $D$ on vertex set $[r]$, we say edge weights $d(e)\in [0,1]$, $e\in E(D)$, are \emph{$\epsilon$-complete} if the following holds.
\stepcounter{propcounter}
\begin{enumerate}[label = {\bfseries\emph{\Alph{propcounter}}}]
\item\label{prop:likeregularity} For each $j\in [r]$, $d(j,j)=1$ and, for all but at most $\epsilon r$ values of $i\in [r]\setminus \{j\}$, $d({i,j})+d({j,i})=1$.
\end{enumerate}
\end{defn}

From this perspective, an allocation of the vertices of an oriented tree $T$ to the regularity clusters is a map from $V(T)$ to the vertices of a complete looped digraph $D$ with edge weights representing the density of edges between regularity clusters. Considering the role of $\mu$ in Proposition~\ref{prop:branch-embed}, it is important that whenever endpoints of an edge of $T$ are allocated to different clusters, those clusters form an $\epsilon$-regular pair of density at least $\mu$. When $D$ is introduced in the proof later, we will adjust the weights on some edges to ensure that $d(i,j)=0$ whenever $d(V_i,V_j)<\mu$. Thus, a valid allocation will be a homomorphism from $T$ to $D$, defined as follows.

\begin{defn}
Given a digraph $H$ and a complete looped digraph $D$ with associated edge-weights $d(e)\in[0,1]$, $e\in E(D)$, we say that a function $\phi:V(H)\to V(D)$ is a \emph{homomorphism from $H$ to $D$} if $d(\phi(v),\phi(w))>0$ whenever $vw\in E(H)$.
\end{defn}

We need to find such a homomorphism from $T$ to $D$ satisfying additional properties, such as a limit to how many vertices are assigned to each cluster. The allocation we find for Theorem~\ref{thm:n+k+an-partial} will always assign the vertices of $T_0$ to a single cluster, whose index we call $j_t$, and then distribute the vertices of the components of $T-V(T_0)$ across the other regularity clusters. Rather than considering the components of $T-V(T_0)$ directly, we will work with a small vertex-weighted digraph $H$, which represents an average of the components of $T-T(V_0)$. We will find a probability distribution $\mathcal{D}$ on the set of homomorphisms from $H$ to $D$ so that when we assign vertices of each component of $T-V(T_0)$ according to an independent sampling of $\mathcal{D}$, then, with high probability, the resulting homomorphism from $T$ to $D$ has the required properties. Working in this randomised setting allows us to obtain this homomorphism concisely and without the need to consider the many distinct oriented trees which may appear as components of $T-V(T_0)$.

The existence of the probability distribution $\mathcal{D}$ is asserted by Theorem~\ref{thm:extending-distillation} below, which is used in this section as a starting point for Theorem~\ref{thm:n+k+an-partial}. The proof of Theorem~\ref{thm:extending-distillation} itself is far more involved, and so we defer this to Section~\ref{sect:technical}. Before stating Theorem~\ref{thm:extending-distillation}, and also explaining the statement in more detail, we first define the digraph $H$ used to represent the components of $T-V(T_0)$ (see also Figure~\ref{fig:H}). To simplify the notation relating to $H$, we make use of the following definition.

\begin{defn}\label{def:fully-looped-forest}
Say that a digraph $F$ is a \emph{fully-looped oriented forest} if $F$ has a looped edge on every vertex, and the deletion of all looped edges from $F$ leaves an oriented forest.
\end{defn}


\newcommand{\HDefinitionIntro}{Let $H$ be the fully-looped oriented forest with vertex and edge sets given by}
\newcommand{\HDefinitionEquation}{
\begin{aligned}
    V(H)&=\left\{x^+,y^+,z^+,u^+,w^+,\bar{x}^+,\bar{z}^+,\bar{u}^+,\bar{w}^+,x^-,y^-,z^-,u^-,w^-,\bar{x}^-,\bar{z}^-,\bar{u}^-,\bar{w}^-\right\},\\
    E(H)&=\left\{
        \begin{array}{c}
            x^+y^+,z^+x^+,z^+u^+,w^+z^+,\bar{z}^+\bar{x}^+,\bar{z}^+\bar{u}^+,\bar{w}^+\bar{z}^+,\\
            y^-x^-,x^-z^-,u^-z^-,z^-w^-,\bar{x}^-\bar{z}^-,\bar{u}^-\bar{z}^-,\bar{z}^-\bar{w}^-
        \end{array}
    \right\}\cup\{vv:v\in V(H)\}.
\end{aligned}}
\newcommand{\HDefinitionOutro}{For each $\diamond\in\{+,-\}$, let $X^\diamond=\{x^\diamond,\bar{x}^\diamond\}$. Let $X=X^+\cup X^-$.}

\newcommand{\HDefinitionInitial}{
\HDefinitionIntro
\begin{equation}\label{eq:H-definition}
\HDefinitionEquation
\end{equation}
\HDefinitionOutro}

\newcommand{\HDefinitionSubsequent}{
\HDefinitionIntro
\begin{equation}\tag{\ref{eq:H-definition}}
\HDefinitionEquation
\end{equation}
\HDefinitionOutro}


\HDefinitionInitial\,

\begin{figure}
\DiagramH
\vspace{-0.75cm}
\caption{The fully-looped oriented forest $H$ (with looped edges omitted).}\label{fig:H}
\end{figure}

We are now ready to state Theorem~\ref{thm:extending-distillation}. Recall that, for the function $\beta:V(H)\to[0,1]$, if $A\subseteq V(H)$ we will often write $\beta(A)$ to mean $\sum_{v\in A}\beta(v)$ and $\beta(v_1,\ldots,v_k)$ to mean $\beta(\{v_1,\ldots,v_k\})$.

\stepcounter{propcounter}
\newcounter{technicalcounter}
\setcounter{technicalcounter}{\value{propcounter}}
\begin{restatable}{theorem}{technical}\label{thm:extending-distillation}
Let $1/r\ll\epsilon\ll\mu\ll\alpha<1$. Let $H$ be the fully-looped oriented forest with vertex and edge sets given by~\eqref{eq:H-definition}. For each $\diamond\in\{+,-\}$, set $X^\diamond=\{x^\diamond,\bar{x}^\diamond\}$, and set $X=X^+\cup X^-$. Let $\beta:V(H)\to[0,1]$ be a function satisfying $\sum_{v\in V(H)}\beta(v)=1$ with $\beta(y^+)\geqslant\beta(x^+)$ and $\beta(y^-)\geqslant\beta(x^-)$, and, for every $v\in V(H)$, $\beta(v)\geqslant\mu$. Let $D$ be a complete looped digraph on vertex set $[r]$ with $\epsilon$-complete edge weights $d(e)$ for $e\in E(D)$. Let
\begin{equation}\label{eq:gamma-def}
    \gamma=\max{\{\beta(x^+,\bar{x}^+),\beta(z^+,\bar{z}^+)\}}+\max{\{\beta(x^-,\bar{x}^-),\beta(z^-,\bar{z}^-)\}}.
\end{equation}
Then, there is a fixed $j_t\in [r]$ and a probability distribution $\mathcal{D}$ on the set of functions from $V(H)$ to $V(D)$, such that, if $\phi$ is sampled according to $\mathcal{D}$, then the following properties hold.
\begin{enumerate}[label = {\bfseries \emph{\Alph{technicalcounter}\arabic{enumi}}}]
    \item\label{ext-good-plan-1} With probability 1, $\phi$ is a homomorphism from $H$ to $D$, and $j_t\notin\phi(\{x^+,\bar{x}^+,x^-,\bar{x}^-\})$.
    \item \label{ext-good-plan-2} For each $j\in[r]$, $\E(\beta(\phi^{-1}(j)))\leq \frac{1+\gamma+\alpha}{r}$.
    \item\label{ext-good-plan-3} For each $j\in[r]$, either
    \begin{enumerate}[label = {\bfseries \emph{\Alph{technicalcounter}\arabic{enumi}.\arabic{enumii}}}]
        \item\label{ext-good-plan-3-1} $\E(\beta(\phi^{-1}(j)\cap X^+))\leq d(j_t,j)\cdot \frac{1+\gamma+\alpha}{r}$ and $\E(\beta(\phi^{-1}(j)\cap X))\leq d(j,j_t)\cdot\frac{1+\gamma+\alpha}{r}$, or
        \item\label{ext-good-plan-3-2} $\E(\beta(\phi^{-1}(j)\cap X^-))\leq d(j,j_t)\cdot \frac{1+\gamma+\alpha}{r}$ and $\E(\beta(\phi^{-1}(j)\cap X))\leq d(j_t,j)\cdot\frac{1+\gamma+\alpha}{r}$.
    \end{enumerate}
    \item\label{ext-good-plan-4} With probability 1, we have $|\phi(e)|=2$ for every non-looped edge $e$ of $H$.
\end{enumerate}
\end{restatable}

The technical nature of Theorem~\ref{thm:extending-distillation} is a result of the complications involved in finding an appropriate allocation of the vertices of $T-V(T_0)$ to the regularity clusters represented by $D$. Having chosen a single cluster (indexed by $j_t$) for the embedding of $T_0$, the restriction on where a vertex $u$ in $T-V(T_0)$ can be embedded depends on the path from $T_0$ to $u$ in $T$, and its edge directions. However, it will turn out that all we need to consider is what proportion of the vertices in $T-V(T_0)$ have paths from $T_0$ beginning with any given oriented path of length at most 3. Accordingly, in the application of Theorem~\ref{thm:extending-distillation}, each vertex $u$ of $T-V(T_0)$ will be associated to a vertex of $H$ depending on the orientation of the first few edges on the path from $T_0$ to $u$. Then, each vertex $v$ of $H$ will be given weight $\beta(v)$ roughly equal to the proportion of vertices of $T-V(T_0)$ associated to $v$. It is in this sense that the digraph $H$ together with the weight function $\beta$ represents an average of the components of $T-V(T_0)$. Vertices in $X$ represent the vertices connected by an edge to $T_0$ in $T$, which may therefore be neighbours of any very high degree vertices in $T_0$, and so we need to pay particular attention to how often they are assigned to each regularity cluster.

Identifying $H$ and $\beta$ in this way helps contextualise the statement of Theorem~\ref{thm:extending-distillation}. Suppose $j_t$ is fixed as the index corresponding to the image of $T_0$, and $\phi$ is sampled according to $\mathcal{D}$. By \itref{ext-good-plan-1}, we will have with probability 1 that $\phi$ is a homomorphism, so that the regularity properties can be used to embed any edges assigned between two regularity clusters, and, for convenience, no vertex in $X$ is assigned to $j_t$. \itref{ext-good-plan-2} ensures that (on average) not too much weight is allocated to a single cluster. \itref{ext-good-plan-3} ensures that (on average) the weight of vertices in $X^+$ (i.e., those which need to be attached by an out-edge to $T_0$, which is allocated to $j_t$), or $X^-$, allocated to each cluster is not too much, where the limit is dictated by the appropriate density from that cluster to $j_t$, or from $j_t$ to that cluster. Once $T_0$ has been embedded into a suitable subset of $V_{j_t}$, this provision ensures that the embedding may be easily extended to greedily cover $N_T(V(T_0))$, which is otherwise a key difficulty. Finally, \itref{ext-good-plan-4} is present to later ensure that vertices of $N_T(V(T_0))$ are allocated to a different cluster to their neighbours, which will assist with the embedding process. 

We use the parameter $\gamma$ (see \eqref{eq:gamma-def} in Theorem~\ref{thm:extending-distillation}) to control the total size of components we can embed using $\phi$ relative to the size of the tournament from which $D$ is derived. As we use this for Theorem~\ref{thm:n+k+an} it should be related to the number of leaves. In the application of Theorem~\ref{thm:extending-distillation}, the weight on the vertices with base label `$x$' or `$z$' is distributed so that it can be bounded based on the number of leaves of the original tree (and in certain cases uses a lower bound than that required for Theorem~\ref{thm:n+k+an}).

We now sketch the proof of Theorem~\ref{thm:n+k+an-partial} from Theorem~\ref{thm:extending-distillation}. Given an $n$-vertex $k$-leaf oriented tree $T$ satisfying the conditions of the theorem, we must construct an embedding $\psi$ of $T$ into an arbitrary $((1+\overline{\alpha})n+k)$-vertex tournament $G$. In broad terms, we do the following.
\begin{enumerate}
    \item Define a homomorphism $f$ from $T-V(T_0)$ to the fully-looped oriented forest $H$, with $f^{-1}(\{x^\diamond,\overline{x}^\diamond\})=N_T^\diamond(V(T_0))$ for each $\diamond\in\{+,-\}$.
    \item Define a weight function $\beta:V(H)\to[0,1]$, so that $\beta(v)$ is approximately the proportion of vertices of $T-V(T_0)$ mapped onto $v$ by $f$.
    \item Take an $\epsilon$-regular partition of (a large subtournament of) $G$ with vertex classes $V_1,\ldots,V_r$. Let $D$ be the complete looped digraph on $[r]$ with associated edge weights $d(j,j')$ approximately corresponding to the density of edges between $\epsilon$-regular pairs $(V_j,V_{j'})$.
    \item Apply Theorem~\ref{thm:extending-distillation} to obtain a probability distribution $\mathcal{D}$ satisfying \itref{ext-good-plan-1}--\itref{ext-good-plan-4}.
    \item Identify a good set of vertices $Z\subseteq V_{j_t}$ so that, for each $z\in Z$, if $\phi$ is sampled according to $\mathcal{D}$, then with high probability $z$ has at least the expected number of $\diamond$-neighbours in $V_{\phi(x^\diamond)}$ and $V_{\phi(\overline{x}^\diamond)}$ for each $\diamond\in\{+,-\}$.
    \item Using Theorem~\ref{thm:any_linear_bound}, let $\psi:T_0\to G$ embed $T_0$ into $Z$.
    \item By sampling $\phi$ according to $\mathcal{D}$ (with a small amount of conditioning) for each component of $T-V(T_0)$, and then composing $f$ with the outcomes, obtain a homomorphism $\hat{\varphi}$ from $T-V(T_0)$ to $D$.
    \item\label{extend-to-neighbours} Extend $\psi$ to cover $N_T(V(T_0))$, with $v$ embedded to $V_{\hat{\varphi}(v)}$ for every $v\in V(T_0)$. Note that this is easy with the established properties of $Z$, as $T[N_T(V(T_0))]$ consists of isolated vertices.
    \item\label{extend-to-rest} Use Proposition~\ref{prop:branch-embed} for each component of $T-V(T_0)$ to extend $\psi$ to cover $V(T)$, with $v$ embedded to $V_{\hat{\varphi}(v)}$ for every $v\in V(T_0)$.
\end{enumerate}


In the actual proof, some of the steps above require additional details that are omitted from this simplified overview. One notable case concerns steps~\ref{extend-to-neighbours} and~\ref{extend-to-rest}. Step~\ref{extend-to-neighbours} is essential, as if we were to embed the components of $T-V(T_0)$ one-by-one without first fixing the images of $N_T(V(T_0))$, we could fail due to occupying too many out-neighbours or too many in-neighbours of vertices in $\psi(V(T_0))$ which need lots of components attaching. However, having fixed the images of $N_T(V(T_0))$ in step~\ref{extend-to-neighbours}, the embedding of the components of $T-V(T_0)$ is restricted in a way that possibly prevents us from being able to exactly follow the allocation given by $\hat{\varphi}$ (in particular, Proposition~\ref{prop:branch-embed} cannot apply directly, as one vertex in each component has already had its embedding fixed). To handle this, we consider large components and small components separately. Precisely, using $S_x$ to denote the component of $T-V(T_0)$ containing $x\in N_T(V(T_0))$, we partition $N_T(V(T_0))=X_0\cup Y_0$ such that $S_x$ is at most constant-sized whenever $x\in X_0$ (and larger than constant-sized whenever $x\in Y_0$). We then embed as much of $\cup_{x\in X_0\cup Y_0}V(S_x)$ as possible. Given $x\in X_0$, if it is not possible to embed $S_x$ according to $\hat{\varphi}$, then we may still be able to find an embedding for $S_x$ by switching its allocation with an identical component. Indeed, if a significant number of $x\in X_0$ are yet to have their corresponding component embedded, then the bound on $|S_x|$ for $x\in X_0$ helps to find a suitable identical pair. Thus, we can extend the embedding to cover most of $\cup_{x\in X_0}V(S_x)$. On the other hand, $Y_0$ is small, and so the corresponding larger components can all be embedded greedily using specially reserved sets for their roots. Finally, any remaining components of $T-V(T_0)$ that are not embedded can then be handled by greedily embedding them to a random subset $U\subseteq V(G)$ reserved at the beginning of the proof.

\newcounter{embedstep}
\newcommand{\embedstep}{\stepcounter{embedstep}\textbf{[Step~\arabic{embedstep}]\ }}
\begin{proof}[Proof of Theorem~\ref{thm:n+k+an-partial}]
Let $\alpha=\bar{\alpha}/35$ and introduce constants $\mu,\epsilon,r_1,r_2$ such that $\eta\ll1/r_2\ll1/r_1\ll\epsilon\ll\mu\ll\alpha$. For each $x\in N_T(V(T_0))$, let $S_x$ be the component of $T-V(T_0)$ containing $x$.

Let $G$ be a $((1+35\alpha)n+k)$-vertex tournament, and note that $n\leqslant |G|\leqslant 3n$. Let $U\subseteq V(G)$ be a random subset, with elements from $V(G)$ chosen independently at random with probability $2\alpha$, and let $V'$ be the set of vertices $v\in V(G)\setminus U$ with $d_G^\pm(v,U)\geqslant4\alpha^2n$. By Proposition~\ref{prop:random-subset}, we may proceed assuming that $|U|\geqslant\alpha n$ and $|V'|\geqslant((1+11\alpha)n+k)$, where the step numbering matches the sketch given before this proof.

\embedstep Define $f_0:N_T(V(T_0))\to\{x^+,\bar{x}^+,x^-,\bar{x}^-\}$ as follows. For each $\diamond\in\{+,-\}$ and $v\in N_T^\diamond(V(T_0))$, set $f_0(v)=x^\diamond$ if $N_T^\diamond(v)\setminus V(T_0)\neq\emptyset$, and set $f_0(v)=\bar{x}^\diamond$ if $N_T^\diamond(v)\setminus V(T_0)=\emptyset$. Then, let $f:V(T)\setminus V(T_0)\to V(H)$ be the homomorphism from $T-V(T_0)$ to $H$ extending $f_0$ such that $f^{-1}(\{x^+,\bar{x}^+,x^-,\bar{x}^-\})=N_T(V(T_0))$, $f^{-1}(\{z^+,\bar{z}^+\})=N_T^-(N_T^+(V(T_0)))\setminus V(T_0)$, and $f^{-1}(\{z^-,\bar{z}^-\})=N_T^+(N_T^-(V(T_0)))\setminus V(T_0)$. (Note that this homomorphism is unique, as each vertex in $(N_T^-(N_T^+(V(T_0)))\cup N_T^+(N_T^-(V(T_0))))\setminus V(T_0)$ has only one viable candidate for its image among $\{z^+,\bar{z}^+,z^-,\bar{z}^-\}$, and once those vertices have their images fixed, each remaining vertex has only one viable candidate for its image among $\{y^+,u^+,w^+,\bar{u}^+,\bar{w}^+,y^-,u^-,w^-,\bar{u}^-,\bar{w}^-\}$.)

\embedstep Let $\beta:V(H)\to[0,1]$ be given by setting, for each $v\in V(H)$,
\begin{equation}\label{eq:beta_definition}
    \beta(v)=\frac{|f^{-1}(v)|+2\mu n}{|V(T)\setminus V(T_0)|+36\mu n},
\end{equation}
and note that, because $|V(H)|=18$, $\beta$ is a function satisfying $\sum_{v\in V(H)}\beta(v)=1$, and $\beta(v)\geq\mu$ for every $v\in V(H)$. Set
\begin{equation}\label{eq:gamma}
    \gamma=\max{\{\beta(x^+,\bar{x}^+),\beta(z^+,\bar{z}^+)\}}+\max{\{\beta(x^-,\bar{x}^-),\beta(z^-,\bar{z}^-)\}}.
\end{equation}
We remark that, for each $\diamond\in\{+,-\}$, if $v\in f^{-1}(x^\diamond)$, then there is some $v'\in N_T^\diamond(v)\setminus V(T_0)$ with $f(v')=y^\diamond$. Therefore, $\beta(y^+)\geq\beta(x^+)$ and $\beta(y^-)\geq\beta(x^-)$. We also remark that, for each $x\in N_T(V(T_0))$, the number of leaves of $T$ appearing in $S_x$ is at least $\max{\{1,|f^{-1}(\{z^+,\bar{z}^+,z^-,\bar{z}^-\})\cap V(S_x)|\}}$, and hence
\begin{align*}
    k&\geqslant\sum_{x\in N_T(V(T_0))}\max{\{1,|f^{-1}(\{z^+,\bar{z}^+,z^-,\bar{z}^-\})\cap V(S_x)|\}}\\
    &\geqslant\max{\{|N_T^+(V(T_0))|,|f^{-1}(\{z^+,\bar{z}^+\})|\}}+\max{\{|N_T^-(V(T_0))|,|f^{-1}(\{z^-,\bar{z}^-\})|\}}\\
    &=\max{\{|f^{-1}(\{x^+,\bar{x}^+\})|,|f^{-1}(\{z^+,\bar{z}^+\})|\}}+\max{\{|f^{-1}(\{x^-,\bar{x}^-\})|,|f^{-1}(\{z^-,\bar{z}^-\})|\}}.
\end{align*}
Therefore, by \eqref{eq:beta_definition} and \eqref{eq:gamma}, $\gamma\leq\alpha+k/n$ and hence $|V'|\geqslant(1+\gamma+10\alpha)\cdot n$.

\embedstep By Corollary~\ref{cor:regular_subtournament}, there is some $r$ with $r_1\leq r\leq r_2$ and disjoint subsets $V_1,\ldots,V_r\subseteq V'$, with $|V_1|=\ldots=|V_r|\geqslant (1+\gamma+9\alpha)\cdot n/r$, such that $V_1\cup\ldots\cup V_r$ is an $\epsilon$-regular partition of $G[V_1\cup\ldots\cup V_r]$. Let $D$ be a complete looped digraph on vertex set $[r]$ with edge weights $d(e)$, $e\in E(D)$ given by setting $d(j,j)=1$ for every $j\in[r]$, $d(j,j')=0$ for every $jj'\in E(D)$ for which $j\neq j'$ and $(V_j,V_{j'})$ is not an $\epsilon$-regular pair, and, for every $jj'\in E(D)$ for which $(V_j,V_{j'})$ forms an $\epsilon$-regular pair, setting
\begin{equation*}
d(j,j')=
\begin{cases*}
    1 & if $d(V_j,V_{j'})>1-\mu$, \\
    d(V_{j},V_{j'}) & if $\mu\leqslant d(V_{j},V_{j'})\leqslant1-\mu$,\\
    0 & if $d(V_{j},V_{j'})<\mu$.
\end{cases*}
\end{equation*}
We remark that the edge weights $d(e)$, $e\in E(D)$ are $\sqrt{\epsilon}$-complete, and, for $j\neq j'$, if $d(j,j')>0$, then $(V_j,V_{j'})$ is an $\epsilon$-regular pair with density satisfying $d(V_j,V_{j'})\geqslant\mu$ and $d(V_j,V_{j'})\geqslant(1-\mu)\cdot d(j,j')$.

\embedstep By Theorem~\ref{thm:extending-distillation} applied to $\beta$ and $D$, there is some $j_t\in [r]$ and probability distribution $\mathcal{D}$ on the set of functions from $V(H)$ to $V(D)$, such that, if $\phi$ is sampled according to $\mathcal{D}$, then \itref{ext-good-plan-1}-\itref{ext-good-plan-4} hold. Let $J_1$ be the set of $j\in[r]$ for which we have \itref{ext-good-plan-3-1}, and let $J_2=[r]\setminus J_1$, so that \itref{ext-good-plan-3-2} holds for every $j\in J_2$. For $j\in[r]$, let $U_j,W_j\subseteq V_j$ be disjoint subsets with $|U_j|=(1+\gamma+4\alpha)\cdot n/r$ and $|W_j|=3\alpha\cdot n/r$.

\embedstep Let $Z$ be the set of $z\in V_{j_t}\setminus(U_{j_t}\cup W_{j_t})$ such that the following holds with probability at least $1-\sqrt{\epsilon}$ whenever $\phi$ is sampled according to $\mathcal{D}$.
\begin{equation}\label{goodstuff}
\begin{aligned}
    d_G^+(z,U_{\phi(x^+)})&\geqslant d(j_t,\phi(x^+))\cdot(1+\gamma+3\alpha)\cdot n/r, &\qquad& d_G^+(z,W_{\phi(x^+)})\geqslant 2\alpha\mu\cdot n/r,\\
    d_G^+(z,U_{\phi(\bar{x}^+)})&\geqslant d(j_t,\phi(\bar{x}^+))\cdot(1+\gamma+3\alpha)\cdot n/r, &\qquad& d_G^+(z,W_{\phi(\bar{x}^+)})\geqslant 2\alpha\mu\cdot n/r,\\
    d_G^-(z,U_{\phi(x^-)})&\geqslant d(\phi(x^-),j_t)\cdot(1+\gamma+3\alpha)\cdot n/r, &\qquad& d_G^-(z,W_{\phi(x^-)})\geqslant 2\alpha\mu\cdot n/r,\\
    d_G^-(z,U_{\phi(\bar{x}^-)})&\geqslant d(\phi(\bar{x}^-),j_t)\cdot(1+\gamma+3\alpha)\cdot n/r, &\qquad& d_G^-(z,W_{\phi(\bar{x}^-)})\geqslant 2\alpha\mu\cdot n/r.
\end{aligned}
\end{equation}

\begin{claim}\label{clm:Z-size}
$|Z|\geqslant\alpha\cdot n/r$.
\end{claim}
\begin{proof}[Proof of Claim~\ref{clm:Z-size}]
Let $\bar{Z}$ be the set of $z\in V_{j_t}$ such that \eqref{goodstuff} fails with probability at least $\sqrt{\epsilon}$. If $|\bar{Z}|<\alpha\cdot n/r$, then, as $|V_{j_t}\setminus(U_{j_t}\cup W_{j_t})|\geqslant2\alpha\cdot n/r$, the claim follows. So assume for contradiction that $|\bar{Z}|\geqslant\alpha\cdot n/r$.

Let $\Omega$ be the set of homomorphisms $\bar{\phi}:H\to D$ such that $j_t\notin\bar{\phi}(\{x^+,\bar{x}^+,x^-,\bar{x}^-\})$ and $d(j_t,\bar{\phi}(x^+))$, $d(j_t,\bar{\phi}(\bar{x}^+))$, $d(\bar{\phi}(x^-),j_t))$, and $d(\bar{\phi}(\bar{x}^-),j_t)$ are all positive, and hence, by our choice of $d(e)$, $e\in E(D)$, are all at least $\mu$. Note that, by \itref{ext-good-plan-1}, $\mathbb{P}(\phi\in\Omega)=1$. Given $\bar{\phi}\in\Omega$, let $B_{\bar{\phi}}$ be the set of $z\in V_{j_t}$ such that \eqref{goodstuff} fails for $\phi=\bar{\phi}$. We claim that $|B_{\bar{\phi}}|\leqslant24\epsilon\cdot n/r$ for every $\bar{\phi}\in\Omega$. Indeed, if $\bar{\phi}\in\Omega$, then $(V_{j_t},V_{\bar{\phi}(x^+)})$ is an $\epsilon$-regular pair of density $d(V_{j_t},V_{\bar{\phi}(x^+)})\geqslant \min{\{d(j_t,\bar{\phi}(x^+)),1-\mu\}}\geqslant\mu$, and so the number of $z\in V_{j_t}$ for which we do not have $d_G^+(z,U_{\bar{\phi}(x^+)})\geqslant d(j_t,\bar{\phi}(x^+))\cdot(1+\gamma+3\alpha)\cdot n/r$ is at most $\epsilon|V_{j_t}|\leqslant3\epsilon\cdot n/r$ (using, for example, Proposition~\ref{prop:outneighbourhood_to_many_clusters}, with $r=1$ and $\mu'=d(V_{j_t},V_{\bar{\phi}(x^+)})$). More generally, if $\bar{\phi}\in\Omega$, then $(V_{j_t},V_{\bar{\phi}(x^+)})$, $(V_{j_t},V_{\bar{\phi}(\bar{x}^+)})$, $(V_{\bar{\phi}(x^-)},V_{j_t})$, and $(V_{\bar{\phi}(\bar{x}^-)},V_{j_t})$ all form $\epsilon$-regular pairs (of density at least $\mu$), and thus each one of the inequalities of \eqref{goodstuff} fails for at most $3\epsilon\cdot n/r$ many $z\in V_{j_t}$. Hence we have $|B_{\bar{\phi}}|\leqslant8\cdot3\epsilon \cdot n/r=24\epsilon\cdot n/r$ for every $\bar{\phi}\in\Omega$, as claimed. But then
\begin{equation*}
    \alpha\sqrt{\epsilon}\cdot n/r\leqslant|\bar{Z}|\cdot\sqrt{\epsilon}\leqslant\sum_{\bar{\phi}\in\Omega}\mathbb{P}(\phi=\bar{\phi})\cdot|B_{\bar{\phi}}|\leqslant24\epsilon\cdot n/r,
\end{equation*}
a contradiction as $\epsilon\ll\alpha$.
\renewcommand{\qedsymbol}{$\boxdot$}
\end{proof}
\renewcommand{\qedsymbol}{$\square$}

\embedstep Note that, by Claim~\ref{clm:Z-size} and as $\eta\ll\alpha,1/r_2$, $|Z|\geqslant3\eta n\geqslant 3|T_0|$. Therefore, using Theorem~\ref{thm:any_linear_bound}, let $\psi:T_0\to G$ be an embedding so that $\psi(V(T_0))\subseteq Z$. For each $x\in N_T(V(T_0))$, let $z_x\in Z$ be the image under $\psi$ of the unique neighbour of $x$ in $V(T_0)$. Our aim now is to extend $\psi$ to cover the components $S_x$, $x\in N_T(V(T_0))$, with each $\psi(x)$ in the appropriate in- or out-neighbourhood of $z_x$.

\embedstep Given $v\in V(T)\setminus V(T_0)$, let $x(v)\in N_T(V(T_0))$ be the unique vertex such that $v\in V(S_{x(v)})$. For each $x\in N_T(V(T_0))$, choose a homomorphism $\phi_x:H\to D$ by sampling $\phi$ according to $\mathcal{D}$ conditioned on the event that \eqref{goodstuff} holds for $z=z_x$. Define a function $\hat{\varphi}:V(T)\setminus V(T_0)\to[r]$ by setting $\hat{\varphi}(v)=\phi_{x(v)}(f(v))$. We remark that $\hat{\varphi}$ is a homomorphism from $T-V(T_0)$ to $D$, with $|\hat{\varphi}(V(S_x))|\leqslant|H|$ for every $x\in N_T(V(T_0))$.

Before continuing with the final steps of the embedding procedure, we will first establish some additional notation and useful properties of $\hat{\varphi}$ . Let $X_0$ be the set of $x\in N_T(V(T_0))$ with $|S_x|\leqslant1/\mu^3$, and let $Y_0=N_T(V(T_0))\setminus X_0$, so that $|S_x|>1/\mu^3$ whenever $x\in Y_0$. Note that $|Y_0|\leqslant\mu^3n$. Roughly speaking, we will try to embed each $v\in V(T)\setminus V(T_0)$ into $V_{\hat{\varphi}(v)}$, with each $v\in Y_0$ embedded into $W_{\hat{\varphi}(v)}$ and each $v\in V(T)\setminus (V(T_0)\cup Y_0)$ embedded into $U_{\hat{\varphi}(v)}$. This motivates the following claim, which we will prove later. For this, for each $j\in[r]$ and $\diamond\in\{+,-\}$, let $X_j^\diamond$ (respectively, $Y_j^\diamond$) be the set of vertices in $X_0$ (respectively, $Y_0$) which are $\diamond$-neighbours of $V(T_0)$ and allocated to $V_j$ by $\hat{\varphi}$. That is, for each $j\in [r]$ and $\diamond\in\{+,-\}$, let $X^\diamond_j=X_0\cap f^{-1}(X^\diamond)\cap\hat{\varphi}^{-1}(j)$ and $Y^\diamond_j=Y_0\cap f^{-1}(X^\diamond)\cap\hat{\varphi}^{-1}(j)$.

\begin{claim}\label{clm:varphi}
With probability at least $3/4$, the following properties hold.
\stepcounter{propcounter}
\begin{enumerate}[label =\emph{\textbf{\Alph{propcounter}\arabic{enumi}}}]
    \item\label{varphi1} For every $j\in[r]$, $|\hat{\varphi}^{-1}(j)|\leqslant(1+\gamma+3\alpha)\cdot n/r$.
    \item\label{varphi2} For every $j\in[r]$ and $x\in X_j^+\cup X_j^-$,
    \begin{enumerate}[label=\emph{\textbf{\Alph{propcounter}\arabic{enumi}.\arabic{enumii}}}]
        \item\label{varphi2-1} if $j\in J_1$, then $d_G^+(z_x,U_j)\geqslant|X_j^+|$ if $x\in X_j^+$ and $d_G^-(z_x,U_j)\geqslant|X_j^+\cup X_j^-|$ if $x\in X_j^-$;
        \item\label{varphi2-2} if $j\in J_2$, then $d_G^-(z_x,U_j)\geqslant|X_j^-|$ if $x\in X_j^-$ and $d_G^+(z_x,U_j)\geqslant|X_j^+\cup X_j^-|$ if $x\in X_j^+$.
    \end{enumerate}
    \item\label{varphi3} $|Y_j^+\cup Y_j^-|\leqslant\alpha\mu\cdot n/r$ for every $j\in[r]$.
\end{enumerate}
\end{claim}

We therefore proceed with the assumption that properties \itref{varphi1}-\itref{varphi3} hold.

\embedstep Extend $\psi$ to cover $X_0$ as follows.
\begin{itemize}
    \item For each $j\in J_1$, using \itref{varphi2-1}, greedily extend $\psi$ to first cover $X_j^+$, and then to cover $X_j^-$, so that $\psi(X_j^+\cup X_j^-)\subseteq U_j$.
    \item For each $j\in J_2$, using \itref{varphi2-2} greedily extend $\psi$ to first cover $X_j^-$, and then to cover $X_j^+$, so that $\psi(X_j^+\cup X_j^-)\subseteq U_j$.
\end{itemize}

\embedstep Let $X'\subseteq X_0\cup Y_0$ be a maximal set such that there exists a homomorphism $\varphi$ from $T-V(T_0)$ to $D$ and an extension of $\psi$ covering $\cup_{x\in X'}V(S_x)$ such that the following properties hold.
\stepcounter{propcounter}
\begin{enumerate}[label =\textbf{\Alph{propcounter}\arabic{enumi}}]
    \item\label{extension-1}$\varphi(v)=\hat{\varphi}(v)$ for every $v\in X_0\cup(\cup_{x\in Y_0}V(S_x))$.
    \item\label{extension-2}$|\varphi(V(S_x))|\leqslant|H|$ for every $x\in X_0\cup Y_0$.
    \item\label{extension-3}$|\varphi^{-1}(j)|=|\hat{\varphi}^{-1}(j)|$ for every $j\in[r]$.
    \item\label{extension-4} If $x\in X'\cap Y_0$ then $\psi(x)\in W_{\varphi(x)}$, and if $v\in (\cup_{x\in X'}V(S_x))\setminus Y_0$ then $\psi(v)\in U_{\varphi(v)}$.
\end{enumerate}
We remark that this is well-defined, as we may take $X'=\emptyset$ and $\varphi=\hat{\varphi}$.

For this maximal $X'$, take $(\varphi,\psi)$ so that \ref{extension-1}-\ref{extension-4} hold, and let $A=\psi(V(T_0)\cup X_0\cup(\cup_{x\in X'}V(S_x)))$. Note that, by~\itref{varphi1},~\ref{extension-3} and~\ref{extension-4}, we have
\begin{equation}\label{eq:U-space}
|U_j\setminus A|\geq\alpha\cdot n/r
\end{equation}
for every $j\in[r]$. We now show that $X'$ includes all of $Y_0$ and almost all of $X_0$.
\begin{claim}\label{clm:Y-in-X'}
$Y_0\subseteq X'$.
\end{claim}
\begin{proof}[Proof of Claim~\ref{clm:Y-in-X'}]
For any $\diamond\in\{+,-\}$ and $x\in N_T^\diamond(V(T_0))\cap Y_0$, we have 
\begin{equation*}
    |N_G^\diamond(z_x,W_{\varphi(x)})\setminus A|\overset{(\ref{goodstuff})}{\geq}2\alpha\mu\cdot n/r-|Y_j^+\cup Y_j^-|\overset{\itref{varphi3}}{\geq}\alpha\mu\cdot n/r.
\end{equation*}
So if $x\in Y_0\setminus X'$, then, by Proposition~\ref{prop:branch-embed} and \eqref{eq:U-space}, $\psi$ can be extended to cover $V(S_x)$ with $\psi(x)\in W_{\varphi(x)}\setminus A$ and $\psi(v)\in U_{\varphi(v)}\setminus A$ whenever $v\in V(S_x)\setminus\{x\}$, contradicting the maximality of $X'$. So we must have $Y_0\subseteq X'$.
\renewcommand{\qedsymbol}{$\boxdot$}
\end{proof}
\renewcommand{\qedsymbol}{$\square$}

\begin{claim}\label{clm:X-in-X'}
$|X_0\setminus X'|\leqslant \mu^4n$.
\end{claim}
\begin{proof}[Proof of Claim~\ref{clm:X-in-X'}]
For each $m\in\mathbb{N}$, let $g(m)$ denote the number of rooted oriented trees with at most $m$ vertices. Suppose, for contradiction, that $|X_0\setminus X'|>\mu^4n$. Then there is some $j\in[r]$ with $|(X_j^+\cup X_j^-)\setminus X'|>\mu^4\cdot n/r$. Therefore, there is some rooted oriented tree $S$ such that, if $X_j^S$ is the set of $x\in(X_j^+\cup X_j^-)\setminus X'$ for which $S_x$ is isomorphic to $S$, then $|X_j^S|\geqslant(\mu^4/g(\lfloor1/\mu^3\rfloor))\cdot n/r$.

Choose $x_1\in X_j^S$ arbitrarily. By Proposition~\ref{prop:branch-embed} and \eqref{eq:U-space}, there is a copy of $S_{x_1}$ in $G$, with each $v\in V(S_{x_1})\setminus\{x_1\}$ copied to $U_{\varphi(v)}\setminus A$ and $x_1$ copied to $\psi(X_j^S)$, and let $x_2\in X_j^S$ be such that $\psi(x_2)$ is the image of $x_1$ in this copy. Because $S_{x_1}$ and $S_{x_2}$ are isomorphic, we may regard this as a copy of $S_{x_2}$, and use this copy to extend $\psi$ to cover $V(S_{x_2})$. \itref{ext-good-plan-4} here ensures that the only vertex in $\psi(X_j^S)$ appearing in the copy of $S_{x_2}$ is that identified with $x_2$, and so this extension of $\phi$ does not conflict with the existing images of $X_j^S$.

Let $\rho$ be an automorphism of $T-V(T_0)$ with $\rho(S_{x_1})=S_{x_2}$, $\rho(S_{x_2})=S_{x_1}$, and $\rho(v)=v$ whenever $v\notin V(S_{x_1})\cup V(S_{x_2})$. Note that $\psi(v)\in U_{\varphi(\rho(v))}$ whenever $v\in (\cup_{x\in X'\cup\{x_2\}}V(S_x))\setminus Y_0$, and so $\varphi\circ\rho$ is a homomorphism from $T-V(T_0)$ to $D$ also satisfying~\ref{extension-1}-\ref{extension-4}. So using this extension of $\psi$ and the homomorphism $\varphi\circ\rho$, we may add $x_2$ to $X'$, a contradiction.
\renewcommand{\qedsymbol}{$\boxdot$}
\end{proof}
\renewcommand{\qedsymbol}{$\square$}

We now have an embedding of a subtree $T[\psi^{-1}(A)]\subseteq T$ into $G[V']$, where, using Claims~\ref{clm:Y-in-X'} and~\ref{clm:X-in-X'},
\begin{equation*}
    |V(T)\setminus\psi^{-1}(A)|\leqslant\sum_{x\in (X_0\cup Y_0)\setminus X'}|S_x|\leqslant\mu n.
\end{equation*}
Recall that we also have $d_G^\pm(v,U)\geqslant4\alpha^2n\geqslant3\mu n$ for every $v\in V'$. Therefore, by Corollary~\ref{cor:tree-extension}, this embedding can be extended to an embedding of $T$ into $G$ with the vertices of $V(T)\setminus\psi^{-1}(A)$ embedded into $U$. All that remains now is to prove Claim~\ref{clm:varphi}.

\begin{proof}[Proof of Claim~\ref{clm:varphi}]
We will prove that each of the properties \itref{varphi1}-\itref{varphi3} fails with probability at most $1/16$, and so the claim then follows.

\itref{varphi1}: As each $\phi_x$ was chosen previously by sampling $\phi$ conditioned on an event which holds with probability at least $(1-\sqrt{\epsilon})$, we have that for any $v\in V(T)\setminus V(T_0)$, $j\in[r]$,
\begin{equation}\label{eq:conditioning}
    \mathbb{P}(\phi_{x(v)}(f(v))=j)\leqslant(1-\sqrt{\epsilon})^{-1}\mathbb{P}(\phi(f(v))=j).
\end{equation}
Thus, we find that, for any $w\in V(H)$, $j\in[r]$,
\begin{align}
    \E(|\hat{\varphi}^{-1}(j)\cap f^{-1}(w)|)&=\sum_{v\in f^{-1}(w)}\mathbb{P}(\phi_{x(v)}(f(v))=j)\nonumber\\
    &\overset{\eqref{eq:conditioning}}{\leqslant}\sum_{v\in f^{-1}(w)}(1-\sqrt{\epsilon})^{-1}\mathbb{P}(\phi(f(v))=j)\overset{\text{(\ref{eq:beta_definition})}}{\leqslant}(1+\sqrt{\mu})\cdot\E(\beta(\phi^{-1}(j)\cap \{w\}))\cdot n.\label{eq:aux-rv}
\end{align}
For any $j\in[r]$, we have
\begin{align}
    \E(|\hat{\varphi}^{-1}(j)|)&=\sum_{w\in V(H)}\E(|\hat{\varphi}^{-1}(j)\cap f^{-1}(w)|)\overset{\text{(\ref{eq:aux-rv}),\itref{ext-good-plan-2}}}{\leqslant}\E(\beta(\phi^{-1}(j)))\cdot n+\alpha\cdot n/r\overset{\itref{ext-good-plan-2}}{\leqslant}(1+\gamma+2\alpha)\cdot n/r.\label{eq:exp-j}
\end{align}
Therefore, as
\begin{equation}\label{eq:sum-squares}
    \sum_{x\in X_0\cup Y_0}|S_x|^2\leqslant \sum_{x\in X_0\cup Y_0}|S_x|\cdot\max_{x\in X_0\cup Y_0}{|S_x|}\leqslant \eta n^2,
\end{equation}
we find that
\begin{align*}
    \mathbb{P}(|\hat{\varphi}^{-1}(j)|>(1+\gamma+3\alpha)\cdot n/r)&\overset{\eqref{eq:exp-j}}{\leqslant}\mathbb{P}(||\hat{\varphi}^{-1}(j)|-\E(|\hat{\varphi}^{-1}(j)|)|\geqslant\alpha\cdot n/r)\\
    &\hspace{-0.65cm}\overset{\text{Theorem~\ref{thm:hoeffding-inequality}}}{\leqslant}2\exp{\left(-\frac{2\alpha^2\cdot n^2/r^2}{\sum_{x\in X_0\cup Y_0}|S_x|^2}\right)}\overset{\eqref{eq:sum-squares}}{\leqslant}2\exp{\left(-\frac{2\alpha^2}{\eta r^2}\right)},
\end{align*}
and so the probability that \itref{varphi1} fails is at most $2r\cdot\exp{(-2\alpha^2/\eta r^2)}<1/16$.

\itref{varphi2}: We first note that, for any $j\in[r]$ and $\diamond\in\{+,-\}$,
\begin{equation}\label{eq:exp-j-X}
    \E(|X_j^\diamond|)=\sum_{w\in X^\diamond}\E(|\hat{\varphi}^{-1}(j)\cap f^{-1}(w)|)\overset{\text{(\ref{eq:aux-rv})}}{\leqslant}(1+\sqrt{\mu})\cdot\E(\beta(\phi^{-1}(j)\cap X^\diamond))\cdot n.
\end{equation}
Also,
\begin{align*}
    \mathbb{P}(||X_j^\diamond|-\E(|X_j^\diamond|)|\geqslant\mu^2\cdot n/r)&\overset{\text{Theorem~\ref{thm:hoeffding-inequality}}}{\leqslant}2\exp{\left(-\frac{2\mu^4\cdot n^2/r^2}{|X_0\cap f^{-1}(X^\diamond)|}\right)}\leqslant2\exp{\left(-\frac{2\mu^4}{r^2}\cdot n\right)}.
\end{align*}
Therefore, with probability at least $1-4r\cdot\exp{(-(2\mu^4/r^2)\cdot n)}>15/16$, we have
\begin{equation}\label{eq:exp-Xj-reduced}
    ||X_j^\diamond|-\E(|X_j^\diamond|)|\leqslant\mu^2\cdot n/r\qquad\text{for every $j\in[r]$, $\diamond\in\{+,-\}$.}
\end{equation}
Thus, it is enough to show that \itref{varphi2} follows from \eqref{eq:exp-Xj-reduced}. Indeed, for $j\in J_1$, if $\mathbb{P}(|X_j^+|>0),\mathbb{P}(|X_j^-|>0)>0$, then $d(j_t,j),d(j,j_t)>\mu$, and so for any $x\in X_j^+$ we have
\begin{equation*}
    |X_j^+|\overset{\eqref{eq:exp-Xj-reduced}}{\leqslant}\E(|X_j^+|)+\mu^2\cdot n/r\overset{\eqref{eq:exp-j-X},\itref{ext-good-plan-3-1}}{\leqslant}(1+\sqrt{\mu})d(j_t,j)(1+\gamma+\alpha)\cdot n/r+\mu^2\cdot n/r\overset{\eqref{goodstuff}}{\leqslant}d_G^+(z_x,U_j),
\end{equation*}
and for any $x\in X_j^-$ we have
\begin{align*}
    |X_j^+\cup X_j^-|&\overset{\eqref{eq:exp-Xj-reduced}}{\leqslant}\E(|X_j^+\cup X_j^-|)+2\mu^2\cdot n/r\\
    &\hspace{-0.40cm}\overset{\eqref{eq:exp-j-X},\itref{ext-good-plan-3-1}}{\leqslant}(1+\sqrt{\mu})d(j,j_t)(1+\gamma+\alpha)\cdot n/r+2\mu^2\cdot n/r\overset{\eqref{goodstuff}}{\leqslant}d_G^-(z_x,U_j),
\end{align*}
and so~\itref{varphi2-1} holds. If instead $|X_j^+|=0$ with probability 1 or $|X_j^-|=0$ with probability 1, then the same conclusion holds. Similarly, if $j\in J_2$ then \eqref{eq:exp-Xj-reduced} implies~\itref{varphi2-2}.

\itref{varphi3}: Note that if $x\in N_T(V(T_0))$ and $j\in[r]$, then, because $\beta(f(x))\geqslant\mu$,
\begin{align*}
    \mathbb{P}(\hat{\varphi}(x)=j)&=\mathbb{P}(\phi_x(f(x))=j)\overset{\eqref{eq:conditioning}}{\leqslant}(1-\sqrt{\epsilon})^{-1}\mathbb{P}(\phi(f(x))=j)\\
    &=(1-\sqrt{\epsilon})^{-1}\frac{1}{\beta(f(x))}\cdot\E(\beta(\phi^{-1}(j)\cap\{f(x)\}))\overset{\text{\itref{ext-good-plan-2}}}{\leqslant}6/\mu r,
\end{align*}
and so, for any $j\in[r]$, we have $\E(|Y_j^+\cup Y_j^-|)\leqslant6\mu^2\cdot n/r$. Therefore, for any $j\in[r]$,
\begin{align*}
\mathbb{P}(|Y_j^+\cup Y_j^-|>\alpha\mu\cdot n/r)&\leqslant\mathbb{P}(||Y_j^+\cup Y_j^-|-\E(|Y_j^+\cup Y_j^-|)|>\mu^2\cdot n/r)\\
&\hspace{-0.65cm}\overset{\text{Theorem~\ref{thm:hoeffding-inequality}}}{\leqslant}2\exp{\left(-\frac{2\mu^4\cdot n^2/r^2}{|Y_0|}\right)}\leqslant2\exp{\left(-\frac{2\mu^4}{r^2}\cdot n\right)},
\end{align*}
and so, the probability that \itref{varphi3} fails is at most $r\cdot\exp{(-(2\mu^4/r^2)\cdot n)}<1/16$.
\renewcommand{\qedsymbol}{$\boxdot$}
\qedhere
\renewcommand{\qedsymbol}{$\square$}
\qedsymbol
\end{proof}
\renewcommand{\qedsymbol}{}
\end{proof}
\renewcommand{\qedsymbol}{$\square$}

\section{Theorem~\ref{thm:n+an}: embedding the core and attached small  trees}\label{sect:T_1-bounded}
In this section, following the proof outline in Section~\ref{sect:outline}, we embed $T_0$ and $T_1$ for Theorem~\ref{thm:n+an}, doing so in the form of the following result, Theorem~\ref{thm:n+an-partial}. (This compares to our work in Section~\ref{sect:T_1-unbounded} for Theorem~\ref{thm:n+k+an}, proving Theorem~\ref{thm:n+k+an-partial}.) 

\begin{theorem}\label{thm:n+an-partial}
Let $1/n\ll\eta\ll\alpha$. Suppose $T$ is an $n$-vertex oriented tree with a subtree $T_0\subseteq T$, such that $|T_0|\leqslant\eta n$ and $T$ is formed from $T_0$ by attaching to each vertex $v$ of $T_0$ a tree $S_v$ with $|S_v|\leqslant\eta n$. Then, any $(1+\alpha)n$-vertex tournament contains a copy of $T$.
\end{theorem}

Note that there is no direct maximum degree imposed on $T$ in Theorem~\ref{thm:n+an-partial}, but, as (exactly) one tree is attached to each vertex in $T_0$ to get $T$, it follows that $\Delta(T)\leq 2\eta n$. As with Theorem~\ref{thm:n+k+an-partial}, the proof of Theorem~\ref{thm:n+an-partial} consists of a procedure for constructing an embedding of any relevant oriented tree into an arbitrary tournament of size $(1+\alpha)n$. On a high level, the procedure is as follows.

\begin{enumerate}
    \item Take an $\epsilon$-regular partition of (a large subtournament of) $G$ with vertex classes $V_1,\ldots,V_r$. Let $R$ be a $\sqrt{\epsilon}$-almost tournament (see Section~\ref{sect:regularity}) on $[r]$ such that $(V_i,V_j)$ is an $\epsilon$-regular pair of not-too-small density whenever $i\to_R j$.
    \item\label{s-allocation} For some $s$, find subsets $\{j_\ell\},I_\ell^+,I_\ell^-\subseteq[r]$ for $\ell\in[s]$, all disjoint, along with a partition $V(T_0)=X_1\cup\ldots\cup X_s$, such that
    \begin{enumerate}
        \item the edge relations indicated in Figure~\ref{fig:caterpillar} are satisfied in $R$,
        \item there are no edges of $T_0$ directed from $X_i$ to $X_j$ whenever $i>j$, and
        \item for each $\diamond\in\{+,-\}$ and $\ell\in[r]$, the total number of vertices contained in components of $T-T_0$ attached to $X_\ell$ by a $\diamond$-edge is at most $|I_\ell^\diamond|\cdot(1+\alpha/4)\cdot n/r$.
    \end{enumerate}
    \item Take any partition $V(T_0)=Y_1\cup\ldots\cup Y_\tau$ such that, for each $t\in[\tau]$, $T[Y_t]$ is a connected component of some $T_0[X_\ell]$ and $T_0[Y_1\cup\ldots\cup Y_t]$ is a tree.
    \item\label{s-embedding} Initially let $\psi$ be an empty embedding. Then, for each $t\in[\tau]$ in turn, extend $\psi$ to cover $\cup_{v\in Y_t}V(S_v)$ by embedding $T[Y_t]\subseteq T[X_\ell]$ into $G[V_{j_\ell}]$ and, for each $\diamond\in\{+,-\}$, embedding any components of $T-T_0$ attached to $Y_t$ by a $\diamond$-edge into $G[\cup_{i\in I_\ell^\diamond}V_i]$.
\end{enumerate}

\begin{figure}
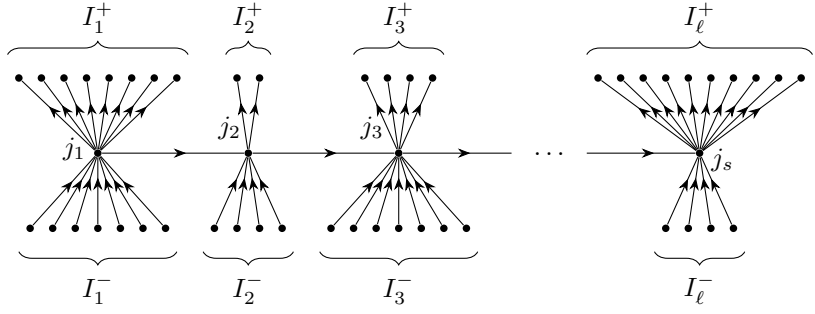

\DiagramCaterpillarVertices
\vspace{-0.75cm}
\caption{A `caterpillar-like' digraph, as appearing in Lemma~\ref{lm:caterpillar_updated} and Corollary~\ref{cor:caterpillar}. While the other edges are omitted for legibility, $j_{1}\rightarrow j_{2}\rightarrow \ldots \rightarrow j_{s}$ is the underlying directed path of a transitive tournament.}\label{fig:caterpillar}
\end{figure}

In order to ensure the successful execution of this procedure, two key technical lemmas are required. The first of these, which is proven in Section~\ref{sect:caterpillar}, handles the entirety of step~\ref{s-allocation}, where vertices of $T$ are effectively allocated to regularity clusters. The other, which is proven in Section~\ref{sect:n+an-embedding} before the full proof of Theorem~\ref{thm:n+an-partial}, shows that a star-like tree can be embedded into a collection of regularity clusters which form a star in the corresponding reduced graph, thus verifying each of the claimed extensions of $\psi$ in step~\ref{s-embedding} is possible.


\subsection{Allocating vertices for Theorem~\ref{thm:n+an-partial}}\label{sect:caterpillar}

Let us consider step~\ref{s-allocation} of the embedding procedure for Theorem~\ref{thm:n+an-partial} in more detail. In order to allocate vertices of $V(T)$ approximately evenly among regularity clusters, we will find a `caterpillar-like' digraph, realised through the sets $\{j_\ell\},I_\ell^+,I_\ell^-\subseteq[r]$ for $\ell\in[s]$, covering most of $R$ (see Figure~\ref{fig:caterpillar}). The vertices $j_1,\dots,j_s$ form a transitive tournament running along the body of the caterpillar, with $I_\ell^+$ and $I_\ell^-$ being the out- and in-leaves attached to $j_\ell$. As we will embed each $v\in V(T_0)$ into some $V_{j_\ell}$ and components of $T-T_0$ attached to $v$ by a $\diamond$-edge into $\cup_{i\in I_\ell^\diamond}V_i$, not just any caterpillar will do; we must ensure that for some partition $V(T_0)=X_1\cup\ldots\cup X_s$, where the edges between $X_i$ and $X_j$ are faithful to the ordering of $i$ and $j$, $I_\ell^\diamond$ is always large enough so that there is enough space in $\cup_{i\in I_\ell^\diamond}V_i$ to embed components of $T-T_0$ attached to $X_\ell$ by a $\diamond$-edge.

For this reason, the `caterpillar-like' digraph and the corresponding partition of $V(T_0)$ are essentially found together. A general method for finding a `caterpillar-like' digraph satisfying certain constraints is presented in Lemma~\ref{lm:caterpillar_updated}. This is then applied to a weight function naturally arising from the relative distribution of components of $T-T_0$, in order to produce a full description of the `caterpillar-like' digraph and corresponding partition of $V(T_0)$ in Corollary~\ref{cor:caterpillar}. As Corollary~\ref{cor:caterpillar} will eventually apply to the reduced digraph $R$ for an $\epsilon$-regular partition $V_1\cup\ldots\cup V_r$ (see Section~\ref{sect:regularity}), it applies to an $\epsilon$-almost tournament $R$. In its statement, \itref{caterpillar11} ensures that $j_1,\ldots,j_s$ forms a transitive tournament, \itref{caterpillar01} ensures the faithfulness of the partition $V(T_0)=X_1\cup\ldots\cup X_s$, and \itref{caterpillar21} ensures enough space in the corresponding regularity clusters for the subsequent embedding. Thus overall, Corollary~\ref{cor:caterpillar} ensures the successful completion of step~\ref{s-allocation} of the embedding procedure sketched earlier.


\begin{lemma}\label{lm:caterpillar_updated}
Let $\eps>0$ and $\bar{s},m,r\in \N$. Let $n^+_i,n^-_i\in \N$, $i\in [\bar{s}]$, be such that $m\leq n_i^++n_i^-\leq 4m$ for each $i\in [\bar{s}]$. Suppose that $R$ is an oriented graph on $[r]$ in which, for each $j\in[r]$,
\begin{equation}\label{eqn:Rbigenough}
d^+_R(j)+d^-_R(j)\geq|R|-m\geq(25+1000\log{\bar{s}})m+\sum_{i\in [\bar{s}]}(n_i^++n_i^-).
\end{equation}
Then, there is some $s\in [\bar{s}]$ for which there exists $0=i_0<i_1<\ldots<i_{s-1}<i_{s}=\bar{s}$, and subsets $\{j_\ell\}, I_\ell^+,I_\ell^-\subseteq[r]$ for $\ell\in[s]$, all disjoint, with the following properties.
\stepcounter{propcounter}
\begin{enumerate}[label =\emph{\textbf{\Alph{propcounter}\arabic{enumi}}}]
\item\label{caterpillar1} $j_{\ell_1}\rightarrow_Rj_{\ell_2}$ whenever $\ell_1<\ell_2$.
\item\label{caterpillar2} For each $\ell\in[s]$ and $\diamond\in\{+,-\}$, we have $I_\ell^\diamond\subseteq N_R^\diamond(j_\ell)$, and
\begin{equation*}
    |I_\ell^\diamond|=\sum_{i=i_{\ell-1}+1}^{i_\ell}n^\diamond_i.
\end{equation*}
\end{enumerate}
\end{lemma}

\begin{proof} Fix $m\in \N$ and $\eps>0$. We will show, by strong induction on $\bar{s}$, that the lemma holds for each $\bar{s}\geq 1$.

First, suppose $\bar{s}=1$. It follows from \eqref{eqn:Rbigenough} that $R$ is a $(1/25)$-almost tournament with $|R|\geqslant25m$, and therefore, by Proposition~\ref{prop:max_digraph_split}, there is some $j_1\in[r]$ such that $d_R^+(j_1),d_R^-(j_1)\geqslant |R|/5 \geq 4m$. If we set $I_1^\diamond\subseteq N_R^\diamond(j_1)$ with $|I_1^\diamond|=n_1^\diamond$ for $\diamond\in\{+,-\}$, then all the required properties are satisfied.

Suppose then that $\bar{s}>1$. It follows from \eqref{eqn:Rbigenough} that $R$ is an $(1/25)$-almost tournament with $|R|\geqslant25m$, and therefore, by Proposition~\ref{prop:max_digraph_split}, there is some $j\in[r]$ such that $d_R^+(j),d_R^-(j)\geqslant |R|/5$. Now, by \eqref{eqn:Rbigenough}, we have $d_R^+(j)+d^-_R(j)\geq \sum_{i\in [\bar{s}]}(n^+(i)+n^-(i))$.
Therefore, at least one of $\sum_{i\in[\bar{s}]}n^+(i)\leqslant d_R^+(j)$ or $\sum_{i\in[\bar{s}]}n^-(i)\leqslant d_R^-(j)$ holds. If both inequalities hold, then the desired result follows by taking $s=1$, $i_0=0$, $i_1=\bar{s}$, $j_1=j$, and $I_1^\diamond\subseteq N_R^\diamond(j)$ with $|I_1^\diamond|=\sum_{i\in [\bar{s}]}n_i^\diamond$ for each $\diamond\in\{+,-\}$. Otherwise, by directional duality, we may assume that $\sum_{i\in[\bar{s}]}n^+(i)\leqslant d_R^+(j)$ and $\sum_{i\in[\bar{s}]}n^-(i)> d_R^-(j)$.

Then, let $s'\in [\bar{s}-1]$ be maximal such that
\[
\sum_{i\in[s']}n^-(i)\leq d_R^-(j).
\]
Note that, as
\[
d_R^-(j)\geq \frac{|R|}{5}\overset{\eqref{eqn:Rbigenough}}{\geq} \frac{1}{5}\sum_{i\in [\bar{s}]}(n_i^++n_i^-)\geq \frac{\bar{s}m}{5}
\]
and
$n_i^-\leq (n_i^++n_i^-)\leq 4m$ for each $i\in [\bar{s}]$, we have $s'\geq \bar{s}/20$. Furthermore, by the maximality of $s'$, we have
\begin{equation}\label{eqn:frommin}
\sum_{i\in[s']}n^-(i)\geq d_R^-(j)-4m.
\end{equation}

Let $i_0=0$, $i_1=s'$, and $j_1=j$. Let $I_1^-\subseteq d_R^-(j)$ have size $\sum_{i\in[s']}n^-(i)$. Using that $d_R^+(j)\geq \sum_{i\in [\bar{s}]}n^+(i)$, let $I_1^+\subseteq d_R^+(j)$ have size $\sum_{i\in [s']}n^+(i)$ and let $I=N_R^+(j)\setminus I^+_1$.

 Now, $\bar{s}-s'\leq \bar{s}-\bar{s}/20= 19\bar{s}/20$ so that $1000\log (\bar{s}-s')\leq -5+1000\log \bar{s}$, and hence
\begin{align*}
|I|&= d_R^+(j)-|I_I^+|\overset{\eqref{eqn:frommin}}{\geq} d_R^+(j)-|I_1^+|+d_R^-(j)-|I_1^-|-4m\\
&\hspace{-1mm}\overset{\eqref{eqn:Rbigenough}}{\geq} (25+1000\log \bar{s})m+\sum_{i\in [\bar{s}]}(n_i^++n_i^-)-|I_1^+|-|I_1^-|-4m\\
&= (21+1000\log \bar{s})m+\sum_{i\in [\bar{s}]\setminus [s']}(n_i^++n_i^-)\\
&\geq (26+1000\log (\bar{s}-s'))m+\sum_{i\in [\bar{s}]\setminus [s']}(n_i^++n_i^-).
\end{align*}
Let $R'=R[I]$, and note that, for each $j\in V(R')$, by \eqref{eqn:Rbigenough} we have $d_{R'}^+(j)+d_{R'}^-(j)\geq |R'|-m= |I|-m$, so that, in combination with the above calculation,
\begin{equation*}
d^+_{R'}(v)+d^-_{R'}(v)\geq|R'|-m\geq(25+1000\log (\bar{s}-s'))m+\sum_{i\in [\bar{s}]\setminus [s']}(n_i^++n_i^-),
\end{equation*}
for each $v\in V(R')$. Therefore, by the inductive hypothesis for $\bar{s}-s'$, there is (with relabelling) some $s\in[\bar{s}]$ for which there exists $s'=i_1<i_2<\ldots<i_{s}=\bar{s}$ and subsets $\{j_\ell\}, I_\ell^+, I_\ell^-\subseteq V(R')=N_{R}^+(j)\setminus I_1^+$ for $\ell\in[s]\setminus [1]$, all disjoint, such that $j_{\ell_1}\rightarrow_{R}j_{\ell_2}$ whenever $\ell_1<\ell_2$, and, for each $\ell\in[s]$ and $\diamond\in\{+,-\}$, we have $|I_\ell^\diamond|=\sum_{i=i_{\ell-1}+1}^{i_\ell}n^\diamond_i$. Thus, the required properties are satisfied, completing the proof.
\end{proof}

\begin{corollary}\label{cor:caterpillar} Let $1/n\ll \eps,\eta, 1/r\ll\alpha\leq1$. Suppose $T$ is an $n$-vertex oriented tree with a subtree $T_0\subseteq T$, such that $|T_0|\leqslant\eta n$ and $T$ is formed from $T_0$ by attaching to each vertex $v$ of $T_0$ a tree $S_v^+$ in which $v$ only has out-neighbours and a tree $S^-_v$ in which $v$ only has in-neighbours, so that $|S^+_v|,|S^-_v|\leqslant\eta n$.
Let $R$ be a $\eps$-almost tournament with vertex set $[r]$.

Then, there is some $s\leqslant\alpha/100\epsilon$ for which there exists a partition $V(T_0)=X_1\cup\ldots\cup X_{s}$ and subsets $\{j_\ell\}, I_\ell^+,I_\ell^-\subseteq[r]$ for $\ell\in[s]$, all disjoint, with the following properties.
\stepcounter{propcounter}
\begin{enumerate}[label =\emph{\textbf{\Alph{propcounter}\arabic{enumi}}}]
\item\label{caterpillar11} $j_{\ell_1}\rightarrow_Rj_{\ell_2}$ whenever $\ell_1<\ell_2$.
\item\label{caterpillar01} There are no edges of $T_0$ directed from $X_i$ to $X_j$ with $i,j\in [s]$ and $i>j$.
\item\label{caterpillar21} For each $\ell\in[s]$ and $\diamond\in\{+,-\}$, we have $I_\ell^\diamond\subseteq N_R^\diamond(j_\ell)$, and
\begin{equation*}
    |I_\ell^\diamond|\geq \frac{r}{(1+\alpha/4)n}\cdot \sum_{v\in X_\ell}|S_v^\diamond|.
\end{equation*}
\end{enumerate}
\end{corollary}

\begin{proof} Pick $c\geq2\eta$ such that $\epsilon,1/r\ll c\ll \alpha$. Let $\bar{m}=cn$. Let $n_0=|T_0|$ and let $v_1,\ldots,v_{n_0}$ order $V(T_0)$ such that $i<j$ whenever $v_i\rightarrow_T v_j$. Let $\bar{s}$ be the largest integer for which there are integers $0=k_0<k_1<\ldots<k_{\bar{s}}\leqslant n_0$ such that $\bar{m}\leq \sum_{k=k_{\ell-1}+1}^{k_{\ell}}(|S_{v_k}^+|+|S_{v_k}^-|)\leq 2\bar{m}$ for each $\ell\in [\bar{s}]$. Now, as $|S_{v_k}^+|+|S_{v_k}^-|\leq\bar{m}$ for each $k\in [n_0]$, we must have by this maximality that $\sum_{k=k_{\bar{s}}+1}^{n_0}|S_{v_k}^+|+|S_{v_k}^-|<\bar{m}$, and therefore, as $T$ has $n$ vertices, we have that $\bar{s}\geq n/3\bar{m}=1/3c\geq 1$. Furthermore, setting $W_\ell=\{v_{k_{\ell-1}+1},\ldots,v_{k_\ell}\}$ for each $\ell\in [\bar{s}-1]$ and $W_{\bar{s}}=\{v_{k_{\bar{s}-1}+1},\ldots,v_{n_0}\}$, we have, for each $\ell\in [\bar{s}]$, that
\[
\bar{m}\leq \sum_{v\in W_\ell}(|S_v^+|+|S_v^-|)\leq 3\bar{m}.
\]
Finally, note that
\begin{equation}\label{eqn:sbound}
\frac{n}{3\bar{m}}\leq \bar{s}\leq \frac{2n}{\bar{m}}.
\end{equation}

Now, for each $i\in [\bar{s}]$, let
\[
n_i^\diamond=\left\lceil \frac{r}{n(1+\alpha/4)}\sum_{v\in W_i}|S_v^\diamond| \right\rceil.
\]
Let $m=r\bar{m}/{n}(1+\alpha/4)$, so that $cr/2\leq m\leq cr$ and, for each $i\in [\bar{s}]$, $m\leq n_i^++n_i^-\leq 4m$. From~\eqref{eqn:sbound}, we have $r/4m\leq \bar{s}\leq 2r/m$. Therefore, as $\bar{s}\geq 1/3c$ and $1/r\ll c\ll \alpha$, we have
\begin{equation}\label{eqn:logss}
2\bar{s}+(26+1000\log \bar{s})m\leq \frac{4r}{m}+\left(\frac{10^5\log \bar{s}}{\bar{s}}\right)r\leq \frac{8}{c}+\frac{\alpha r}{16}\leq \frac{\alpha r}{8}.
\end{equation}

Note that
\[
\sum_{i\in [\bar{s}]}(n_i^++n_i^-)\leq 2\bar{s}+\frac{r}{(1+\alpha/4)n}\cdot \sum_{i\in [\bar{s}]}\sum_{v\in W_i}(|S_v^+|+|S_v^-|)\leq 2\bar{s}+\frac{r(1+\eta)}{(1+\alpha/4)}\leq 2\bar{s}+(1-\alpha/8)r,
\]
as $\eta,1/r\ll \alpha$, so that, by \eqref{eqn:logss}, we have
\[
r\geq (26+1000\log{\bar{s}})m+\sum_{i\in [\bar{s}]}(n^+_i+n_i^-).
\]
Finally, we have $m\geq cr/2\geq \eps r$, so that, as $R$ is an $\eps$-almost tournament, for each $v\in V(R)$, we have $d^+_R(v)+d^-_R(v)\geq |R|-\eps|R|\geq |R|-m$.

Thus, by Lemma~\ref{lm:caterpillar_updated}, there is some $s\in [\bar{s}]$ for which there exists $0=i_0<i_1<\ldots<i_{s-1}<i_{s}=\bar{s}$, and subsets $\{j_\ell\}, I_\ell^+,I_\ell^-\subseteq[r]$ for $\ell\in[s]$, all disjoint, such that \itref{caterpillar1} and \itref{caterpillar2} hold. Letting $X_\ell=\cup_{i=j_{\ell-1}}^{j_\ell}W_i$ for each $\ell\in [s]$ then gives the required partition.
\end{proof}


\subsection{Embedding vertices for Theorem~\ref{thm:n+an-partial}}\label{sect:n+an-embedding}

Before proving Theorem~\ref{thm:n+an-partial}, we require one further lemma, to be used repeatedly in step~\ref{s-embedding} of the embedding procedure described previously. Lemma~\ref{lm:bounded_degree_partial_1} asserts that vertices of $T$ allocated to a single star segment of the `caterpillar-like' digraph can be embedded into the corresponding regularity clusters. More precisely, in the context of the earlier sketch, it shows that, for $Y\subseteq X_\ell$, it is possible to embed $T[Y]$ into $G[V_{j_\ell}]$, with components of $T-T_0$ attached to $Y$ by an out-edge into $G[\cup_{i\in I_\ell^+} V_i]$ and components of $T-T_0$ attached to $Y$ by an in-edge into $G[\cup_{i\in I_\ell^-}V_i]$, provided enough available space remains in the relevant clusters (where in the statement of Lemma~\ref{lm:bounded_degree_partial_1} we identify $T[Y]$ with $T_0$, $V_{j_\ell}$ with $V_0$, $\cup_{i\in I_\ell^+}V_i$ with $V_1^+,\ldots,V_k^+$, and $\cup_{i\in I_\ell^-}V_i$ with $V_1^-,\ldots,V_\ell^-$).

\begin{lemma}\label{lm:bounded_degree_partial_1}
Fix $\alpha\geqslant\beta>0$, $\mu>0$ and let $1/m\ll\eta\ll1/r\ll\epsilon\ll\gamma\ll \mu,\beta$. Let $G$ be a tournament. Suppose, for some $k,\ell\leqslant r$, there are disjoint subsets $V_0,V_1^+,\ldots,V_k^+,V_1^-,\ldots,V_\ell^-$ of $V(G)$, all of size $(1+\alpha)m$, such that $(V_0,V_i^+)$ is an $\epsilon$-regular pair of density at least $\mu$ for $i\in[k]$, and $(V_i^-,V_0)$ is an $\epsilon$-regular pair of density at least $\mu$ for $i\in[\ell]$.

Suppose $T$ is an oriented tree with a subtree $T_0\subseteq T$, such that $|T_0|\leqslant\eta m$, and $T$ is formed from $T_0$ by attaching to each vertex $v$ of $T_0$ trees $S_v^+$, $S_v^-$ with $d^-_{S_v^+}(v)=0$, $d^+_{S_v^-}(v)=0$, and $|S_v^+|,|S_v^-|\leqslant\eta m$.

Let $W\subseteq V_0$ be a set with $|W|\geqslant\gamma m$, and let $U_i^+\subseteq V_i^+$, $i\in[k]$, and $U_i^-\subseteq V_i^-$, $i\in[\ell]$ be sets such that $\sum_{i\in[k]}|U_i^+|\geqslant\sum_{v\in V(T_0)}|S_v^+|+k\beta m$ and $\sum_{i\in[\ell]}|U_i^-|\geqslant\sum_{v\in V(T_0)}|S_v^-|+\ell\beta m$.

Then, there is a copy of $T$ in $G$, with $T_0$ copied to $W$ and $T-V(T_0)$ copied to $U_1^+\cup\ldots\cup U_k^+\cup U_1^-\cup\ldots\cup U_\ell^-$.

\end{lemma}

\begin{figure}
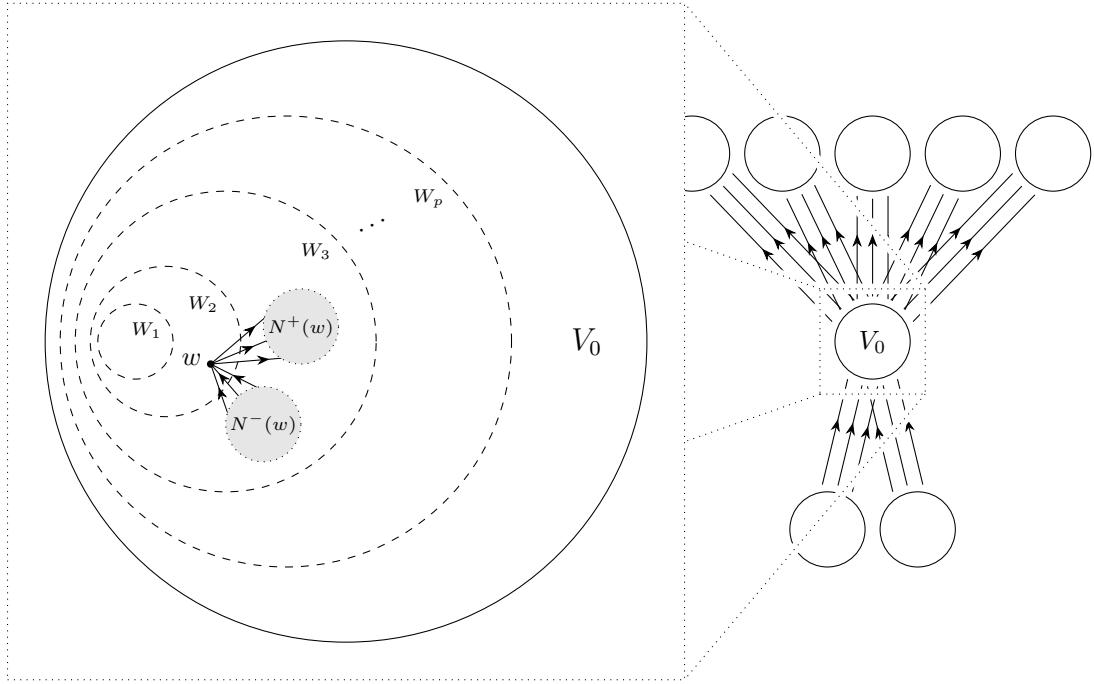

\DiagramCaterpillarEmbedding
\vspace{-0.75cm}
\caption{The sets $W_1\subseteq\ldots\subseteq W_p$ in the proof of Lemma~\ref{lm:bounded_degree_partial_1}. The sets are chosen so that each vertex $w\in W_j$ has sufficiently many in- and out-neighbours in $W_{j+1}$.}\label{fig:caterpillar-embedding}
\end{figure}

\begin{proof}
For the smallest possible $p$, take a partition $V(T_0)=X_1\cup\ldots\cup X_p$ such that, for each $j\in[p]$, $T_0[X_1\cup\ldots\cup X_j]$ is a tree, $\sum_{v\in X_j}|S_v^+|\leqslant k\beta \mu m/4$, and $\sum_{v\in X_j}|S_v^-|\leqslant\ell\beta \mu m/4$. This is possible for $p=|T_0|$, so a smallest such $p$ will exist. We in fact claim that $p\leqslant32/\beta\mu$. Indeed, for this smallest possible $p$, take a partition that minimises $\sum_{j\in[p]}j|X_j|$. Suppose there is some $j'<p$ for which both $\sum_{v\in X_{j'}}|S_v^+|\leqslant k\beta\mu m/8$ and $\sum_{v\in X_{j'}}|S_v^-|\leqslant\ell\beta\mu m/8$. Let $x\in X_{j'+1}$ be such that $T_0[X_1\cup\ldots\cup X_{j'}\cup\{x\}]$ is a tree. Then moving $x$ from $X_{{j'}+1}$ to $X_{j'}$ produces a partition which contradicts the minimality of $\sum_{j\in[p]}j|X_j|$. Thus we have
\begin{equation*}
    p\leqslant 1+\frac{\sum_{v\in V(T_0)}|S_v^+|}{(k\beta\mu m/8)}+\frac{\sum_{v\in V(T_0)}|S_v^-|}{(\ell\beta\mu m/8)}\leqslant1+\frac{8\sum_{i\in[k]}|U_i^+|}{k\beta \mu m}+\frac{8\sum_{i\in[\ell]}|U_i^-|}{\ell\beta\mu m}\leqslant32/\beta \mu.
\end{equation*}

Now find vertex sets $W_1\subseteq W_2\subseteq\ldots\subseteq W_p=W$ such that $|W_j|\geqslant|W_{j+1}|/8$ for each $j\in[p-1]$, and $d^\pm(w,W_{j+1})\geqslant|W_{j+1}|/8$ for each $j\in[p-1]$, $w\in W_j$.
This is possible by starting with $W_p$ and iteratively using that fact that at most $|W_{j+1}|/4$ vertices $w$ of $W_{j+1}$ have $d^+(w,W_{j+1})\leqslant|W_{j+1}|/8$, and at most $|W_{j+1}|/4$ vertices $w$ of $W_{j+1}$ have $d^-(w,W_{j+1})\leqslant|W_{j+1}|/8$.

We will now embed $T$ in $p$ stages as follows. At stage $j$, suppose we have already embedded $T[\cup_{j'=1}^{j-1}\cup_{v\in X_{j'}}(V(S_v^+)\cup V(S_v^-))]$. For each $v\in X_1\cup\ldots\cup X_{j-1}$ in turn, consider the forest $F_v^+$ consisting of trees of $T_0[X_j]$ attached to $v$ by out-neighbours of $v$, and suppose $v$ has already been copied to some $w\in W_{j-1}$ (for the case $j=1$, regard all of $T_0[X_j]$ as components attached to a single auxiliary vertex $v$ by out-neighbours, where $v$ has already been copied to an auxiliary vertex $w$ satisfying $W_1\subseteq N_G^+(w)$). Let $Z_{j,v}^+$ be the set of unoccupied out-neighbours of $w$ in $W_j$ which each have at least $3k\beta\mu m/4$ unoccupied out-neighbours in $\cup_{i\in[k]}U_i^+$ as well as $3\ell\beta\mu m/4$ unoccupied in-neighbours in $\cup_{i\in[\ell]}U_i^-$. Because there are always at least $k\beta m$ unoccupied vertices in $\cup_{i\in[k]}U_i^+$ and $\ell\beta m$ unoccupied vertices in $\cup_{i\in[\ell]}U_i^-$, Proposition~\ref{prop:outneighbourhood_to_many_clusters} implies $|Z_{j,v}^+|\geqslant|N^+(w,W_j)|-|T_0|-2\epsilon m\geqslant |W_1|/8-3\epsilon m\geqslant|W|/8^{p+1}\geqslant3\eta m$. Therefore, by Theorem~\ref{thm:any_linear_bound}, there is a copy of $F_v^+$ in $Z_{j,v}^+$. Then, for each $v'\in V(F_v^+)$, if $v'$ has now been copied to $w'$, find a copy of $S_{v'}^+-v'$ in the unoccupied vertices of $N^+(w',\cup_{i\in[k]}U_i^+)$. Because $w'\in Z_{j,v}^+$, and only at most $\sum_{v''\in X_j,v''\neq v'}|S_{v''}^+|$ additional vertices of $N^+(w',\cup_{i\in[k]}U_i^+)$ may become occupied since choosing $Z_{j,v}^+$, at least $3k\beta \mu m/4-\sum_{v''\in X_j,v''\neq v'}|S_{v''}^+|\geqslant 3|S_{v'}^+|$ vertices of $N^+(w',\cup_{i\in[k]}U_i^+)$ remain unoccupied, allowing the copy of $S_{v'}^+-v'$ to be found using Theorem~\ref{thm:any_linear_bound}. Similarly, find a copy of $S_{v'}^--v'$ in the unoccupied vertices of $N^-(w',\cup_{i\in[\ell]}U_i^-)$. We then do the same for the forest $F_v^-$ consisting of trees of $T_0[X_j]$ attached to $X_1\cup\ldots\cup X_{j-1}$ by in-neighbours. Performing this process for each $v\in X_1\cup\ldots\cup X_{j-1}$ completes stage $j$ of the embedding procedure. Upon the completion of stage $p$, we obtain a copy of $T$ in $G$, with $T_0$ copied to $W$ and $T-|T_0|$ copied to $U_1^+\cup\ldots\cup U_k^+\cup U_1^-\cup\ldots\cup U_\ell^-$.
\end{proof}

We now combine Corollary~\ref{cor:caterpillar} and Lemma~\ref{lm:bounded_degree_partial_1} to prove Theorem~\ref{thm:n+an-partial}, indicating throughout where the proof aligns with the procedure sketched at the beginning of this section.

\setcounter{embedstep}{0}
\begin{proof}[Proof of Theorem~\ref{thm:n+an-partial}]
Set $\beta=\alpha/3$, $\mu=1/2$, and introduce constants $\epsilon, r_1, r_2$ such that $\eta\ll1/r_2\ll1/r_1\ll\epsilon\ll\beta$. Let $G$ be a $(1+\alpha)n$-vertex tournament.

\embedstep By Corollary~\ref{cor:regular_subtournament}, there is a subtournament $G'\subseteq G$ with $|G'|\geqslant(1+2\beta)n$, and an $\epsilon$-regular partition $V(G')=V_1\cup\ldots\cup V_r$ with $r_1\leqslant r\leqslant r_2$. Let $R$ be a $\sqrt{\epsilon}$-almost tournament with vertex set $[r]$, such that $(V_i,V_j)$ is an $\epsilon$-regular pair of density at least $\mu$ whenever $i\rightarrow_Rj$.

\embedstep For each $v\in V(T_0)$, let $S_v^+\subseteq S_v$ be the subtree of $S_v$ induced by the vertices whose path from $v$ begins with an out-edge, and let $S_v^-\subseteq S_v$ be the subtree of $S_v$ induced by the vertices whose path from $v$ begins with an in-edge. We remark that $S_v^+$ and $S_v^-$ both contain the vertex $v$, and that $|S_v^+|,|S_v^-|\leqslant\eta n$ for every $v\in V(T_0)$. By Corollary~\ref{cor:caterpillar}, there is some $s\leqslant\alpha/100\epsilon$ for which there exists a partition $V(T_0)=X_1\cup\ldots\cup X_s$ and subsets $\{j_\ell\},I_\ell^+,I_\ell^-\subseteq[r]$ for $\ell\in[s]$, all disjoint, satisfying properties \itref{caterpillar11}-\itref{caterpillar21}. In particular, for each $\ell\in[s]$ and $\diamond\in\{+,-\}$, we have
\begin{equation}
    \sum_{v\in X_\ell}|S_{v}^\diamond|+|I_\ell^\diamond|\beta\cdot n/r\overset{\itref{caterpillar21}}{\leqslant}|I_\ell^\diamond|(1+2\beta)\cdot n/r=\sum_{j\in I_\ell^\diamond}|V_j| \label{eq:caterpillar_embedding_bound}
\end{equation}

Set $\gamma=\beta\mu/8r$. Using \itref{caterpillar11}, Proposition~\ref{prop:outneighbourhood_to_many_clusters}, and $s\leqslant\alpha/100\epsilon$, for each $\ell\in[s]$ at most $s\epsilon(1+\alpha)\cdot n/r\leqslant 2\gamma n$ vertices $w$ of $V_{j_\ell}$ have either some $\ell'>\ell$ for which $d^+(w,V_{j_{\ell'}})\leqslant 4\gamma n$, or some $\ell'<\ell$ for which $d^-(w,V_{j_{\ell'}})\leqslant 4\gamma n$. Therefore, we may take subsets $W_{j_\ell}\subseteq V_{j_\ell}$ for $\ell\in[s]$ such that, $d^+(w,W_{j_{i}})\geqslant2\gamma n$ whenever $i>\ell$ and $w\in W_{j_\ell}$, and $d^-(w,W_{j_{i}})\geqslant2\gamma n$ whenever $i<\ell$ and $w\in W_{j_\ell}$.

\embedstep Now let $V(T_0)=Y_1\cup\ldots\cup Y_\tau$ be any partition such that
\stepcounter{propcounter}
\begin{enumerate}[label =\textbf{\Alph{propcounter}\arabic{enumi}}]
\item For each $t\in [\tau]$, $T[Y_t]$ is a connected component of $T_0[X_\ell]$ for some $\ell\in[s]$.
\item For each $t\in [\tau]$, $T_0[Y_1\cup \ldots\cup Y_t]$ is a tree.\label{Y3}
\end{enumerate}

\embedstep We will now embed $T$ into $G$ so that $X_\ell$ is copied to $W_{j_\ell}$ for each $\ell\in[s]$, and $\cup_{v\in X_\ell}V(S_{v}^\diamond)$ is copied to $\cup_{j\in I_\ell^\diamond}V_j$ for each $\ell\in[s]$ and $\diamond\in\{+,-\}$. The embedding is given in $\tau$ stages as follows. Let $Q_0$ be the empty graph. Suppose after stage $t-1$, we have embedded $T[\cup_{v\in Y_1\cup\ldots\cup Y_{t-1}}V(S_v)]$ to get $Q_{t-1}$. Let $\ell\in[s]$ be such that $Y_t\subseteq X_\ell$. If $t=1$, set $A_t=W_{j_\ell}$. Otherwise, if $t>1$, let $y_t$ be the unique vertex of $Y_1\cup\ldots\cup Y_{t-1}$ with a neighbour in $Y_t$, let $\diamond\in\{+,-\}$ be such that the neighbour in $Y_t$ is a $\diamond$-neighbour, let $z_t$ be the image of $y_t$ in $Q_{t-1}$, and set $A_t=N^\diamond(z_t,W_{j_\ell})\setminus V(Q_{t-1})$. Note that in both cases we find $|A_t|\geqslant\gamma n$. Also, we find for $\diamond\in\{+,-\}$ that
\[
\sum_{j\in I_\ell^\diamond}|V_j\setminus V(Q_{t-1})|\overset{\eqref{eq:caterpillar_embedding_bound}}{\geqslant}\sum_{v\in Y_t}|S_v^\diamond|+|I_\ell^\diamond|\beta n/r.
\]
Therefore, by~\itref{caterpillar21} and Lemma~\ref{lm:bounded_degree_partial_1}, there is a copy of $T[Y_t\cup(\cup_{v\in Y_t}S_v)]$ in $G$ with $Y_t$ copied to $A_t\subseteq W_{j_\ell}$, and $(\cup_{v\in Y_t}V(S_v^\diamond))\setminus Y_t$ copied to $(\cup_{j\in I_\ell^\diamond}V_j)\setminus V(Q_{t-1})$ for $\diamond\in\{+,-\}$. Thus we obtain a copy of $T$ after stage $\tau$.
\end{proof}

\section{Proof of Theorem~\ref{thm:n+k+an} and Theorem~\ref{thm:n+an}}\label{sect:final-proofs}
Recall the decomposition of our tree $T$ from Section~\ref{sect:outline} as $T_0\subseteq T_1\subseteq T_2\subseteq T_3\subseteq T_4=T$. In Sections~\ref{sect:T_1-unbounded} and \ref{sect:T_1-bounded} respectively, we showed how to embed $T_0$ and extend this to $T_1$ for both Theorems~\ref{thm:n+k+an} and~\ref{thm:n+an}. In this section, we will show how a copy of $T_1$ can be extended to a copy of $T$, completing the proof of both theorems. As noted in the proof outline, the main challenge here is to embed the vertices in $V(T_3)\setminus V(T_2)$, where these vertices form paths with constant length between vertices in $T_2$. Indeed, firstly, $T_2-V(T_1)$ is a forest of constant-sized components not directly connected to $T_1$ (see \itref{tree3}), which can be embedded greedily using, for example, Theorem~\ref{thm:any_linear_bound}. Secondly, to reach $T_4$ from $T_3$ we add small tree components on to $T_3$, which is already connected. This can be done by reserving a small random subset of vertices $U$ (using Proposition~\ref{prop:random-subset}) and carrying out the rest of the embedding in the vertices with sufficient out- and in-degree to $U$. Such an embedding can then be completed greedily, giving an embedding of $T_4=T$.

Thus, most of this section will be dedicated to showing how we can extend a copy of $T_2$ to a copy of $T_3$ (using a method effective for both Theorems~\ref{thm:n+k+an} and \ref{thm:n+an}). Recall that $T_3$ is obtained from $T_2$ by attaching paths of fixed length by their endpoints (see \itref{tree4}), but such that the total number of vertices contained in such paths is only a small proportion of the resulting tree (see \itref{tree5}). Thus, with a copy of $T_2$ already found, we will often wish to find paths of a fixed length between certain attachment points. By ensuring these attachment points have plenty of out- and in-neighbours, we need only to be able to connect linear-sized sets with paths of fixed but small length, while avoiding some small set of vertices already used in some paths. As we will see, paths with changes of direction are comparatively easy to find, so we only consider whether we can find such paths so that they are directed paths. We will call tournaments with this connection property \emph{well-connected}, as follows.

\begin{defn}\label{defn:connector}
We say a tournament $G$ is \emph{$(a,b,\ell)$-well-connected} if, for every $A_1,A_2\subseteq V(G)$ with $|A_1|,|A_2|\geqslant a$ and $B\subseteq V(G)$ with $|B|\leqslant b$, there is a directed path in $V(G)\setminus B$ from $A_1$ to $A_2$ with length $\ell$.
\end{defn}

In Lemma~\ref{lm:main-theorems-connector}, we will see that both of our main theorems hold if the tournament $G$ is well-connected. Of course, not every tournament is well-connected, but, in Lemma~\ref{lm:not-connector-split} we will see that any tournament that is not well-connected contains a bipartition of most of its vertices, so that all the relevant edges are directed in the same direction across the bipartition. Through the repeated application of Lemma~\ref{lm:not-connector-split}, we can then decompose the vertices of any tournament as $V(G)=B\cup W_1\cup\ldots\cup W_r$, so that $B$ is small, all possible edges are directed from $W_i$ to $W_j$ for $1\leqslant i<j\leqslant r$, and each $G[W_i]$ is either small or well-connected (see Lemma~\ref{lm:transitive-decomposition}). We then assign the vertices of $T$ to the sets $W_1,\ldots,W_r$, so that any edge of $T$ assigned between some $W_i$ and $W_j$ with $i<j$ is to be embedded as directed from $W_i$ into $W_j$. Thus, we can embed the vertices of $T$ assigned to $W_i$ into $G[W_i]$ independently for each $i\in [r]$, while knowing the other edges of $T$ can then be embedded. As noted in the proof sketch, this is a streamlined version of techniques by K\"uhn, Mycroft and Osthus~\cite{KUE-MYC-OST,KUE-MYC-OST-2}. In~\cite{KUE-MYC-OST,KUE-MYC-OST-2}, a notion of robust out-expansion is used, from which our well-connected property can be derived. As we do not need any other results of robust out-expansion (most notably, we do not use a Hamilton, or almost-spanning, cycle in the reduced digraph), we use the well-connected property directly. This allows the decomposition of~\cite[Lemma~5.2]{KUE-MYC-OST-2} to be simplified to find bipartitions with all the edges directed from one side to another, rather than just most of the edges.

In Section~\ref{sect:connectors}, we will prove a number of results on well-connected tournaments, including the tournament decomposition discussed above. Then, in Section~\ref{sect:final-final-proofs}, after showing our main results hold for well-connected tournaments (i.e., Lemma~\ref{lm:main-theorems-connector}), we prove both Theorems~\ref{thm:n+k+an} and~\ref{thm:n+an}.

\subsection{Well-connected tournaments}\label{sect:connectors}

We start by proving two simple properties of well-connected tournaments in Lemma~\ref{lm:connector-properties}. The first is that removing a small number of vertices from a well-connected tournament maintains some (potentially slightly weaker) connection property. The second shows that $(a,b,\ell)$-well-connected tournaments robustly contain paths of length $\ell$, regardless of the desired orientation of the paths' edges. While Definition~\ref{defn:connector} only refers to directed paths, Thomason \cite{THO} showed that a path with at least one change of direction can be found between two sufficiently large subsets of any tournament, covering all other cases.

\begin{lemma}\label{lm:connector-properties}
Let $a,b,\ell\geqslant 0$, and suppose $G$ is a $(a,b,\ell)$-well-connected tournament.
\begin{enumerate}[label = \textbf{\roman{enumi})}]
    \item\label{connector1} If $C\subseteq V(G)$ has size $c\leqslant b$, then $G-C$ is $(a,b-c,\ell)$-well-connected.
    \item\label{connector2} Suppose $P$ is an oriented path of length $\ell$, and $A_1,A_2,B\subseteq V(G)$ satisfy $|A_1|,|A_2|\geqslant a$, $|B|\leqslant b$. If $a\geqslant b+\ell+3$, then there is a copy of $P$ in $G-B$, with its first vertex in $A_1$ and its last vertex in $A_2$.
\end{enumerate}
\end{lemma}

The second property uses the fact that an oriented path with at least one direction can be found between two sufficiently large subsets of any tournament, as summarised by Lemma~\ref{lm:non-directed-path}. This follows from a combination of several preliminary lemmas used in Thomason's proof that every $(n+1)$ vertex tournament contains a copy of every $n$-vertex oriented path~\cite{THO}. For a formal treatment of how Lemma~\ref{lm:non-directed-path} follows from Thomason's results, we refer the reader to \cite[Corollary~1.14]{BEN}.

\begin{lemma}\label{lm:non-directed-path}
    Let $P$ be an oriented path of length $\ell$ which is not a directed path. Let $G$ be a tournament, and let $X$ and $Y$ be subsets of $V(G)$ of order at least $\ell+3$. Then, there is a copy of $P$ in $G$ with first vertex in $X$ and last vertex in $Y$.
\end{lemma}

\begin{proof}[Proof of Lemma~\ref{lm:connector-properties}]
First, fix a subset $C\subseteq V(G)$ with size $c$. Then, if $A_1,A_2,B\subseteq V(G)$ satisfy $|A_1|,|A_2|\geqslant a$ and $|B|\leqslant b-c$, then, because $G$ is $(a,b,\ell)$-well-connected, there is a directed path in $V(G)\setminus(B\cup C)$ from $A_1$ to $A_2$ with length $\ell$. Therefore, $G-C$ is $(a,b-c,\ell)$-well-connected and \ref{connector1} holds.

Next, suppose $P$ is an oriented path of length $\ell$, and $A_1,A_2,B\subseteq V(G)$ satisfy $|A_1|,|A_2|\geqslant a$, $|B|\leqslant b$. If $P$ is a directed path, then, because $G$ is $(a,b,\ell)$-well-connected there is a copy of $P$ in $G-B$ with first vertex in $A_1$  and last vertex in $A_2$. On the other hand, if $P$ has a change of direction, then by Lemma~\ref{lm:non-directed-path} there is a copy of $P$ in $G[(A_1\cup A_2)\setminus B]$, with first vertex in $A_1$ and last vertex in $A_2$. Therefore, \ref{connector2} holds.
\end{proof}

We will need to set aside a random subset of vertices to use to attach paths to $T_2$ to obtain $T_3$. We need therefore to show  that random subsets of well-connected tournaments can be used to find connecting paths of this sort. To do this, we use \emph{median orders}. Given a tournament $G$, an ordering $v_1,\ldots,v_n$ of $V(G)$ is a median order if it maximises the number of pairs $i<j$ with $v_iv_j\in E(G)$. We use median orders only to recall the following lemma from~\cite{BEN-MON} and apply it to prove Lemma~\ref{lm:random-connector}. For further useful properties of median orders applicable to embedding trees in tournaments see~\cite{HAV-THO,DRO-HAV,BEN-MON}.

\begin{lemma}[{\cite[Lemma~2.9]{BEN-MON}}]\label{lm:connecting-path-length-3}
Suppose $G$ is a tournament with a median order $v_1,\ldots,v_n$. Then, for any $1\leqslant i<j\leqslant n$ with $j-i\geqslant7$, and $A\subset V(G)\setminus\{v_i,v_j\}$ with $|A|\leqslant (j-i-7)/6$, there is a directed $v_i,v_j$-path in $G-A$ with length~3.
\end{lemma}

We are now ready to state and prove our lemma, showing that, with high probability, a random subset of vertices in a well-connected tournament induces a well-connected tournament, as follows.

\begin{lemma}\label{lm:random-connector}
Let $1/n\ll\eta\ll1/\ell,\epsilon\ll p$. Suppose $G$ is a $(\epsilon n,\eta n,\ell)$-well-connected tournament with $|G|\leqslant 3n$, and that $U\subseteq V(G)$ is a random subset with vertices included independently with probability $p$. Then, with high probability, $G[U]$ is $(6\epsilon n,\eta^2 n,\ell+6)$-well-connected.
\end{lemma}
\begin{proof}
Let $v_1,\ldots,v_m$ be a median order for $G$. Let $W_1$ and $W_2$ respectively denote the first and last $\epsilon n $ vertices of the median order. Let $V'$ be the middle $m-4\epsilon n$ vertices of the median order.

It is enough to show that, with high probability, for every $v,w\in V'$, there are at least $2\eta^2 n$ internally vertex-disjoint directed $v,w$-paths with length $\ell+6$ and with all internal vertices in $U$. Indeed, then for any $A_1,A_2\subseteq U$ with $|A_1|,|A_2|\geqslant6\epsilon n$ and $B\subseteq U$ with $|B|\leqslant\eta^2 n$, there is some $v\in (V'\cap A_1)\setminus B$ and $w\in (V'\cap A_2)\setminus(B\cup\{v\})$, and hence at least $2\eta^2n$ internally vertex-disjoint directed paths from $A_1\setminus B$ to $A_2\setminus B$ in $G[U]$ with length $\ell+6$. Of these paths, at most $\eta^2n$ contain some internal vertex in $B$, and so there is some directed path in $U\setminus B$ from $A_1$ to $A_2$ of length $\ell+6$, thus demonstrating $G[U]$ is $(6\epsilon n,\eta^2 n,\ell+6)$-well-connected.

Fix $v,w\in V'$. Because $G$ is $(\epsilon n,\eta n,\ell)$-well-connected, and $|W_1|,|W_2|\geqslant\epsilon n$, we can greedily find at least $\eta n/2\ell$ vertex-disjoint directed paths in $V(G)\setminus\{v,w\}$ from $W_2$ to $W_1$ with length $\ell$. Using Lemma~\ref{lm:connecting-path-length-3}, we can greedily and disjointly connect $v$ to the first vertex of each path by a directed path of length 3, while avoiding all other vertices used so far. Indeed, at least $\epsilon n$ vertices in the median order lie between $v$ and the first vertex of each path, while the total number of vertices to be avoided each time is at most $\eta n\leqslant(\epsilon n-7)/6$. Similarly, we can also disjointly connect the last vertex of each path to $w$ by a directed path of length 3, also avoiding any vertex used previously. Therefore, we have at least $\eta n/2\ell$ internally vertex-disjoint directed paths in $V(G)$ from $v$ to $w$, each with length $\ell+6$.

Let $X_{v,w}$ be the number of these directed $v,w$-paths which additionally have all internal vertices in $U$, and note that $X_{v,w}$ is a binomial variable with $\E X_{v,w}\geqslant p^{\ell+5}\eta n/2\ell>3\eta^2 n$. From Lemma~\ref{lm:chernoff}, we have
\begin{equation*}
    \mathbb{P}(X_{v,w}\leqslant 2\eta^2n)\leqslant\mathbb{P}(|X_{v,w}-\E X_{v,w}|\geqslant\E X_{v,w}/3)\leqslant 2\exp{(-\E X_{v,w}/27)}\leqslant2\exp{(-\eta^2 n/9)}.
\end{equation*}
Thus, the probability that the desired property fails is at most $18n^2\exp{(-\eta^2n/9)}$, and so, as $1/n\ll\eta$, the conclusion of the lemma holds with high probability.
\end{proof}

Next we will show that, if a tournament is not well-connected, then, except for a small subset of vertices, we may partition the vertices in two so that all the edges between the parts are directed into the same part.

\begin{lemma}\label{lm:not-connector-split}
Let $\epsilon >0$, $\ell\in\mathbb{N}$, and $1/n\ll\eta\ll\epsilon,1/\ell$. Suppose $G$ is a tournament with $|G|\leqslant 3n$ that is not $(\epsilon n,\eta n,\ell)$-well-connected. Then, there is a partition $V(G)=W_1\cup W_2\cup B$ so that $|W_1|,|W_2|\geqslant \epsilon n/2$, $|B|\leqslant4\ell^{-1}n$, and $x\rightarrow y$ for every $x\in W_1$, $y\in W_2$.
\end{lemma}
\begin{proof}
Using that $G$ is not $(\epsilon n, \eta n, \ell)$-well-connected, let $A_1,A_2,B_0\subseteq V(G)$ be sets such that $|A_1|,|A_2|\geqslant\epsilon n$, $|B_0|\leqslant\eta n$, and there is no directed path in $V(G)\setminus B_0$ from $A_1$ to $A_2$ with length $\ell$. Construct a chain of subsets $U_0\subseteq U_1\subseteq\ldots\subseteq U_\ell$ as follows. Let $U_0=A_1\setminus B_0$, and, for $i\in[\ell]$, let $U_i$ be the set of vertices $x\in V(G)$ such that there is a directed path from $A_1$ to $x$ in $V(G)\setminus B_0$ with length at most $i$. We remark that $|(U_r\setminus U_{r-1})\cap A_2|\leqslant\ell+1$ for any $r\leqslant\ell$, else, by taking a directed path of length $\ell-r$ in $(U_r\setminus U_{r-1})\cap A_2$ together with a path of length $r$ from $A_1$ to that path's initial vertex, we would be able to find a directed path in $V(G)\setminus B_0$ from $A_1$ to $A_2$ of length $\ell$. In particular, we have $|U_r\cap A_2|\leqslant(\ell+1)^2$ for any $r\leqslant\ell$.

Let $r\in[\ell]$ be minimal such that $|U_r\setminus U_{r-1}|\leqslant3\ell^{-1}n$. Set $W_2=U_{r-1}$, $B=B_0\cup(U_r\setminus U_{r-1})$, and $W_1=V(G)\setminus (U_r\cup B_0)$, so that $W_1\cup W_2\cup B$ is a partition of $V(G)$. Because $A_1\setminus B_0\subseteq W_2$, we have $|W_2|\geqslant\epsilon n/2$. Because $A_2\setminus(U_r\cup B_0)\subseteq W_1$ and $1/n\ll\eta\ll\epsilon,1/\ell$, we have $|W_1|\geqslant \epsilon n-(\ell+1)^2-\eta n\geqslant \epsilon n/2$. From the choice of $r$, we have $|B|\leqslant3\ell^{-1}n+\eta n\leqslant 4\ell^{-1}n$. Finally, the fact that $x\rightarrow y$ for every $x\in W_1$, $y\in W_2$ follows from the definition of $U_r$ and $U_{r-1}$.
\end{proof}

Using a repeated application of Lemma~\ref{lm:not-connector-split}, we are now ready to state and prove the tournament decomposition referred to at the start of this section.

\begin{lemma}\label{lm:transitive-decomposition}
Suppose $\eta\ll\epsilon$ and let $\ell=\lceil\epsilon^{-3}\rceil$. Suppose $G$ is a tournament with $|G|\leqslant 3n$. Then, there is a partition $V(G)=B\cup W_1\cup\ldots\cup W_r$ so that $|B|\leqslant\epsilon n$ and the following properties hold.
\stepcounter{propcounter}
\begin{enumerate}[label=\emph{\textbf{\Alph{propcounter}\arabic{enumi}}}]
    \item\label{td1} If $1\leqslant i<j\leqslant r$ and $x\in W_i$, $y\in W_j$, then $x\rightarrow y$.
    \item\label{td2} For $i\in[r]$, if $|W_i|\geqslant\sqrt{\epsilon}n$, then $G[W_i]$ is $(\epsilon n,\eta n,\ell)$-well-connected.
\end{enumerate}
\end{lemma}
\begin{proof}
Initially, set $B^{(1)}=\emptyset$ and $W_1^{(1)}=V(G)$. Then, for $r\geqslant1$, do the following. We are given a partition $V(G)=B^{(r)}\cup W_1^{(r)}\cup\ldots\cup W_r^{(r)}$ with $|B^{(r)}|\leqslant 5r\epsilon^3n$, such that $|W_i^{(r)}|\geqslant\epsilon n/2$ for each $i\in[r]$, and, if $1\leqslant i<j\leqslant r$ and $x\in W_i^{(r)}$, $y\in W_j^{(r)}$, then $x\rightarrow y$. If we have that $G[W_i^{(r)}]$ is $(\epsilon n,\eta n,\ell)$ well-connected whenever $|W_i^{(r)}|\geqslant\sqrt{\epsilon}n$, then set $B=B^{(r)}$ and $W_i=W_i^{(r)}$ for $i\in[r]$. Otherwise, let $j\in[r]$ be such that $W_j^{(r)}$ is not $(\epsilon n,\eta n,\ell)$-well-connected, with $|W_j^{(r)}|$ maximal (so $|W_j^{(r)}|\geqslant\sqrt{\epsilon}n$). By Lemma~\ref{lm:not-connector-split}, there is a partition $W_j^{(r)}=U_1\cup U_2\cup B_r$ so that $|U_1|,|U_2|\geqslant \epsilon n/2$, $|B_r|\leqslant4\ell^{-1}n\leqslant5\epsilon^3n$, and $x\rightarrow y$ for every $x\in U_1$ and $y\in U_2$. We then set
\begin{align*}
    B^{(r+1)}&=B^{(r)}\cup B_r\\
    W_i^{(r+1)}&=
        \begin{cases}
        W_i^{(r)} & \text{if $1\leqslant i<j$} \\
        U_1 & \text{if $i=j$}\\
        U_2 & \text{if $i=j+1$}\\
        W_{i-1}^{(r)} & \text{ if $j+1<i\leqslant r+1$}
        \end{cases}
\end{align*}
We remark that $V(G)=B^{(r+1)}\cup W_1^{(r+1)}\cup\ldots\cup W_{r+1}^{(r+1)}$ is a partition with $|B^{(r+1)}|\leqslant 5(r+1)\epsilon^2n$, such that $|W_i^{(r+1)}|\geqslant\epsilon n/2$ for each $i\in[r]$, and, if $1\leqslant i<j\leqslant r+1$ and $x\in W_i^{(r+1)}$, $y\in W_j^{(r+1)}$, then $x\rightarrow y$, and so the procedure may continue.

On the $r$\textsuperscript{th} iteration of this procedure, the largest $|W_i^{(r)}|$ that is not $(\epsilon n,\eta n,\ell)$-well-connected has size at most $3n-(r-1)\cdot \epsilon n/2$, and so the procedure will terminate after at most $6\epsilon^{-1}$ iterations, at which point we find $|B^{(r)}|\leqslant30\epsilon^2n\leqslant\epsilon n$.
\end{proof}

\subsection{Proof of Theorem~\ref{thm:n+k+an} and Theorem~\ref{thm:n+an}}\label{sect:final-final-proofs}

We first prove that our two main theorems hold when the tournament is well-connected.

\newcommand{\othereta}{\hat{\eta}}

\begin{lemma}\label{lm:main-theorems-connector}
Suppose $1/n\ll\eta\ll\epsilon\ll\alpha$ and let $\ell=\lceil\epsilon^{-3}\rceil$. Suppose $G$ is a tournament which is $(\epsilon n,\eta n,\ell)$-well-connected, and that $T$ is an $n$-vertex oriented tree.
\begin{enumerate}[label =\emph{(\arabic{enumi})}]
    \item\label{main-lemma-1} Suppose that $|G|=((1+\alpha)n+k)$ where $k$ is the number of leaves of $T$. Then, $G$ contains a copy of $T$.
    \item\label{main-lemma-2} Suppose that $c$ is a constant such that $1/n\ll c\ll\eta$, that $|G|=(1+\alpha)n$, and that $\Delta(T)\leqslant cn$. Then, $G$ contains a copy of $T$.
\end{enumerate}
\end{lemma}

\begin{proof}
The proof for each statement of this theorem is nearly identical, so here we will present a proof for \itref{main-lemma-1}, and explain in the footnotes any places where the proof for \itref{main-lemma-2} differs.

Fix $\alpha>0$ and introduce a constant $m$ such that $1/n\ll1/m\ll\eta\ll\epsilon\ll\alpha$. Let $\othereta=(\eta/5)^4$ so that $\eta=5\othereta^{1/4}$. Fix an $n$-vertex $k$-leaf oriented tree $T$ and let $G$ be a $((1+\alpha)n+k)$-vertex tournament which is $(\epsilon n,5\othereta^{1/4}n,\ell)$-well-connected. We will show that $G$ contains a copy of $T$, thus proving \itref{main-lemma-1}.\footnote{For \itref{main-lemma-2}, fix $\alpha>0$ and introduce a constant $m$ such that $1/n\ll c\ll1/m\ll\othereta\ll\epsilon\ll\alpha$. Fix an $n$-vertex oriented tree $T$ with $\Delta(T)\leqslant cn$ and let $G$ be a $(1+\alpha)n$-vertex tournament which is $(\epsilon n,5\othereta^{1/4}n,\ell)$-well-connected. We will show that $G$ contains a copy of $T$, thus proving \itref{main-lemma-2}.}

Let $U_0\subseteq V(G)$ be a random subset, with elements from $V(G)$ chosen independently at random with probability $2\sqrt{\othereta}$, and let $W_0$ be the set of vertices $v$ in $V(G)\setminus U_0$ with $d^\pm(v,U_0)\geqslant4\othereta n$. By Proposition~\ref{prop:random-subset}, we have that $|V(G)\setminus W_0|\leqslant 24\sqrt{\othereta} n$ with high probability. $V(G)\setminus U_0$ may be regarded as a random subset of $V(G)$ with elements chosen independently at random with probability $1-2\sqrt{\othereta}$, and so, by Lemma~\ref{lm:random-connector}, we have that $G[V(G)\setminus U_0]$ is $(6\epsilon n,25\sqrt{\othereta}n,\ell+6)$-well-connected with high probability. Therefore, we may proceed assuming that $|V(G)\setminus W_0|\leqslant24\sqrt{\othereta}n$, and, using Lemma~\ref{lm:connector-properties}~i), that $G[W_0]$ is $(6\epsilon n,\sqrt{\othereta}n,\ell+6)$-well-connected.

Let $U_1\subseteq W_0$ be a random subset, with elements from $W_0$ chosen independently at random with probability $\alpha/36$, and let $W_1$ be the set of vertices $v$ in $W_0\setminus U_1$ with $d^\pm(v,U_1)\geqslant36\epsilon n$. By Proposition~\ref{prop:random-subset}, we have that $|W_1\setminus W_0|\leqslant \alpha n/3$ with high probability, and, by Lemma~\ref{lm:random-connector}, we have that $G[U_1]$ is $(36\epsilon n,\othereta n,\ell+12)$-well-connected with high probability. Therefore, we may proceed assuming that $|W_1|\geqslant((1+\alpha/2)n+k)$, and that $G[U_1]$ is $(36\epsilon n,\othereta n,\ell+12)$-well-connected. \footnote{For \itref{main-lemma-2}, we may proceed assuming that $|W_1|\geqslant(1+\alpha/2)n$, and that $G[U_1]$ is $(36\epsilon n,\othereta n,\ell+12)$-well-connected.}

Let $q=\ell+14$. By Lemma~\ref{lm:tree_decomposition}, there exist forests $T_0\subseteq T_1\subseteq T_2\subseteq T_3\subseteq T_4=T$, such that $T_3$ is a tree and properties \itref{tree1}-\itref{tree5} hold. By Theorem~\ref{thm:n+k+an-partial}, $G[W_1]$ contains a copy, $R_1$ say, of $T_1$.\footnote{For \itref{main-lemma-2}, as $\Delta(T)\leqslant cn$, each tree $S_v$ of \itref{tree2} satisfies $|S_v|\leqslant cmn+1\leqslant\othereta n$. By Theorem~\ref{thm:n+an-partial}, $G[W_1]$ has a copy, $R_1$ say, of $T_1$.} By Theorem~\ref{thm:any_linear_bound} applied iteratively to the components of $T_2-V(T_1)$, $G[W_1]-V(R_1)$ then contains a copy of $T_2-V(T_1)$, which, taken together with $R_1$, gives a copy, $R_2$ say, of $T_2$.

Let $P_1,\ldots,P_r$ be the paths of length $\ell+14$ attached to $T_2$ to obtain $T_3$. For $i\in[r]$, let $x_i,y_i$ be the endvertices of $P_i$, let $P_i'=P_i-x_i-y_i$ (so that $P_i'$ has length $\ell+12$), and let $x_i',y_i'$ be the images of $x_i,y_i$ in $R_2$. For each $i\in[r]$ in turn, using Lemma~\ref{lm:connector-properties}~ii), there is a copy $Q_i$ of $P_i'$ in the unoccupied vertices of $G[U_1]$, with first vertex in $N^{\diamond_1}(x_i',U_1)$ and last vertex in $N^{\diamond_2}(y_i',U_1)$, where $\diamond_1,\diamond_2\in\{+,-\}$ are taken so that $x_i'Q_iy_i'$ gives a copy of $P_i$. We remark that we may always proceed as, by~\itref{tree5}, the total number of vertices embedded into $U_1$ is at most $|T_3\setminus T_2|\leqslant\othereta n$, and $G[U_1]$ is $(36\epsilon n,\othereta n,\ell+12)$-well-connected with $d^\pm(v,U_1)\geqslant36\epsilon n$ for every $v\in W_1$. Thus, we obtain a copy, $R_3$ say, of $T_3$ in $G[W_0]$. Finally, using Corollary~\ref{cor:tree-extension}, $R_3$ can be extended to an copy of $T_4=T$ in $G$, with the vertices of $V(T_4)\setminus V(T_3)$ copied to $U_0$.
\end{proof}

To finish the proof of both our main results simultaneously we will use \emph{superadditive set functions}.
In the proof of each case, we define a function $f_T:\mathcal{P}(V(T))\to\mathbb{N}_0$ on the power set of $V(T)$, where, for $A\subseteq V(T)$, $f_T(A)$ may be interpreted as representing a rough upper bound on the number of vertices required to guarantee a copy of $T[A]$, according to each of the theorems. The only limitation on $f_T$ required for the proof to work is for it to be \emph{superadditive}.

\begin{defn}
Given a set $X$, we say that a set function $f:\mathcal{P}(X)\to\mathbb{N}_0$ is \emph{superadditive} if $f(A\cup B)\geqslant f(A)+f(B)$ for any disjoint sets $A,B\subseteq X$.
\end{defn}

In particular, we will use the property that, if $f:\mathcal{P}(X)\to\mathbb{N}_0$ is a superadditive set function and $X=A_1\cup\ldots\cup A_r$ is a partition, then
\begin{equation}\label{eq:superadditive-partition}
    f(X)\geqslant\sum_{i\in[r]}f(A_i).
\end{equation}
We also remark that superadditive set functions are increasing, in the sense that if $A\subseteq B$, then $f(A)\leqslant f(B)$.

Given a tree $T$, we will have one particular superadditive set function $f_T:\mathcal{P}(V(T))\to\mathbb{N}_0$ for each main theorem. For Theorem~\ref{thm:n+k+an}, we  take $f_T(A)=|A|+k(A)-2s(A)$, where $k(A)$ denotes the number of leaves of the forest $T[A]$ (and isolated vertices count as two leaves), and $s(A)$ denotes the number of components of the forest $T[A]$. For Theorem~\ref{thm:n+an}, we take  $f_T(A)=|A|$. To see that the first function is superadditive, let $A,B\subseteq V(T)$ be disjoint sets, and compare the forest $T[A\cup B]$ to the forest $T[A]\cup T[B]$. $T[A]\cup T[B]$ can be reached from $T[A\cup B]$ by removing the edges with one endpoint in each of $A$ and $B$ one at a time. Each time an edge is removed, the total number of vertices remains the same, the total number of leaves increases by at most 2, and the total number of components increases by 1. Thus we find that $f_T(A\cup B)\geqslant f_T(A)+f_T(B)$.

We are now ready to prove Theorems~\ref{thm:n+k+an} and~\ref{thm:n+an} using Lemma~\ref{lm:main-theorems-connector}. The proof for each theorem is nearly identical, so we will present a proof for Theorem~\ref{thm:n+k+an}, and explain in the footnotes any places where the proof for Theorem~\ref{thm:n+an} differs. In each case, using Lemma~\ref{lm:transitive-decomposition}, we will have a partition of most of the vertex set of the tournament $G$ into sets $W_1,\ldots,W_r$ such that, if $u\in W_i$ and $v\in W_j$ with $i<j$, then $uv\in E(G)$, and such that each $G[W_i]$ is well-connected if $W_i$ is not too small (i.e., \itref{td1} and \itref{td2} hold). It remains to find a good way to
partition the tree $T$ to embed it across this decomposition. 

For each $i\in [r]$, if $|W_i|\geq \sqrt{\eps}n$, then let $w_i=(1-\alpha/4)|W_i|$, and otherwise let $w_i=|W_i|$. We want to assign a set $U_i$ of $w_i$ vertices of the tree to embed in $W_i$, for each $i\in [r]$. First, if $|W_r|\geq \sqrt{\eps}n$, then we order  the vertices of $T$ as $v_1,\ldots,v_n$  so that all edges of $T$ go forwards in this ordering, and let $U_r$ be the set of the last $w_r$ vertices in this ordering (in this case, we will be able to embed $T[U_r]$ into $G[W_r]$ by \itref{td2} and Lemma~\ref{lm:main-theorems-connector}). If $|W_r|<\sqrt{\eps}n$, then, if possible, we let $U_r$ be $w_r$ out-leaves of $T$, and if it is not possible then we stop. If we have not stopped, then we remove $U_r$ from $T$ and repeat this procedure to find $U_{r-1}$, and so on. Note that if this stops then either we have assigned vertices for each $W_i$, or the remaining forest has at most $\sqrt{\eps}n$ out-leaves. In the latter case we carry out a similar assignment for $W_1,W_2,\ldots$. When this stops either all the vertices have been assigned or the remaining forest has at most $\sqrt{\eps}n$ in-leaves as well as at most $\sqrt{\eps}n$ out-leaves --- such a forest we can embed with $O(\sqrt{\eps}n)$ spare vertices using Theorem~\ref{thm:n+Ck}.

\begin{proof}[Proof of Theorem~\ref{thm:n+k+an} (with appropriate alterations for Theorem~\ref{thm:n+an} indicated)]
Fix $\alpha>0$, and note that we may additionally assume that $\alpha\leqslant1$. Introduce constants $\epsilon, \eta, n_0$ such that $1/n_0\ll\eta\ll\epsilon\ll\alpha$. Given a tree $T$, let $f_T:\mathcal{P}(V(T))\to\mathbb{N}_0$ be the superadditive set function defined by $f_T(A)=|A|+k(A)-2s(A)$, where $k(A)$ denotes the number of leaves of the forest $T[A]$ (and isolated vertices count as two leaves), and $s(A)$ denotes the number of components of the forest $T[A]$.\footnote{For Theorem~\ref{thm:n+an}, fix $\alpha>0$ and introduce constants $\epsilon, \eta, c, n_0$ such that $1/n_0\ll c\ll\eta\ll\epsilon\ll\alpha$. Given a tree $T$, let $f_T:\mathcal{P}(V(T))\to\mathbb{N}_0$ be the superadditive set function defined by $f_T(A)=|A|$.}

Note that, for any $A\subseteq V(T)$,
\begin{equation}\label{eq:f-bounds}
    |A|\leqslant f_T(A)\leqslant 2|A|.
\end{equation}

Let $n\geqslant n_0$. Fix an $n$-vertex $k$-leaf oriented tree $T$ and let $G$ be a $((1+\alpha)n+k)$-vertex tournament, so that $f_T(V(T))+\alpha n\leqslant|G|\leqslant 3n$. We will show that $G$ contains a copy of $T$, thus proving the theorem.\footnote{For Theorem~\ref{thm:n+an}, fix an $n$-vertex oriented tree $T$ with $\Delta(T)\leqslant cn$ and let $G$ be a $(1+\alpha)n$-vertex tournament, so that $f_T(V(T))+\alpha n=|G|\leqslant 3n$. We will show that $G$ contains a copy of $T$, thus proving the theorem.}

By Lemma~\ref{lm:transitive-decomposition}, there is a partition $V(G)=B\cup W_1\cup\ldots\cup W_r$ so that $|B|\leqslant\epsilon n$ and the properties \itref{td1} and \itref{td2} hold (with $\ell=\lceil\epsilon^{-3}\rceil$). For each $i\in [r]$, if $|W_i|\geq \sqrt{\eps}n$, then let $w_i=(1-\alpha/4)|W_i|$, and otherwise let $w_i=|W_i|$.
Partition $[r]$ into intervals $I^-$, $I$, $I^+$ (in that order), so that $I$ is minimal subject to there being disjoint sets $U_i\subseteq V(T)$, $i\in I^-\cup I^+$, for which the following hold.
    \stepcounter{propcounter}
    \begin{enumerate}[label ={\textbf{\Alph{propcounter}\arabic{enumi}}}]
\item\label{assign-1} For each $i\in I^-\cup I^+$, $(1-\epsilon)w_i\leqslant f_T(U_i)\leqslant w_i$.
\item There are no edges from $\cup_{i\in I^+}U_i$ to $V(T)\setminus (\cup_{i\in I^+}U_i)$ in $T$.\label{assign-2}
\item There are no edges from $V(T)\setminus (\cup_{i\in I^-}U_i)$ to $\cup_{i\in I^-}U_i$ in $T$.\label{assign-3}
\item If $|W_i|<\sqrt{\eps}n$, then there are no edges in $T[U_i]$.\label{assign-4}
\item If $i,j\in I^-\cup I^+$ with $i<j$, then there are no edges from $U_j$ to $U_i$ in $T$.\label{assign-5}
\end{enumerate}

Note that this is possible as $I=[r]$ is a valid partition. Let $T'=T-\cup_{i\in I^+\cup I^-}U_i$ and $W=\cup_{i\in I}W_i$. We will show that, for each $i\in I^-\cup I^+$, $G[W_i]$ contains a copy of $T[U_i]$, and $G[W]$ contains a copy of $T'$. Putting these together then gives a copy of $T$, by \ref{assign-2}, \ref{assign-3}, \ref{assign-5}, and \itref{td1}.

If $i\in I^-\cup I^+$ satisfies $|W_i|\geq \sqrt{\eps}n$, then iteratively add edges between leaves (or isolated vertices) of distinct components of $T[U_i]$ to obtain an oriented tree $S_i$ with $V(S_i)=U_i$ and $|S_i|+k_i=f_T(U_i)$, where $k_i$ denotes the number of leaves of $S_i$ (indeed, each additional edge reduces the number of leaves and isolated vertices by 2 and reduces the number of components by 1).\footnote{For Theorem~\ref{thm:n+an}, if $i\in I^-\cup I^+$ satisfies $|W_i|\geqslant\sqrt{\epsilon}n$, then iteratively add edges between distinct components of $T[U_i]$ to obtain an oriented tree $S_i$ with $V(S_i)=U_i$ and $|S_i|+k_i=f_T(U_i)$, where $k_i=0$. Note that this may be accomplished while only adding at most two edges adjacent to any given vertex, and so we may further assume that $\Delta(S_i)\leqslant cn+2\leqslant 2cn$.} We now have
\begin{equation*}
|S_i|=|U_i|\overset{\eqref{eq:f-bounds}}{\geqslant} f_T(U_i)/2\overset{\ref{assign-1}}{\geqslant}(1-\epsilon)w_i/2\geqslant\sqrt{\epsilon}n/4,
\end{equation*}
and
\begin{equation*}
    (1+\alpha/4)|S_i|+k_i\leqslant f_T(U_i)+(\alpha/4)\cdot|U_i|\leqslant(1+\alpha/4)\cdot f_T(U_i)\overset{\ref{assign-1}}{\leqslant}(1+\alpha/4)w_i\leqslant|W_i|.
\end{equation*}
Therefore, by applying Lemma~\ref{lm:main-theorems-connector} with $G[W_i]$ playing the role of $G$ and recalling that $G$ is $(\epsilon n,\eta n,\ell)$-connected by \itref{td2}, $G[W_i]$ contains a copy of $S_i$ and hence also $T[U_i]$. On the other hand, if $|W_i|<\sqrt{\eps}n$, then by \ref{assign-4}, $G[W_i]$ contains a copy of $T[U_i]$, noting that we have $f_T(U_i)=|U_i|$ in this case.

It is left to show that $G[W]$ contains a copy of $T'$. Note that this is trivial if $I=\emptyset$, and so we can assume $I\neq\emptyset$ and label $j_1,j_2\in[r]$ so that $I$ is the interval from $j_1$ to $j_2$. Also, note that, because $\sum_{i\in I^-\cup I^+}f_T(U_i)\leqslant2n$,
\begin{align*}
    |W|&=|G|-|B|-\sum_{i\in I^-\cup I^+}|W_i|\overset{\ref{assign-1}}{\geqslant} f_T(V(T))+\alpha n - \epsilon n - (1-\alpha/4)^{-1}(1-\epsilon)^{-1}\sum_{i\in I^-\cup I^+}f_T(U_i)\\
    &\geqslant f_T(V(T))-\sum_{i\in I^-\cup I^+}f_T(U_i)+\alpha n/4\overset{\eqref{eq:superadditive-partition}}{\geqslant} f_T(V(T'))+\alpha n/4\overset{\eqref{eq:f-bounds}}{\geqslant} |T'|+\alpha n/4.
\end{align*}
If $|W_{j_2}|\geq \sqrt{\eps}n$, then we must have $f_T(V(T'))<(1-\epsilon)w_{j_2}$, otherwise we could order $V(T')$ as $v_1,\ldots,v_{|T'|}$ so that all edges of $T'$ go forwards in this ordering, and define $U_{j_2}=\{v_s,\ldots,v_{|T'|}$\} for some $s$ chosen such that $(1-\epsilon)w_{j_2}\leqslant f_T(U_{j_2})\leqslant w_{j_2}$, a contradiction to the minimality of $I$. Thus, if $|W_{j_2}|\geqslant\sqrt{\epsilon}n$, then $G[W_{j_2}]$, and hence $G[W]$, contains a copy of $T'$ by Lemma~\ref{lm:main-theorems-connector}. Similarly, if $|W_{j_1}|\geq \sqrt{\eps}n$, then $G[W_{j_1}]$, and hence $G[W]$, contains a copy of $T'$. We must have then that $|W_{j_1}|< \sqrt{\eps}n$ and  $|W_{j_2}|< \sqrt{\eps}n$. Thus, by the minimality of $I$, $T'$ has at most $w_{j_2}\leq \sqrt{\eps}n$ out-leaves and at most  $w_{j_1}\leq \sqrt{\eps}n$ in-leaves. As $|W|\geq |T'|+\alpha n/4$, $G[W]$ then contains a copy of $T'$ by Theorem~\ref{thm:n+Ck}, as required.
\end{proof}

\section{Proof of Theorem~\ref{thm:extending-distillation}}\label{sect:technical}
In this section we prove Theorem~\ref{thm:extending-distillation}, which, in the notation in Section~\ref{sect:outline}, finds an index $j_t$ for a regularity cluster for the core $T_0$ of a tree, and a random homomorphism of a fixed digraph $H$ with vertex weight function $\beta$ representing an average component of $T_1-V(T_0)$ (where here, and throughout this section, we use the term \emph{random homomorphism} to refer to any random variable taking values in the set of all possible homomorphisms, in the sense of Theorem~\ref{thm:extending-distillation}). For convenience, we restate the definition of $H$ (see Figure~\ref{fig:H}) and Theorem~\ref{thm:extending-distillation}. \HDefinitionSubsequent

\technical*

To prove Theorem~\ref{thm:extending-distillation}, we first make two key simplifications before dividing into three critical cases. Our first simplification is to work only with the vertices of $H$ representing components attached by an out-edge from $T_0$ (see Figure~\ref{fig:H}). Let $H^+$ and $H^-$ be the subdigraphs of $H$ induced on the vertices with $+$ and $-$ in the superscript, respectively (i.e., the right and the left parts of $H$ in Figure~\ref{fig:H}). Considering each possible location $j\in [r]$ for $j_t$ in Theorem~\ref{thm:extending-distillation}, either a) $j$ has enough weight on its out-edges that a random embedding of $H^-$ can be extended relatively easily (with perhaps some modification) to one of $H$ satisfying our requirements, or b) $j$ has enough weight on its in-edges to similarly extend a random embedding of $H^+$. If many $j\in [r]$ satisfy a), then we may randomly embed $H^-$ into the weighted looped digraph $D$ induced on these $j$. If not, then enough $j\in [r]$ satisfy b), so that we may randomly embed $H^+$ into the weighted looped digraph $D$ induced on these $j$. By appealing to directional duality if necessary, it suffices then to prove a simplified version of Theorem~\ref{thm:extending-distillation} with weight only on $H^+$.

Our second simplification to Theorem~\ref{thm:extending-distillation} is to drop the condition \itref{ext-good-plan-4}; we later show this condition can be recovered without undue difficulty. These two simplifications of Theorem~\ref{thm:extending-distillation} result in Theorem~\ref{thm:overarching}, which we state in Section~\ref{sect:overarching-subcases} after introducing a notational framework of `distillations' in Section~\ref{sect:distillation-setup} in order to have a concise and consistent language for the proofs in this section. The proof of Theorem~\ref{thm:overarching} varies depending on the weight distribution $\beta$ on the vertices in $H$. This falls into three main cases, which we also state in Section~\ref{sect:overarching-subcases}, in the form of Lemmas~\ref{lm:distillation-case-2},~\ref{lm:distillation-case-1} and~\ref{lm:distillation-case-3}, before deducing Theorem~\ref{thm:overarching} from these cases. We then prove the lemma for each of these cases in Section~\ref{sect:3cases-proofs}, before finally deducing Theorem~\ref{thm:extending-distillation} from Theorem~\ref{thm:overarching} in Section~\ref{sect:extending-new}.

\subsection{Distillations}\label{sect:distillation-setup}
We prove Theorem~\ref{thm:extending-distillation} from three specific cases where, roughly speaking, $H$ is replaced by simpler subgraphs of $H$. 
In order to have a concise and consistent language for proving these cases, we will use the notion of a \emph{distillation}, as follows.

\begin{defn}\label{defn:distillation}
A \emph{distillation} is a triple $\mathcal{F}=(F,X,\beta)$, where $F$ is a fully-looped oriented forest, $X\subseteq V(F)$ is a set containing precisely one vertex in each component of $F$, and $\beta:V(F)\to[0,1]$ satisfies $\sum_{v\in V(F)}\beta(v)=1$.
\end{defn}

In Theorem~\ref{thm:extending-distillation}, we have a distillation $(H,X,\beta)$ which we used to represent the average component of $T-V(T_0)$ for Theorem~\ref{thm:n+k+an-partial}. There is some flexibility in how we could have chosen this distillation --- for example we could move all the weight from $y^+$ to $x^+$, or from $u^+$ to $z^+$ and still have a useful distillation of the average component if we can find a matching random homomorphism. However, $H$ is the smallest digraph that records enough structure in the average component to allow every relevant distillation to have a matching random homomorphism. The construction of the random homomorphism falls into three cases depending on the distribution of the weight --- in each case we can move weight off some vertices (different in each case) to simplify the digraph in the distillation for which we find a random homomorphism.

To describe which simplifications of distillations are valid in this way formally, and prove this validity, we will define a transitive relation $\hookrightarrow$ between distillations. Very roughly, given two distillations $\mathcal{F}_0$ and $\mathcal{F}_1$, if $\mathcal{F}_0\hookrightarrow\mathcal{F}_1$, then we can move weight in $\mathcal{F}_0$ (and possibly delete vertices) to transform it into $\mathcal{F}_1$. Formally, we define the relation as follows.

\begin{defn}\label{defn:arrows}
Given distillations $\mathcal{F}_i=(F_i,X_i,\beta_i)$ for $i\in\{0,1\}$, say $\mathcal{F}_0\hookrightarrow\mathcal{F}_1$ if there is a random homomorphism $\rho:F_0\to F_1$ with the following properties.
\stepcounter{propcounter}
\begin{enumerate}[label = {\bfseries \emph{\Alph{propcounter}\arabic{enumi}}}]
\item\label{arrows-1} With probability 1, $\rho(X_0)\subseteq X_1$.
\item \label{arrows-2} $\E(\beta_0(\rho^{-1}(v)))=\beta_1(v)$ for every $v\in V(F_1)$.
\end{enumerate}
\end{defn}

Finally, we need a notion of which distillations are useful --- i.e., which distillations have a matching random homomorphism with properties like those in Theorem~\ref{thm:extending-distillation}. It will be convenient to consider a small collection of distillations and allow a sampling of the random homomorphism to take any one of them as its domain, and so we define the following notion of $\gamma$-goodness on sets of distillations. Roughly speaking, $\gamma$ corresponds to the extra proportion of vertices we need to embed the tree, as in the use of Theorem~\ref{thm:extending-distillation}.

\begin{defn}\label{defn:gamma-good}
Given $\gamma\geq0$, and distillations $\mathcal{F}_i=(F_i,X_i,\beta_i)$, $i\in [m]$, we say $\{\mathcal{F}_i\}_{i=1}^m$ is \emph{$\gamma$-good} if the following holds for any fixed $\alpha>0$: if $1/r\ll\epsilon\ll\alpha$ and $D$ is a complete looped digraph on vertex set $[r]$ with $\epsilon$-complete edge weights $d(e)$, $e\in E(D)$, then there exists some $j_t\in[r]$ and a random $(\phi,i(\phi))$ with the following properties.
\stepcounter{propcounter}
\begin{enumerate}[label = {\bfseries \emph{\Alph{propcounter}\arabic{enumi}}}]
\item\label{good-plan-1} With probability 1, we have that $i(\phi)\in[m]$, that $\phi$ is a homomorphism from $F_{i(\phi)}$ to $D$, and that $j_t\notin\phi(X_{i(\phi)})$.
\item \label{good-plan-2} For each $j\in[r]$, $\E(\beta_{i(\phi)}(\phi^{-1}(j)))\leq \frac{1+\gamma+\alpha}{r}$.
\item\label{good-plan-3} For each $j\in[r]$, $\E(\beta_{i(\phi)}(\phi^{-1}(j)\cap X_{i(\phi)}))\leq d(j_t,j)\cdot \frac{1+\gamma+\alpha}{r}$.
\end{enumerate}
\end{defn}

Note that if $\{\mathcal{F}_i\}_{i=1}^m$ is $\gamma$-good for some $\gamma\geq0$, then $\{\mathcal{F}_i\}_{i=1}^m$ is $\gamma'$-good for every $\gamma'\geq\gamma$. In addition, a set of distillations is $\gamma$-good if and only if it contains a non-empty subset which is $\gamma$-good.

Finally here, we prove the following key lemma that confirms that if a distillation can be simplified via the relation $\hookrightarrow$ to each one of a family of distillations which are collectively $\gamma$-good, then that original distillation is $\gamma$-good, as follows.

\begin{lemma}\label{lm:transfer-good}
Let $\gamma\geqslant0$, and suppose $\mathcal{F}$ and $\mathcal{G}_1,\ldots,\mathcal{G}_m$ are distillations such that $\mathcal{F}\hookrightarrow\mathcal{G}_i$ for every $i\in[m]$. If $\{\mathcal{G}_i\}_{i=1}^m$ is $\gamma$-good, then $\{\mathcal{F}\}$ is $\gamma$-good.
\end{lemma}
\begin{proof}
Let $\mathcal{F}=(F,X,\beta)$ and $\mathcal{G}_i=(G_i,X_i,\beta_i)$ for each $i\in[m]$. For each $i\in[m]$, let $\rho_{i}:F\to G_i$ be a random homomorphism realising $\mathcal{F}\hookrightarrow\mathcal{G}_i$.

Take $1/r\ll\epsilon\ll\alpha$, and let $D$ be a complete looped digraph on vertex set $[r]$ with $\epsilon$-complete edge weights $d(e)$, $e\in E(D)$. Let $j_t\in [r]$ and $(\phi,i(\phi))$ realise that $\{\mathcal{G}_i\}_{i=1}^m$ is $\gamma$-good in the case of $D$. Define $(\psi,k(\psi))$ as follows. First, sample $(\phi,i(\phi))$. Then, with $i(\phi)$ now fixed, sample $\rho_{i(\phi)}$, and set $\psi=\phi\circ\rho_{i(\phi)}$. Let $k(\psi)=1$ with probability 1, and note that \itref{good-plan-1} holds for $(\psi,k(\psi))$.

For $i\in[m]$ let $A_i$ be the event $\{i(\phi)=i\}$. Then, by the law of total expectation,
\begin{align*}
    \E(\beta(\psi^{-1}(j)))&
    =\sum_{i\in[m]}\P(A_i)\cdot\E(\beta(\psi^{-1}(j))\mid A_i)
    =\sum_{i\in[m]}\P(A_i)\cdot\E(\beta(\rho_i^{-1}(\phi^{-1}(j)))\mid A_i)\\
    &\overset{\itref{arrows-2}}{=}\sum_{i\in[m]}\P(A_i)\cdot\E(\beta_i(\phi^{-1}(j))\mid A_i)
    =\E(\beta_{i(\phi)}(\phi^{-1}(j))),
\end{align*}
so \itref{good-plan-2} holds for $(\psi,k(\psi))$, and
\begin{align*}
    \E(\beta(\psi^{-1}(j)\cap X))&
    =\sum_{i\in[m]}\P(A_i)\cdot\E(\beta(\psi^{-1}(j)\cap X)\mid A_i)
    =\sum_{i\in[m]}\P(A_i)\cdot\E(\beta(\rho_{i}^{-1}(\phi^{-1}(j))\cap X)\mid A_i)\\
    &\leqslant\sum_{i\in[m]}\P(A_i)\cdot\E(\beta(\rho_{i}^{-1}(\phi^{-1}(j)\cap\rho_{i}(X)))\mid A_i)\\
    &\overset{\itref{arrows-1}}{\leqslant}\sum_{i\in[m]}\P(A_i)\cdot\E(\beta(\rho_{i}^{-1}(\phi^{-1}(j)\cap X_{i}))\mid A_i)
    \overset{\itref{arrows-2}}{=}\E(\beta_{i(\phi)}(\phi^{-1}(j)\cap X_{i(\phi)})),
\end{align*}
so \itref{good-plan-3} holds for $(\psi,k(\psi))$.
\end{proof}

\subsection{Statement of overarching theorem and subcases}\label{sect:overarching-subcases}
Let $H_0$ be the fully-looped oriented forest with vertices $\{x,y,z,u,w,\bar{x},\bar{z},\bar{u},\bar{w}\}$ and non-looped edges \linebreak $\{xy,zx,zu,wz,\bar{z}\bar{x},\bar{z}\bar{u},\bar{w}\bar{z}\}$, noting that this is the subdigraph of $H$ defined at the start of this section restricted to the vertices with $+$ in the superscript, and with vertices labelled more concisely (see also Figure~\ref{fig:H}). As noted at the start of this section, we will first prove a version of Theorem~\ref{thm:extending-distillation} for this subdigraph of $H$, without the condition \itref{ext-good-plan-4}, before deducing Theorem~\ref{thm:extending-distillation} from this in Section~\ref{sect:extending-new}. Having introduced a framework of distillations, we can now state this version of Theorem~\ref{thm:extending-distillation} concisely, as follows.

\begin{theorem}\label{thm:overarching}
Let $\beta_0:V(H_0)\to[0,1]$ be a function with $\sum_{v\in V(H_0)}\beta_0(v)=1$ and $\beta_0(y)\geqslant\beta_0(x)$, and set $X_0=\{x,\bar{x}\}$. Set $\gamma=\max{\{\beta_0(x,\bar{x}),\beta_0(z,\bar{z})\}}$. Let $\mathcal{H}_0=(H_0,X_0,\beta_0)$. Then, $\{\mathcal{H}_0\}$ is $\gamma$-good.
\end{theorem}

As mentioned before, the proof of Theorem~\ref{thm:overarching} depends on the weight distribution $\beta_0$ on $H_0$. Dividing into cases, solving them, and showing they combine to prove this theorem is no easy task. Doing so while additionally motivating the choice of these cases is more difficult still. However, while we do concentrate on giving as clear and concise a proof of Theorem~\ref{thm:overarching} as possible, we will give some motivation behind the cases by relating them to an embedding of a tree $T$ into a tournament $G$.

In particular, our notation is designed to make the cases as efficient as possible to check, rather than explain the larger understanding that is necessary to produce these cases. To motivate the cases more directly, we now recall the discussion in Section~\ref{sect:outline}. Our aim is to use Theorem~\ref{thm:extending-distillation} to embed a tree $T$ with a small core $T_0$, where $T-V(T_0)$ is a collection of small components. To do this we take the tournament $G$ and find a regularity partition $V_1\cup \ldots\cup V_r$. We then want to make a careful choice of $j_t\in [r]$ and embed $T_0$ into $V_{j_t}$ before distributing the components of $T-V(T_0)$ across the clusters of the regularity partition. The choice of $j_t$ restricts which cluster each given vertex $v\in V(T)\setminus V(T_0)$ can be embedded to. For example, if the path from $T_0$ to $v$ in $T$ is a directed path towards $v$, and we wish to embed $v$ into the cluster $V_i$, then there must be a directed path from $V_{j_t}$ to $V_i$ of edges with positive weight in the reduced digraph obtained from the regularity partition. Fortunately, for our cases we need to consider at most the direction of the first three edges on the path from $T_0$ to $v$ (and the first edge will always be directed away from $T_0$).

Very roughly, we first divide into two cases corresponding to the following situations, where, for example, we use the notation of a $(++)$-path from $u$ to $v$ to be a length two path from $u$ to $v$ comprised of two edges directed forward from $u$ to $v$, with other notation used similarly.
\begin{itemize}
    \item For most of the vertices $v\in V(T)\setminus V(T_0)$, if $T_0$ is embedded to $V_{j_t}$, $i\in [r]$, and there is a $(++)$-path from $V_{j_t}$ to $V_i$ in $D$, then we could embed $v$ to $V_i$.
\item For most of the vertices $v\in V(T)\setminus V(T_0)$, if $T_0$ is embedded to $V_{j_t}$, $i\in [r]$, and there is a $(+-)$-path from $V_{j_t}$ to $V_i$ in $D$ then we could embed $v$ to $V_i$.
\end{itemize}

These cases correspond roughly to Lemma~\ref{lm:distillation-case-2} and Lemma~\ref{lm:distillation-case-new}, respectively. We then further subdivide the latter case into two additional cases, essentially replacing the $(+-)$-path with a $(+-+)$-path and a $(+--)$-path, respectively. This gives us cases 1, 2, and 3 which, in terms of the weight distribution $\beta$ on $H$ correspond to the following roughly descriptions:
\begin{enumerate}
    \item Most of the weight not on $\{x,\bar{x}\}$ is on $y$.
    \item Most of the weight not on $\{x,\bar{x}\}$ is on $\{y,u,\bar{u}\}$ (but Case 1 does not apply).
    \item Most of the weight not on $\{x,\bar{x}\}$ is on $\{z,w,\bar{z},\bar{w}\}$.
\end{enumerate}


We now use our concept of distillations and the relation $\hookrightarrow$ to state the lemmas corresponding to these cases and combine them to prove Theorem~\ref{thm:overarching}. We first divide Theorem~\ref{thm:overarching} into two lemmas -- based on the distribution of $\beta_0$, we distill $H_0$ into $H_1$ or $H_2$, where for the former we remove the vertices $\{u,w,\bar{u},\bar{w}\}$ and in the latter we remove $\{\bar{x},\bar{z},\bar{u},\bar{w}\}$. This gives Lemma~\ref{lm:distillation-case-2} (corresponding to Case 1 above) and Lemma~\ref{lm:distillation-case-new}. We then break Lemma~\ref{lm:distillation-case-new} into two further lemmas, Lemma~\ref{lm:distillation-case-1} and~\ref{lm:distillation-case-3} which correspond respectively to Case 2 and 3 above, where in each case we have a set of distillations rather than simplifying to just one distillation. The structure of this division is depicted in Figure~\ref{fig:technical-implications}.

\begin{figure}
\DiagramHTechnical
\vspace{-2mm}
\caption{The underlying digraphs of the distillations described in this section (with looped edges omitted).}\label{fig:H-forests}
\end{figure}

Let $H_1$ be the fully-looped oriented forest with vertices $\{x,y,z,\bar{x},\bar{z}\}$ and non-looped edges $\{xy,zx,\bar{z}\bar{x}\}$.

\begin{lemma}\label{lm:distillation-case-2}
Let $\mathcal{H}=(H_1,\{x,\bar{x}\},\beta)$ be a distillation with $\beta(y)\geqslant\beta(z,\bar{z}),\beta(x)$. Then, $\{\mathcal{H}\}$ is $\beta(x,\bar{x})$-good.
\end{lemma}

Let $H_2$ be the fully-looped oriented forest with vertices $\{x,y,z,u,w\}$ and non-looped edges $\{xy,zx,zu,wz\}$.

\begin{lemma}\label{lm:distillation-case-new}
Let $\mathcal{H}=(H_2,\{x\},\beta)$ be a distillation with $\beta(y)\leqslant\beta(z,u,w)$. Then, $\{\mathcal{H}\}$ is $(\max{\{\beta(x),\beta(z)\}})$-good.
\end{lemma}

Let $H_{3,1}$ be the fully-looped oriented forest with vertices $\{x,y\}$ and non-looped edges $\{xy\}$. Let $H_{3,2}$ be the fully-looped oriented forest with vertices $\{x,z\}$ and non-looped edges $\{zx\}$.

\begin{lemma}\label{lm:distillation-case-1}
Let $\beta_x\in[0,1]$ and, set $\beta_1(x)=\beta_2(x)=(1+\beta_x)/2$ and $\beta_1(y)=\beta_2(z)=(1-\beta_x)/2$. For $i\in[2]$, set $\mathcal{H}_{i}=(H_{3,i},\{x\},\beta_{i})$. Then, $\{\mathcal{H}_{i}\}_{i=1}^2$ is $\beta_x$-good. 
\end{lemma}

Let $H_{4,1}$ be the fully-looped oriented forest with vertices $\{x,y,z\}$ and non-looped edges $\{xy,zx\}$.
Let $H_{4,2}$ be the fully-looped oriented forest with vertices $\{x,y\}$ and non-looped edges $\{xy\}$.
Let $H_{4,3}$ be the fully-looped oriented forest with vertices $\{x,z\}$ and non-looped edges $\{zx\}$.
Let $H_{4,4}$ be the fully-looped oriented forest with vertices $\{x,z,w\}$ and non-looped edges $\{zx,wz\}$.
Let $H_{4,5}=H_{4,3}$.

\begin{lemma}\label{lm:distillation-case-3}
Let $\beta_x,\beta_y,\beta_z,\beta_u,\beta_w\in [0,1]$ have sum 1 and $\beta_y+\beta_u\leq\beta_z+\beta_w$. Let $\gamma=\max{\{\beta_x,\beta_z\}}$. Take the following weight functions $\beta_i:V(H_{4,i})\to[0,1]$ for $i\in[5]$.
\[
\beta_1(y)=\max\{\beta_x+\beta_y-\gamma,0\}
\;\;\;\;\;\;\;
\beta_1(z)=\min\{\beta_w+\beta_z,1-\beta_z\}
\;\;\;\;\;\;\;
\beta_1(x)=1-\beta_1(y)-\beta_1(z), 
\]
\[
\beta_2(x)=\beta_3(x)=\beta_x+\beta_z+\beta_w
\;\;\;\;\;\;\;
\beta_2(y)=\beta_3(z)=\beta_y+\beta_u,
\]
\[
\beta_4(x)=\min\{\beta_x+\beta_y+\beta_u,\max\{\beta_x+\beta_y,\beta_z\}\}
\;\;\;\;\;\;\;
\beta_4(w)=\beta_w
\;\;\;\;\;\;\;
\beta_4(z)=1-\beta_4(x)-\beta_4(w),
\]
\[
\beta_5(x)=\beta_x+\beta_y
\;\;\;\;\;\;\;
\beta_5(z)=\beta_z+\beta_u+\beta_w.
\]

For $i\in[5]$, set $\mathcal{H}_i=(H_{4,i},\{x\},\beta_i)$. Then, $\{\mathcal{H}_i\}_{i=1}^5$ is $\gamma$-good.
\end{lemma}

We remark that each set of weights defined in Lemma~\ref{lm:distillation-case-3} sum to $1$, either by the choice of $\beta_1(x)$ or $\beta_4(z)$ or as $\beta_x+\beta_y+\beta_z+\beta_u+\beta_w=1$. Furthermore, from the choices they can all immediately be seen to be non-negative except for $\beta_1(x)$ and $\beta_4(z)$, but this we can also show, as follows.

First note that
\begin{align}
\beta_4(z)&=1-\min\{\beta_x+\beta_y+\beta_u+\beta_w,\max\{\beta_x+\beta_y+\beta_w,\beta_z+\beta_w\}\}\nonumber\\
& =\max\{\beta_z,\min\{\beta_u+\beta_z,\beta_x+\beta_y+\beta_u\}\}.\label{eqn-betay}
\end{align}
Therefore, $\beta_4(z)\geq 0$. For use later, we will show that $\beta_1(x)\geq \beta_4(z)$, which also then confirms that $\beta_1(x)\geq 0$.

First suppose that $\beta_z\leq \beta_x+\beta_y$. Then, $1-\beta_z=\beta_u+\beta_y+\beta_x+\beta_w\geq \beta_w+\beta_z$, so that $\beta_1(z)=\beta_w+\beta_z$, and hence
\[
\beta_1(x)=1-\beta_1(y)-\beta_1(z)=1-(\beta_x+\beta_y-\gamma)-(\beta_w+\beta_z)=\beta_u+\gamma\geq \beta_u+\beta_z.
\]
On the other hand, if $\beta_z>\beta_x+\beta_y$, then $\beta_1(y)=0$, so that
\[
\beta_1(x)=1-\beta_1(z)=\max\{1-(\beta_w+\beta_z),\beta_z\}=\max\{\beta_u+\beta_y+\beta_x,\beta_z\}.
\]
Therefore, in both cases, we have $\beta_1(x)\geq \beta_z$ and either $\beta_1(x)\geq \beta_u+\beta_z$ or $\beta_1(x)\geq \beta_u+\beta_y+\beta_x$. Thus, by \eqref{eqn-betay}, we have 
\begin{equation}\label{eqn-beta1x-beta4y}
    \beta_1(x)\geq \beta_4(z).
\end{equation}

We will now outline how Lemmas~\ref{lm:distillation-case-2},~\ref{lm:distillation-case-new},~\ref{lm:distillation-case-1}, and~\ref{lm:distillation-case-3} together imply Theorem~\ref{thm:overarching} (see Figure~\ref{fig:technical-implications}). First, we will show that Theorem~\ref{thm:overarching} follows from Lemmas~\ref{lm:distillation-case-2} and~\ref{lm:distillation-case-new}. We will then show that Lemma~\ref{lm:distillation-case-new} follows from Lemmas~\ref{lm:distillation-case-1} and~\ref{lm:distillation-case-3}. The proof of each of these implications is a straightforward case of verifying certain $\hookrightarrow$ relations between the relevant distillations hold according to the different values $\beta$ may take, and applying Lemma~\ref{lm:transfer-good} to deduce $\gamma$-goodness. All that will remain then is to prove Lemmas~\ref{lm:distillation-case-2},~\ref{lm:distillation-case-1}, and~\ref{lm:distillation-case-3}, which we do in Section~\ref{sect:3cases-proofs}, and then deduce Theorem~\ref{thm:extending-distillation} from Theorem~\ref{thm:overarching}, which we do in Section~\ref{sect:extending-new}.

\begin{figure}
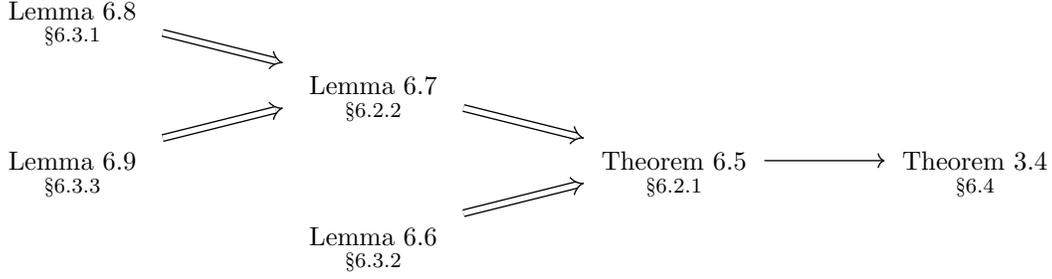

\DiagramTechnicalImplications
\vspace{0cm}
\caption{An overview of how the results of this section combine to prove Theorem~\ref{thm:extending-distillation}. Each implication denoted by $\implies$ indicates a suitable application of Lemma~\ref{lm:transfer-good}.}\label{fig:technical-implications}
\end{figure}

\subsubsection{Proof of Theorem~\ref{thm:overarching} using Lemmas~\ref{lm:transfer-good},~\ref{lm:distillation-case-2}, and~\ref{lm:distillation-case-new}}\label{sect:proof-of-6.5}
Using Lemma~\ref{lm:transfer-good}, it is simple to deduce Theorem~\ref{thm:overarching} from Lemmas~\ref{lm:distillation-case-2} and~\ref{lm:distillation-case-new}.
\begin{proof}[Proof of Theorem~\ref{thm:overarching}.]
Define $\rho_1:H_0\to H_1$ by setting $\rho_1(x)=x$, $\rho_1(y)=y$, $\rho_1(z,u,w)=z$, $\rho_1(\bar{x})=\bar{x}$ and $\rho_1(\bar{z},\bar{u},\bar{w})=\bar{z}$. Define $\rho_2:H_0\to H_2$ by setting $\rho_2(x,\bar{x})=x$, $\rho_2(y)=y$, $\rho_2(z,\bar{z})=z$, $\rho_2(u,\bar{u})=u$ and $\rho_2(w,\bar{w})=w$.

For each $i\in[2]$, let $\beta_i:V(H_i)\to[0,1]$ be given by $\beta_i(v)=\beta_0(\rho_i^{-1}(v))$. Let $\mathcal{H}_1=(H_1,\{x,\bar{x}\},\beta_1)$ and $\mathcal{H}_2=(H_2,\{x\},\beta_2)$. Note that, for each $i\in[2]$, the homomorphism $\rho_i:H_0\to H_i$ realises $\mathcal{H}_0\hookrightarrow\mathcal{H}_i$.

If $\beta_0(y)\geqslant\beta_0(z,u,w,\bar{z},\bar{u},\bar{w})$, then, by Lemma~\ref{lm:distillation-case-2}, $\{\mathcal{H}_1\}$ is $\beta_0(x,\bar{x})$-good. On the other hand, if $\beta_0(y)\leqslant\beta_0(z,u,w,\bar{z},\bar{u},\bar{w})$, then, by Lemma~\ref{lm:distillation-case-new}, $\{\mathcal{H}_2\}$ is $(\max{\{\beta_0(x,\bar{x}),\beta_0(z,\bar{z})\}})$-good. In either of these cases we find $\{\mathcal{H}_0\}$ is $\gamma$-good, by Lemma~\ref{lm:transfer-good}.
\end{proof}

\subsubsection{Proof of Lemma~\ref{lm:distillation-case-new} using Lemmas~\ref{lm:transfer-good},~\ref{lm:distillation-case-1}, and~\ref{lm:distillation-case-3}}\label{sect:proof-of-6.7}
To prove Lemma~\ref{lm:distillation-case-new} follows from Lemmas~\ref{lm:distillation-case-1} and~\ref{lm:distillation-case-3} using Lemma~\ref{lm:transfer-good} requires more checking due to the larger sets of distillations, but this is straightforward, as follows.
\begin{proof}[Proof of Lemma~\ref{lm:distillation-case-new}]
For $v\in V(H_2)$, let $\beta_v=\beta(v)$. Note that we have $\beta_y\leqslant\beta_z+\beta_u+\beta_w$. Define $\mathcal{H}_{3,i}$, $i\in[2]$ and $\mathcal{H}_{4,i}$, $i\in[5]$ as described in Lemmas~\ref{lm:distillation-case-1} and \ref{lm:distillation-case-3}.

We will later prove the following two claims.

\begin{claim}\label{clm:transfer-1}
If $\beta_y+\beta_u\geqslant\beta_z+\beta_w$, then $\mathcal{H}\hookrightarrow\mathcal{H}_{3,i}$ for $i\in[2]$.
\end{claim}

\begin{claim}\label{clm:transfer-2}
If $\beta_y+\beta_u\leqslant\beta_z+\beta_w$, then $\mathcal{H}\hookrightarrow\mathcal{H}_{4,i}$ for $i\in[5]$.
\end{claim}

If $\beta_y+\beta_u\geqslant\beta_z+\beta_w$, then $\mathcal{H}$ is $\gamma$-good by Lemma~\ref{lm:distillation-case-1}, Lemma~\ref{lm:transfer-good} and Claim~\ref{clm:transfer-1}. Otherwise, if $\beta_y+\beta_u\leqslant\beta_z+\beta_w$, then $\mathcal{H}$ is $\gamma$-good by Lemma~\ref{lm:distillation-case-3}, Lemma~\ref{lm:transfer-good} and Claim~\ref{clm:transfer-2}. Therefore, it remains only to prove Claims~\ref{clm:transfer-1} and \ref{clm:transfer-2}.

\begin{proof}[Proof of Claim~\ref{clm:transfer-1}]
For $i\in[2]$, let $\beta_i:V(H_{3,i})\to[0,1]$ be defined as in Lemma~\ref{lm:distillation-case-1}.

To realise $\mathcal{H}\hookrightarrow\mathcal{H}_{3,1}$: If $\beta_y+\beta_u=0$ (and hence, $\beta_1(y)=0$) then let $p_1=0$, and otherwise let
\begin{equation*}
    p_1=\frac{\beta_1(y)}{\beta_y+\beta_u}=\frac{1-\beta_x}{2(\beta_y+\beta_u)}= \frac{\beta_y+\beta_u+\beta_z+\beta_w}{2(\beta_y+\beta_u)}\leq 1.
\end{equation*}
Define $\rho_1:H_2\to H_{3,1}$ by $\rho_1(x,z,w)=x$ and setting $\rho_1(y,u)=y$ with probability $p_1$, and otherwise setting $\rho_1(x,y,z,u,w)=x$.

To realise $\mathcal{H}\hookrightarrow\mathcal{H}_{3,2}$: If $\beta_z+\beta_u+\beta_w=0$ (and hence, $\beta_2(z)=0$) then let $p_2=0$, and otherwise let
\begin{equation*}
    p_2=\frac{\beta_2(z)}{\beta_z+\beta_u+\beta_w}=\frac{1-\beta_x}{2(\beta_z+\beta_u+\beta_w)}= \frac{\beta_y+\beta_z+\beta_u+\beta_w}{2(\beta_z+\beta_u+\beta_w)}\leq 1.
\end{equation*}
Define $\rho_2:H_2\to H_{3,2}$ by $\rho_2(x,y)=x$ and setting $\rho_2(z,u,w)=z$ with probability $p_2$, and otherwise setting $\rho_2(x,y,z,u,w)=x$.
\renewcommand{\qedsymbol}{$\boxdot$}
\end{proof}
\renewcommand{\qedsymbol}{$\square$}

\begin{proof}[Proof of Claim~\ref{clm:transfer-2}]
For $i\in[5]$, let $\beta_i:V(H_{4,i})\to[0,1]$ be defined as in Lemma~\ref{lm:distillation-case-3}.

To realise $\mathcal{H}\hookrightarrow\mathcal{H}_{4,1}$: If $\beta_y=0$ (and hence, $\beta_1(y)=0$) then let $p_1=0$, and otherwise let 
\begin{equation*}
  p_1=\frac{\beta_1(y)}{\beta_y}=\frac{\max\{\beta_x+\beta_y-\max\{\beta_x,\beta_z\},0\}}{\beta_y}\leq \frac{\beta_y}{\beta_y}=1.
\end{equation*}
If $\beta_z+\beta_u+\beta_w=0$ (and hence, $\beta_1(z)=0$) then let $p_1'=0$, and otherwise let
\begin{equation*}
p_1'=\frac{\beta_1(z)}{\beta_z+\beta_u+\beta_w}
\leq
\frac{\beta_w+\beta_z}{\beta_z+\beta_u+\beta_w}\leq 1.
\end{equation*}
Define $\rho_1:H_0\to H_{4,1}$ by $\rho_1(x)=x$, and independently at random with probability $p_1$ setting $\rho_1(y)=y$ and otherwise setting $\rho_1(y)=x$, and independently at random with probability $p_1'$ setting $\rho_1(z,u,w)=z$ and otherwise setting $\rho_1(z,u,w)=x$.

To realise $\mathcal{H}\hookrightarrow\mathcal{H}_{4,2}$: Define $\rho_2:H_2\to H_{4,2}$ by $\rho_2(x,z,w)=x$ and $\rho_2(y,u)=y$.

To realise $\mathcal{H}\hookrightarrow\mathcal{H}_{4,3}$: If $\beta_z+\beta_w+\beta_u=0$ (and hence, $\beta_3(z)=0$) then let $p_3=0$, and otherwise let
\begin{equation*}
    p_3=\frac{\beta_3(z)}{\beta_z+\beta_w+\beta_u}= \frac{\beta_y+\beta_u}{\beta_z+\beta_w+\beta_u}\leq 1.
\end{equation*}
Define $\rho_3:H_2\to H_{4,3}$ by $\rho_3(x,y)=x$ and setting $\rho_3(z,w,u)=z$ with probability $p_3$, and otherwise setting $\rho_3(z,w,u)=x$.

To realise $\mathcal{H}\hookrightarrow\mathcal{H}_{4,4}$: From \eqref{eqn-betay} we have $\beta_z\leq\beta_4(z)\leq\beta_z+\beta_u$. Using this, if $\beta_u=0$ (and hence, $\beta_4(z)=\beta_z$) then let $p_4=0$, and otherwise let
\begin{equation*}
    p_4=\frac{\beta_4(z)-\beta_z}{\beta_u},
\end{equation*}
so that $0\leq p_4\leq 1$ and $p_4\beta_u+\beta_z=\beta_4(z)$. Define $\rho_4:H_2\to H_{4,4}$ by $\rho_4(x,y)=x$, $\rho_4(z)=z$, $\rho_4(w)=w$, and setting $\rho_4(u)=z$ with probability $p_4$, and otherwise setting $\rho_4(u)=x$.

To realise $\mathcal{H}\hookrightarrow\mathcal{H}_{4,5}$: Define $\rho_5:H_2\to H_{4,5}$ by $\rho_5(x,y)=x$ and $\rho_5(z,u,w)=z$.
\renewcommand{\qedsymbol}{$\boxdot$}
\qedhere
\renewcommand{\qedsymbol}{$\square$}
\qedsymbol
\end{proof}
\renewcommand{\qedsymbol}{}
\end{proof}
\renewcommand{\qedsymbol}{$\square$}

\subsection{Proofs of the three cases}\label{sect:3cases-proofs}

We are now ready to prove Lemmas~\ref{lm:distillation-case-2},~\ref{lm:distillation-case-1}, and~\ref{lm:distillation-case-3}, thus completing the proof of Theorem~\ref{thm:overarching}. We give these in order of difficulty, first proving Lemma~\ref{lm:distillation-case-1}, followed by Lemma~\ref{lm:distillation-case-2} and finally Lemma~\ref{lm:distillation-case-3}, with an informal motivating discussion preceding each proof.

\subsubsection{Proof of Lemma~\ref{lm:distillation-case-1}}\label{sec:lmcase1}
In the following proof of Lemma~\ref{lm:distillation-case-1}, we describe the random $(\phi,i(\phi))$ realising the $\beta_x$-goodness of the set $\{\mathcal{H}_i\}_{i=1}^2$. We assume (by relabelling) that $j_t=r$ has at least average out-edge weight, and describe a simple random homomorphism based on these edge weights. As this proof is relatively easy to check we do not motivate this further, but comment on it in the motivation for the other two cases. In the proof, we will use $N_D(j)$ to denote the set of $j'\in V(D)$ with $d(j,j')+d(j',j)>0$.

\begin{proof}[Proof of Lemma~\ref{lm:distillation-case-1}]
Let $\gamma=\beta_x$. Let $1/r\ll\epsilon\ll\alpha$, and let $D$ be a complete looped digraph on vertex set $[r]$ with $\epsilon$-complete edge weights $d(e)$, $e\in E(D)$. We will find a random $(\phi,i(\phi))$ satisfying \itref{good-plan-1}-\itref{good-plan-3}. By relabelling, we can assume that  
\[
\sum_{j\in [r-1]}d(r,j)\geqslant (\textstyle{\frac{1}{2}}-2\epsilon)\cdot r.
\]
For $j\in[r-1]$, choose $0\leqslant d_j\leqslant d(r,j)$ so that $\sum_{j\in[r-1]}d_j=(\frac{1}{2}-2\epsilon)\cdot r$.

Define $(\phi,i(\phi))$ randomly as follows. First, choose $\phi(x)\in [r-1]$ at random, so that $\phi(x)=j$ with probability $d_j/(\frac{1}{2}-2\epsilon)\cdot r$. Then choose $j'\in N_D(\phi(x))$ at random so that $j'=j$ with probability $(1-d_j)/\sum_{j\in N_D(\phi(x))}(1-d_j)$. If $d(\phi(x),j')>0$, then set $i(\phi)=1$ and $\phi(y)=j'$. Otherwise, set $i(\phi)=2$ and $\phi(z)=j'$. Note that, in either case, $\phi$ is a homomorphism from $H_{3,\phi(i)}$ to $D$. By identifying $j_t=r$, \itref{good-plan-1} holds. We now note that, for $j\in[r-1]$,
\begin{equation*}
    \E(\beta_{i(\phi)}(\phi^{-1}(j)\cap\{x\}))=\frac{(1+\beta_x)}{2}\cdot\frac{d_j}{(\frac{1}{2}-2\epsilon)\cdot r}\leqslant d_j\cdot\frac{1+\gamma+\alpha}{r},
\end{equation*}
and, because $\sum_{j\in N_D(j')}(1-d_j)\geqslant(\frac{1}{2}-3\epsilon)\cdot r$ for any $j'\in[r-1]$,
\begin{equation*}
    \E(\beta_{i(\phi)}(\phi^{-1}(j)\cap\{y,z\}))\leqslant\frac{(1-\beta_x)}{2}\cdot\frac{(1-d_j)}{(\frac{1}{2}-3\epsilon)\cdot r}\leqslant (1-d_j)\cdot\frac{1+\gamma+\alpha}{r}.
\end{equation*}
Thus, we deduce \itref{good-plan-2} and \itref{good-plan-3} hold.
\end{proof}

\subsubsection{Proof of Lemma~\ref{lm:distillation-case-2}}\label{sec:lmcase2}
Describing our proofs of the two remaining cases directly is difficult, in part because we are finding homomorphisms to a weighted digraph so we can apply our results to a regularity partition. What we do instead is describe an embedding of a tree $T$ into a tournament $G$ that follows all the major steps in our proof in an analogous, but more comprehensible, way. The embedding we describe is plausible but lacks detail and ignores several subtleties that influence the formal proof -- our aim is to give a step by step embedding (see steps 1 to 8 below, and also Figure~\ref{fig:technical-C}) in a simplified set-up that, by comparison, makes the proof of Lemma~\ref{lm:distillation-case-2} easier to follow.

For this simplified set-up, assume that we have a tree $T$ containing a vertex $t$ with only out-neighbours in $T$, such that $T-t$ consists of small components. Assume further we have a tournament $G$, which is larger than $T$, into which we are attempting to find an embedding $\psi$ of $T$. As $t$ has many out-neighbours in $T$, an obvious choice for $\psi(t)$ is a vertex in $G$ with maximal out-degree -- say $v_t$ is such a vertex and set $\psi(t)=v_t$. Let $A=N_G^+(v_t)$ and $B=V(G)\setminus (A\cup \{v_t\})$. Any component of $T-t$ can be embedded into $G[A]$ to extend the embedding to cover that component (using, for example, Theorem~\ref{thm:any_linear_bound}), but there is not necessarily enough room in $A$ to embed all of the components at once, and so the challenge is to embed components so that many vertices in $B$ are also used.

Let $\hat{H}_1$ have vertex set $\{x,y,z\}$ and edge set $\{xy,zx\}$. For each component of $T-t$, map the out-neighbour of $t$ to $x$, any vertex whose path in $T$ from $t$ begins with two out-edges to $y$, and any vertex whose path in $T$ from $t$ begins with an out-edge then an in-edge to $z$. Thus, all vertices are mapped to $V(\hat{H}_1)$. We always want to embed the vertices of $T$ mapped to $x$ into $A$, so that they are out-neighbours of $v_t$. If all the edges between $A$ and $B$ in $G$ are directed from $A$ to $B$ then, given a component of $T-t$ we could embed vertices mapped to $y$ into $B$ and vertices mapped to $z$ into $A$.  On the other hand, if all the edges between $A$ and $B$ in $G$ are directed from $B$ to $A$ then, given a component, we could embed vertices mapped to $z$ into $B$ and vertices mapped to $y$ into $A$. In practice, we expect the edges between $A$ and $B$ to meet neither extreme, and that some components can be embedded with vertices in $B$ using edges directed from $A$ to $B$ and some using edges directed from $B$ to $A$.

It may be be that all, or almost all, the edges between $A$ and $B$ in $G$ are directed from $A$ to $B$. Here, it is crucial that Lemma~\ref{lm:distillation-case-2} covers the case corresponding to components having enough vertices mapped to $y$ in order to place plenty of vertices into $B$. The other extreme, where almost all the edges between $A$ and $B$ in $G$ are directed from $B$ to $A$, cannot occur unless $B$ is very small, otherwise we can find a vertex in $B$ with high enough out-degree, both in $G[B]$ and into $A$, so that it has higher out-degree than $v_t$.

In trees corresponding to Case 1 (i.e., those for which we use Lemma~\ref{lm:distillation-case-2}), more vertices are mapped to $y$ than $z$, so we prefer to embed components with the vertices mapped to $y$ embedded in $B$. Our embedding is then via the following steps (where we first recap the embedding of $t$), and also sketched in Figure~\ref{fig:technical-C}.

\begin{enumerate}
    \item Embed $t$ to a vertex $v_t$ with a largest out-neighbourhood in $G$, and let $A=N^+_G(v)$ and $B=V(G)\setminus (A\cup\{v_t\})$.
    \item Embed as many components of $T-t$ mapped to $\{x,y,z\}\subseteq V(H_1)$ as possible with the vertex mapped to $x\in V(H_1)$ in $A$ and either a) all other vertices embedded into $B$ or b) the vertices mapped to $y$ embedded into $B$ and the vertices mapped to $z$ embedded into $A$. 
    \item Subject to this, choose the embedding maximising the number of components satisfying a).
    \item Let $A_0$ be the unused vertices in $A$ and $B_0$ be the unused vertices in $B$.
    \item If $B_0$ is small (smaller than the number of leaves of $T$, say), then we do not need to use these vertices and can greedily find the remaining components within $A_0$. Thus, we assume $B_0$ is not too small and that we have not found all our components.
    \item $A_0$ then cannot be too small as we always embedded enough vertices in $B$ compared to $A$.
    \item There cannot be many edges directed from $A_0$ to $B_0$, for otherwise another component could be embedded with vertices mapped to $y$ embedded in $B_0$.
    \item Using a sequence of deductions from the maximality of our component embeddings, we can then find large sets $A_1$ and $B_1$ with $A_0\subseteq A_1\subseteq A$ and $B_0\subseteq B_1\subseteq B$ so that the edges in $G$ are mostly directed from $B_1$ into $A_1$, before concluding there is some vertex in $B_1$ with out-degree approximately $|A_1|+|B_1|/2$, which will be higher than the out-degree of $v_t$, $|A|$, giving a contradiction.
    \end{enumerate}
An example deduction is the following: for components satisfying b), the image of the vertex mapped to $x$ must have few in-neighbours in $B_0$, for otherwise the vertices mapped to $z$ could be embedded into $B_0$ to increase the number of embedded components satisfying a). Consequently, there cannot be many edges from $A_0$ to the images of vertices mapped to $y$ in components satisfying b), else we could move the images of such vertices into $B_0$ and embed a new component using the freed up space. Thus we can add the images of vertices mapped to $y$ in components satisfying b) to $B_1$.

\begin{figure}
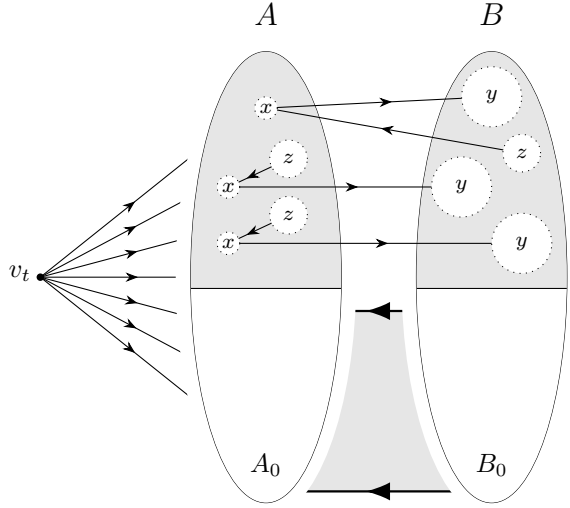

\DiagramTechnicalExampleC
\vspace{-0.75cm}
\caption{An example of how we aim to embed components of $T-t$ into the sets $A$ and $B$ in the simplified set-up. For this case, where more vertices in the components are mapped to $y$, we at first aim to embed these vertices mapped to $y$ in $B$. Once this is no longer possible, we find that edges between the leftover vertices are mostly directed from $B_0$ towards $A_0$.}\label{fig:technical-C}
\end{figure}

We now describe the proof of Lemma~\ref{lm:distillation-case-2} in comparison to these steps. For the lemma, $T$ contains a small core $T_0\subseteq T$ (corresponding to $t$ above) and we have a distillation $(H_1,\{x,\bar{x}\},\beta)$ representing an average component of $T-V(T_0)$, where $H_1$ is the fully-looped oriented forest with vertex set $\{x,y,z,\bar{x},\bar{z}\}$ and edge set $\{xy,zx,\bar{z}\bar{x}\}$. We have a complete looped weighted digraph $D$ which represents a regularity partition, and choose $j_t$ to  maximise the weight on the out-edges from $j_t$ in $D$ (cf.\ 1.\ above). By relabelling, we assume $j_t=r$. Instead of having a partition $A\cup B$ of the other vertices $j\in V(D)$, each vertex is duplicated and lies in both $A$ and $B$ with a weight, representing the proportion of that vertex that is in the out- and in-neighbourhood of $j_t$ respectively (i.e., the proportions $d(j_t,j)$ and $d(j,j_t)$). 

Instead of embedding components, we find homomorphisms $\phi_1,\ldots,\phi_s$ (for some appropriate $s$) from $H_1$, before ultimately picking at random from these homomorphisms to get our required random homomorphism. These homomorphisms effectively allocate space within regularity clusters (represented by vertices of $D$) to embed a batch of components of $T-V(T_0)$, and we similarly aim to allocate as much space in $B$ as possible.
To do this, we first find as many as many homomorphisms $\hat{\phi}_1,\ldots,\hat{\phi}_{s_0}$ from $\hat{H}_1=H_1[\{x,y,z\}]$ as possible so that, ideally, $y$ and $z$ are both embedded into $B$ (see condition~\ref{sub-1} in the proof), and, failing this, at least $y$ is embedded into $B$ (see~\ref{sub-2}), while $x$ is always embedded to $A$  (cf.\ 2.\ and 3.\ above). In doing so, we always ensure that the total weight assigned to any one vertex of $D$ is not too much (see~\ref{sub-3}). Maximising the number of such homomorphisms ($s_0$), and then the number for which \ref{sub-1} is relevant, in fact will allow us to find the remaining homomorphisms $\hat{\phi}_{s_0+1},\ldots,\hat{\phi}_{s}$ before extending them to homomorphisms from $H_1$ (cf.\ 5.\ above). We prove this by assuming it cannot be done and steadily deducing a series of claimed properties of $D$ that ultimately allow us to find a vertex with more weight on its out-edges in $D$ than $j_t$, a contradiction.

\begin{proof}[Proof of Lemma~\ref{lm:distillation-case-2}]
Let $\gamma=\beta(x,\bar{x})$. Let $1/r\ll\epsilon\ll\alpha$. We remark that $\gamma\leqslant 1$, and we may also assume that $\alpha\leqslant 1$, so we have $(1+\gamma+\alpha)\leqslant3$.

Let $D$ be a complete looped digraph on vertex set $[r]$ with $\epsilon$-complete edge weights $d(e)$, $e\in E(D)$. We will find a random $(\phi,i(\phi))$ satisfying \itref{good-plan-1}-\itref{good-plan-3}. By relabelling, we can assume that  
\begin{equation}\label{eq:r-maximal}
\sum_{j\in [r-1]}d(r,j)=\max_{i\in [r]}\sum_{j\in [r]\setminus \{i\}}d(i,j)\geqslant (\textstyle\frac{1}{2}-2\epsilon)\cdot r.
\end{equation}

Take two new disjoint vertex sets $A=\{a_1,\ldots,a_{r-1}\}$ and $B=\{b_1,\ldots,b_{r-1}\}$. Let $\bar{D}$ be the weighted complete looped digraph on $A\cup B$ in which the edges $a_ib_j$, $a_ia_j$, $b_ib_j$ and $b_ia_j$ have weight $d(i,j)$.

Let $s$ be such that $1/s\ll 1/r$. For each $i\in [r-1]$, let $w_{a_i}=d(r,i)\cdot (1+\gamma+\alpha)\cdot s/r$ and $w_{b_i}=(1-d(r,i))\cdot(1+\gamma+\alpha)\cdot s/r$.
Note that
\begin{equation}\label{eqn-sumAB-w}
    \sum_{v\in A\cup B}w_v\geqslant(1+\beta(x,\bar{x})+7\alpha/8)\cdot s,
\end{equation}
and
\begin{equation}\label{eqn-sumB-w}
    \sum_{v\in B}w_v=\sum_{i\in[r-1]}(1-d(r,i))\cdot(1+\gamma+\alpha)\cdot s/r\overset{\eqref{eq:r-maximal}}{<}(\textstyle\frac{1}{2}+\frac{1}{2}\beta(x,\bar{x})+3\alpha/4)\cdot s.
\end{equation}

We aim to find homomorphisms $\phi_1,\ldots,\phi_s:H_1\to \bar{D}$, with $\phi_i(x),\phi_i(\bar{x})\in A$ for each $i\in[s]$ and $\sum_{i\in[s]}\beta(\phi_i^{-1}(v))\leq w_v$ for every $v\in V(\bar{D})$. Then if $\phi:H_1\to D$ is the natural homomorphism induced by a uniform random selection from $\{\phi_1,\ldots,\phi_s\}$, the conclusion of the theorem will hold.

Let $\hat{H}_1=H_1[\{x,y,z\}]$. Let $s_0\leqslant s$ be the largest integer for which there exist homomorphisms $\hat{\phi}_1,\ldots,\hat{\phi}_{s_0}:\hat{H}_1\to \bar{D}$ and indicators $j_1,\ldots,j_{s_0}\in [2]$ such that the following properties hold.
\stepcounter{propcounter}
\begin{enumerate}[label = {\bfseries \Alph{propcounter}\arabic{enumi}}]
\item\label{sub-1} For each $i\in[s_0]$ with $j_i=1$, we have 
    $\hat{\phi}_i(x)\in A$, 
    $\hat{\phi}_i(y)\in B$, and
    $\hat{\phi}_i(z)\in B$.
\item\label{sub-2}
For each $i\in[s_0]$ with $j_i=2$, we have
    $\hat{\phi}_i(x)\in A$, 
    $\hat{\phi}_i(y)\in B$, and
    $\hat{\phi}_i(z)\in A$.
\item\label{sub-3} For each $v\in A\cup B$, $\sum_{i\in [s_0]}\beta(\hat{\phi}_i^{-1}(v))\leq w_v$.
\end{enumerate}

\noindent Subject to this, maximise the number of $i\in[s_0]$ with $j_i=1$. Let $I_1$ be the set of $i\in[s_0]$ with $j_i=1$, and let $I_2$ be the set of $i\in[s_0]$ with $j_i=2$. For each $v\in A\cup B$, let $\hat{w}_v=\sum_{i\in [s_0]}\beta(\hat{\phi}_i^{-1}(v))$, so that, by \ref{sub-3}, we have $\hat{w}_v\leqslant w_v$.

Note that
\begin{equation}\label{eqn-sumAB-w-hat}
\sum_{v\in A\cup B}\hat{w}_v=\sum_{v\in A\cup B}\sum_{i\in [s_0]}\beta(\hat{\phi}_i^{-1}(v))=
\sum_{i\in [s_0]}\beta(\hat{\phi}_i^{-1}(A\cup B))=\beta(x,y,z)\cdot s_0.
\end{equation}

Let $B_0$ be the set of $v\in B$ with  $w_v-\hat{w}_v\geqslant 1$. Let $A_0$ be the set of $v\in A$ with  $w_v-\hat{w}_v\geqslant 2$, noting that we are placing a slightly stronger condition on the definition of $A_0$ to enable a switching argument later on (see the end of the proof of Claim~\ref{clm:case2-4}).

We now show that we are done, unless $\sum_{v\in B}(w_v-\hat{w}_v)$ is not too small.

\begin{claim}\label{clm:case2-J_1-too-large} Either there exists a random $(\phi,i(\phi))$ satisfying \itref{good-plan-1}-\itref{good-plan-3}, or
\begin{equation}\label{eq:1-large}
\sum_{v\in B}(w_v-\hat{w}_v)>(\beta(x,\bar{x})+3\alpha/4)\cdot s.
\end{equation}
\end{claim}
\begin{proof}[Proof of Claim~\ref{clm:case2-J_1-too-large}]
Suppose that $\sum_{v\in B}(w_v-\hat{w}_v)\leqslant(\beta(x,\bar{x})+3\alpha/4)\cdot s$. Then
\begin{align}
    \sum_{v\in A}(w_v-\hat{w}_v)&=\sum_{v\in A\cup B}(w_v-\hat{w}_v)-\sum_{v\in B}(w_v-\hat{w}_v) \nonumber\\
    &\hspace{-0.60cm}\overset{\eqref{eqn-sumAB-w},\eqref{eqn-sumAB-w-hat}}{\geqslant} (1+\beta(x,\bar{x})+7\alpha/8)\cdot s-\beta(x,y,z)\cdot s_0-(\beta(x,\bar{x})+3\alpha/4)\cdot s \nonumber\\
    &= (1+\alpha/8)\cdot s- \beta(x,y,z)\cdot s_0 \nonumber\\
    &=\beta(\bar{x},\bar{z})\cdot s_0+\beta(x,y,z,\bar{x},\bar{z})\cdot (s-s_0)+(\alpha/8)\cdot s. \label{eq:possible-to-extend}
\end{align}
Greedily extend the homomorphisms $\hat{\phi}_1,\ldots,\hat{\phi}_{s_0}:\hat{H}_1\to \bar{D}$ to homomorphisms $\phi_1,\ldots,\phi_{s_0}:H_1\to \bar{D}$, so that, for every $i\in[s_0]$,  $\phi_i|_{\{x,y,x\}}=\hat{\phi_i}$, and $\phi_i(\bar{x}),\phi_i(\bar{z})\in A$. Then, greedily choose homomorphisms $\phi_{s_0+1},\ldots,\phi_{s}$ so that, for every $i\in[s]\setminus[s_0]$, $\phi_i(V(H_1))=\{a_j\}$ for some $a_j\in A$. These steps are possible, while also ensuring that $\sum_{i\in[s]}\beta(\phi_i^{-1}(v))\leq w_v$ for every $v\in V(\bar{D})$, due to \eqref{eq:possible-to-extend}. Then, by defining $(\phi,i(\phi))$ by sampling $\phi$ from $\phi_1,\ldots,\phi_s$ uniformly at random (identifying the result as a map $V(H_1)\to V(D)$ in the natural way) and setting $i(\phi)=1$, we obtain a random $(\phi,i(\phi))$ satisfying \itref{good-plan-1}-\itref{good-plan-3}.
\renewcommand{\qedsymbol}{$\boxdot$}
\end{proof}
\renewcommand{\qedsymbol}{$\square$}

Thus, we may now assume that (\ref{eq:1-large}) holds, and hence also $|B_0|\geqslant2\epsilon r$.
In particular, as $\hat{\phi}_i(y)\in B$ for each $i\in [s_0]$, we have
\begin{equation*}
    \beta(y)\cdot s_0 \leqslant\sum_{v\in B}\hat{w}_v\overset{(\ref{eq:1-large})}{<}\sum_{v\in B}w_v-(\beta(x,\bar{x})+3\alpha/4)\cdot s\overset{\eqref{eqn-sumB-w}}{<}\frac{1}{2}(1-\beta(x,\bar{x}))\cdot s\leqslant\beta(y)\cdot s,
\end{equation*}
and so we have that $s_0<s$.

\begin{claim}\label{clm:case2-1-1}
If $i\in I_2$ and $v\in B_0$, then $d(v,\hat{\phi}_i(x))=0$. Hence, by~\ref{prop:likeregularity}, given $i\in I_2$, there is some $v\in B_0$ with $d(\hat{\phi}_i(x),v)=1$.
\end{claim}
\begin{proof}[Proof of Claim~\ref{clm:case2-1-1}]
Let $i\in I_2$, $v\in B_0$, and suppose that $d(v,\hat{\phi}_i(x))>0$. Then we may instead set $\hat{\phi}_i(z)=v$ and $j_i=1$ and observe that \ref{sub-1}--\ref{sub-3} still hold. As this increases $|I_1|$, this is a contradiction.
\renewcommand{\qedsymbol}{$\boxdot$}
\end{proof}
\renewcommand{\qedsymbol}{$\square$}

Given $v\in B$, let $\bar{w}_v=\beta(y)\cdot |\{i\in I_2:\hat{\phi}_i(y)=v\}|$. We remark that $\bar{w}_v\leqslant\hat{w}_v$, and
\begin{equation}\label{eq:y-space}
    \sum_{j\in[r-1]}\bar{w}_{b_j}\overset{\text{\ref{sub-1},\ref{sub-2}}}{\geqslant}\sum_{j\in[r-1]}\hat{w}_{a_j}-\beta(x)\cdot s.
\end{equation}
Let $B_y$ be the set of $v\in B$ for which $\bar{w}_v\geq 2$.

\begin{claim}\label{clm:case2-2-1}
If $i,i'\in I_2$ are such that $i\neq i'$, then $d(\hat{\phi}_{i'}(y),\hat{\phi}_i(x))=0$. Hence, $d(v,\hat{\phi}_i(x))=0$ whenever $i\in I_2$ and $v\in B_y$.
\end{claim}
\begin{proof}[Proof of Claim~\ref{clm:case2-2-1}]
Let $i,i'\in I_2$ be such that $i\neq i'$, and suppose that $d(\hat{\phi}_{i'}(y),\hat{\phi}_i(x))>0$. Let $v'=\hat{\phi}_{i'}(y)$. By Claim~\ref{clm:case2-1-1}, there is some $v\in B_0$ such that $d(\hat{\phi}_{i'}(x),v)=1$. Then, because $\beta(y)\geqslant\beta(z)$, we may instead set $\hat{\phi}_i(z)=v'$, $\hat{\phi}_{i'}(y)=v$, and $j_i=1$, increasing $|I_1|$, a contradiction.
\renewcommand{\qedsymbol}{$\boxdot$}
\end{proof}
\renewcommand{\qedsymbol}{$\square$}

\begin{claim}\label{clm:case2-3-1}
If $i\in I_2$, then $d(v,\hat{\phi}_i(x))>0$ for at least $\epsilon r$ many $v\in A_0$.
\end{claim}
\begin{proof}[Proof of Claim~\ref{clm:case2-3-1}]
Let $i\in I_2$. Suppose that there are fewer than $\epsilon r$ many $v\in A_0$ for which $d(v,\hat{\phi}_i(x))>0$. So, using Claim~\ref{clm:case2-1-1} and Claim~\ref{clm:case2-2-1}, $d(v,\hat{\phi}_i(x))=0$ for all but at most $\epsilon r$ many $v\in A_0\cup B_0\cup B_y$. Then, by~\itref{prop:likeregularity}, for all but at most $2\epsilon r$ many $v\in A_0\cup B_0\cup B_y$, we have $d(\hat{\phi}_i(x),v)=1$.

Let $j'$ be such that $\hat{\phi}_i(x)=a_{j'}$. Let $J$ be the set of $j\in[r-1]$ such that $(w_{a_j}-\hat{w}_{a_j})+(w_{b_j}-\hat{w}_{b_j})+\bar{w}_{b_j}\geqslant 5$. If $j\in J$, then either $a_j\in A_0$ or $b_j\in B_0\cup B_y$. So $d(j',j)=1$ for all but at most $2\epsilon r$ many $j\in J$, and hence $\sum_{j\in[r]\setminus\{j'\}}d(j',j)\geqslant|J|-3\epsilon r$. Noting that $(w_{a_j}-\hat{w}_{a_j})+(w_{b_j}-\hat{w}_{b_j})+\bar{w}_{b_j}\leqslant w_{a_j}+w_{b_j}=(1+\gamma+\alpha)\cdot s/r$ for any $j\in[r-1]$, we have
\begin{align*}
    \sum_{j\in[r]\setminus\{j'\}}d(j',j)&\cdot(1+\gamma+\alpha)\cdot s/r\geqslant|J|\cdot(1+\gamma+\alpha)\cdot s/r-3\epsilon(1+\gamma+\alpha)\cdot s\\
    &\geqslant\sum_{j\in[r-1]}[(w_{a_j}-\hat{w}_{a_j})+(w_{b_j}-\hat{w}_{b_j})+\bar{w}_{b_j}]-5r-9\epsilon s\\
    &\hspace{-0.15cm}\overset{\eqref{eq:1-large}}{\geqslant}\sum_{j\in[r-1]}w_{a_j}+\sum_{j\in[r-1]}(\bar{w}_{b_j}-\hat{w}_{a_j})+(\beta(x,\bar{x})+\alpha/2)\cdot s\\
    &\hspace{-0.15cm}\overset{\eqref{eq:y-space}}{\geqslant}\sum_{j\in[r-1]}w_{a_j}+(\beta(\bar{x})+\alpha/2)\cdot s>\sum_{j\in[r-1]}d(r,j)\cdot(1+\gamma+\alpha)\cdot s/r,
\end{align*}
contradicting \eqref{eq:r-maximal}.
\renewcommand{\qedsymbol}{$\boxdot$}
\end{proof}
\renewcommand{\qedsymbol}{$\square$}

Let $I_Y$ be the set of $i\in[s_0]$ such that there exist distinct $i_0,\ldots,i_\ell\in[s_0]$ with $i_0=i$, such that
\begin{itemize}
    \item $d(\hat{\phi}_{i_{k-1}}(x),\hat{\phi}_{i_k}(y))>0$ for $k\in[\ell]$, and
    \item $d(\hat{\phi}_{i_\ell}(x),v)>0$ for some $v\in B_0$.
\end{itemize}
We remark that $d(\hat{\phi}_i(x),v)=0$ whenever $i\notin I_Y,v\in B_0$, and also that $d(\hat{\phi}_i(x),\hat{\phi}_{i'}(y))=0$ whenever $i\notin I_Y,i'\in I_Y$.

Let $A_1$ be the set of $v\in A$ with $w_v-\sum_{i\in I_Y}\beta(\hat{\phi}_i^{-1}(v)\cap\{x\})\geqslant 1$, and let $B_1$ be the set of $v\in B$ with $w_v-\hat{w}_{v}+\sum_{i\in I_Y}\beta(\hat{\phi}_i^{-1}(v)\cap\{y\})\geqslant 1$.

\begin{claim}\label{clm:case2-4}
If $u\in A_1$ and $v\in B_1$, then $d(u,v)=0$.
\end{claim}

\begin{proof}[Proof of Claim~\ref{clm:case2-4}]
Let $u\in A_1$ and $v\in B_1$, and suppose for contradiction that $d(u,v)>0$.

If $\hat{\phi}_i(x)=u$ for some $i\notin I_Y$, then we must have $d(u,v')=0$ for every $v'\in B_0$, else $i\in I_Y$. So in particular, we would have $v\in B_1\setminus B_0$, and hence $\hat{\phi}_{i'}(y)=v$ for some $i'\in I_Y$. But then $d(\hat{\phi}_i(x),\hat{\phi}_{i'}(y))>0$ for some $i\notin I_Y$, $i'\in I_Y$, a contradiction.

Therefore, we may assume that $\sum_{i\in[s_0]\setminus I_Y}\beta(\hat{\phi}_i^{-1}(u)\cap\{x\})=0$, and hence
\begin{equation}\label{eq:u-space}
    w_u-\hat{w}_u+\sum_{i\in I_2}\beta(\hat{\phi}_i^{-1}(u)\cap\{z\})=w_u-\sum_{i\in[s_0]}\beta(\hat{\phi}_i^{-1}(u)\cap\{x\})\geqslant1.
\end{equation}

Set $\hat{\phi}_{s_0+1}(x),\hat{\phi}_{s_0+1}(z)=u$ and $\hat{\phi}_{s_0+1}(y)=v$, and set $j_{s_0+1}=2$.

Let $I_Z\subseteq I_2$ be a minimal set such that $\hat{\phi}_i(z)=u$ for $i\in I_Z$ and $w_u-\hat{w}_u+\sum_{i\in I_Z}\beta(\hat{\phi}_i^{-1}(u)\cap\{z\})\geqslant1$, noting that \eqref{eq:u-space} shows such a choice is possible. By minimality of $I_Z$, we have $\sum_{i\in I_Z}\beta(\hat{\phi}_i^{-1}(u)\cap\{z\})\leqslant 2$. Using Claim~\ref{clm:case2-3-1}, choose $u_i\in A_0\setminus\{u\}$ for $i\in I_Z$ with $d(u_i,u)>0$.

If $w_v-\hat{w}_v\geqslant\beta(y)$, then set $\ell=0$, $i_\ell=s_0+1$, and $v^\star=v$. Otherwise, we find there is some $i\in I_Y$ with $\hat{\phi}_i(y)=v$, and so let $i_0,\ldots,i_\ell\in[s_0]$ be distinct with $i_0=i$, and $v^\star\in B_0$ be such that $d(\hat{\phi}_{i_{k-1}}(x),\hat{\phi}_{i_k}(y))>0$ for $k\in[\ell]$ and $d(\hat{\phi}_{i_\ell}(x),v^\star)>0$. In either case, set $v_{i_{k-1}}=\hat{\phi}_{i_k}(y)$ for $k\in[\ell]$, and set $v_{i_{\ell}}=v^\star$.

Now, setting $\hat{\phi}_i(z)=u_i$ for $i\in I_Z$ and $\hat{\phi}_{i_k}(y)=v_{i_k}$ for $k\in\{0\}\cup[\ell]$ yields a contradiction to the maximality of $s_0$, proving the claim.
\renewcommand{\qedsymbol}{$\boxdot$}
\end{proof}
\renewcommand{\qedsymbol}{$\square$}

Let $J_A$ be the set of $j\in[r-1]$ with $a_j\in A_1$ and $J_B$ be the set of $j\in[r-1]$ with $b_j\in B_1$. By Claim~\ref{clm:case2-4}, $J_A$ and $J_B$ are disjoint. Let $j'\in J_B$ be such that $\sum_{j\in J_B}d(j',j)$ is maximised. So by Claim~\ref{clm:case2-4},
\begin{equation}\label{eq:jay-out}
    \sum_{j\in[r]\setminus\{j'\}}d(j',j)\geqslant|J_A|+\frac{1}{2}|J_B|-2\epsilon r.
\end{equation}
Also, because $\beta(y)\geqslant\beta(x)$,
\begin{align}
    \sum_{v\in B}\bigg(w_v-\hat{w}_{v}+&\sum_{i\in I_Y}\beta(\hat{\phi}_i^{-1}(v)\cap\{y\})\bigg)\overset{(\ref{eq:1-large})}{>}(\beta(x,\bar{x})+3\alpha/4)\cdot s+\beta(y)\cdot|I_Y| \nonumber\\
    &\geqslant2\beta(x)\cdot|I_Y|+(3\alpha/4)\cdot s=2\sum_{v\in A}\sum_{i\in I_Y}\beta(\hat{\phi}_i^{-1}(v)\cap\{x\})+(3\alpha/4)\cdot s. \label{eq:A-B-compared}
\end{align}
Therefore,
\begin{align*}
    \sum_{j\in[r]\setminus\{j'\}}d(j',j)\cdot&(1+\gamma+\alpha)\cdot s/r\overset{(\ref{eq:jay-out})}{\geqslant}|J_A|\cdot(1+\gamma+\alpha)\cdot s/r+\frac{1}{2}|J_B|\cdot(1+\gamma+\alpha)\cdot s/r-2\epsilon(1+\gamma+\alpha)\cdot s\\
    &\geqslant\sum_{v\in A}(w_v-\sum_{i\in I_Y}\beta(\hat{\phi}_i^{-1}(v)\cap\{x\}))+\frac{1}{2}\sum_{v\in B}\bigg(w_v-\hat{w}_{v}+\sum_{i\in I_Y}\beta(\hat{\phi}_i^{-1}(v)\cap\{y\})\bigg)-10\epsilon s\\
    &\overset{(\ref{eq:A-B-compared})}{>}\sum_{v\in A}w_v+(3\alpha/8)\cdot s-10\epsilon s>\sum_{j\in[r-1]}d(r,j)\cdot(1+\gamma+\alpha)\cdot s/r,
\end{align*}
contradicting \eqref{eq:r-maximal}.

\end{proof}

\subsubsection{Proof of Lemma~\ref{lm:distillation-case-3}}\label{sec:lmcase3}
In this section, we will prove Lemma~\ref{lm:distillation-case-3}. Similarly as in Section~\ref{sec:lmcase2}, we will outline our strategy in a simplified setting, along with a depiction in Figure~\ref{fig:technical-D}, so that this outline may guide the reader through the technical proof. For this, let $T$ again be a tree containing a vertex $t$ with only out-neighbours in $T$, such that $T-t$ consists of small components. Furthermore, let $G$ be a tournament with more vertices than $T$ into which we are attempting to find an embedding $\psi$ of $T$. We proceed initially with a very similar strategy to that described at the start of Section~\ref{sec:lmcase2}, as follows.

Let $v_t$ be a vertex in $G$ with maximal out-degree and set $\psi(t)=v_t$. Let $A=N^+(v_t)$ and $B=V(G)\setminus (A\cup \{v_t\})$. Just as before, we now aim to embed components of $T-t$ so that they can be attached correctly to $v_t$ but so that as many vertices as possible lie in $B$. For the trees relevant for Lemma~\ref{lm:distillation-case-2}, if we carefully maximised the number of components we embedded, then we covered enough vertices in $B$ that we were able to finish the embedding by embedding the remaining components into the unused vertices in $A$. The problem here is that Lemma~\ref{lm:distillation-case-3} covers trees for which this might not be possible. To see this, consider again $\hat{H}_1$ with vertex set $\{x,y,z\}$ and edge set $\{xy,zx\}$, and, for each component of $T-t$, map the out-neighbour of $t$ to $x$, and map the other vertices to $y$ or $z$ according to the direction of the first edge of their path from $t$ in the component as before. If all edges between $A$ and $B$ in $G$ are directed from $A$ to $B$ then the only vertices we can embed into $B$ are those mapped to $y$, and in the trees relevant to Lemma~\ref{lm:distillation-case-3} there may be few or even none of these! That is, it simply may not be the case that we can embed $T$ with $t$ embedded to $v_t$ as before.

Nevertheless, with $t$ embedded to $v_t$, we attempt to embed as many components as possible subject to a careful maximisation as before. In particular, we always have the vertex mapped to $x$ embedded into $A$ and vertices mapped to $z$ embedded into $B$. Previously, we then made a sequence of deductions that led us to find sets $A_1\subseteq A$ and $B_1\subseteq B$ such that almost all the edges of $G$ were directed from $B_1$ to $A_1$, and this led to a contradiction (see step 8 in Section~\ref{sec:lmcase2}). For Lemma~\ref{lm:distillation-case-3}, we again make a (here, simpler) sequence of deductions. Roughly, if $A_0$ and $B_0$ are the vertices in $A$ and $B$ respectively which are not in the image of any component we have embedded, then $G$ must have almost all the possible edges directed from $A_0$ into $B_0$ as well as certain other properties. The vertices in $B_0$ are then the vertices we struggled to cover when embedding components of the tree. However, if we pick a typical vertex $v_t'$ in $A_0$ and try instead to embed the tree starting with embedding $t$ to $v_t'$, then it is easy to cover vertices in $B_0$ with components in $T-t$ as most of these vertices are in the out-neighbourhood of $v_t'$. Even better, the vertices in $A_0\setminus \{v_t'\}$ can also be used relatively easily as there are many edges from $A_0$ to $B_0$ and, now the embedding of $t$ is changed, components of $T-t$ can be embedded with vertices mapped to $\{x,y\}$ embedded into $B_0$ and vertices mapped to $z$ embedded into $A_0$ (and for the trees covered by Lemma~\ref{lm:distillation-case-3} significantly many vertices are mapped to $z$). Of course, the remaining vertices of $G$ -- those which had components embedded to them -- may now be hard to cover, but, from our maximised component embedding and subsequent deductions, we will have information about how we can adjust the embedded components to attach them instead to the new image, $v_t'$, of $t$ while still using roughly the same vertices for that component. This is not always simple, and we allocate some additional vertices in $B_0$ to each embedded component before finding a new version of that component using the old vertices for that component and possibly also the additional vertices allocated from $B_0$. That we can allocate some additional vertices in this fashion relies on the fact that $G$ is larger than $T$, so we can afford not to use every vertex in $G$ in the final embedding.

\begin{figure}
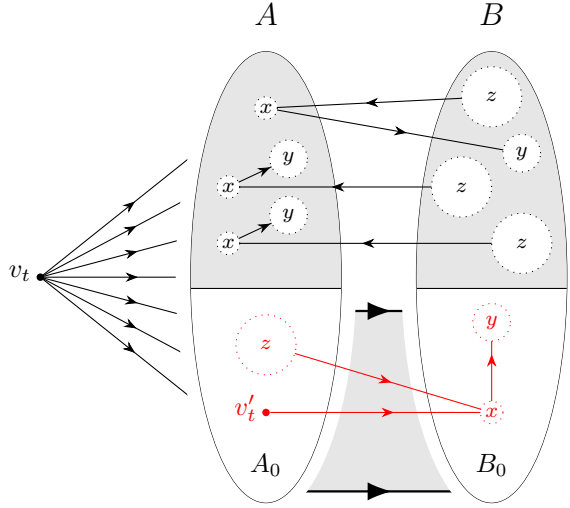

\DiagramTechnicalExampleD
\vspace{-0.75cm}
\caption{In our simplified set-up, for the case corresponding to Lemma~\ref{lm:distillation-case-3}, we may instead begin by aiming to embed vertices of $T-t$ mapped to $z$ in $B$. The component in the lower half of the diagram illustrates how leftover space in $A_0$ and $B_0$ may be used efficiently via a new embedding $v_t'$ for $t$.}\label{fig:technical-D}
\end{figure}

The portioning of the vertices in $B_0$ to allow the rearrangement and reattachment of the components requires quite some delicacy, resulting ultimately in the choice of the functions $\beta_1,\ldots,\beta_5$ in Lemma~\ref{lm:distillation-case-3}. This choice can be motivated, but only at some length and difficulty, so instead we concentrate on writing a proof that can be directly verified. As in Section~\ref{sec:lmcase2}, in the actual setting, with a tree $T$ with small core $T_0$ in mind, given a distillation of an average component, we find homomorphisms to a weighted looped digraph which will represent a regularity partition. We then pick from these homomorphisms randomly to get a random homomorphism. These homomorphisms again essentially allocate room for components of $T-V(T_0)$ in the regularity clusters corresponding to the weighted digraph $D$. 

To find these homomorphisms, we begin by selecting a vertex $j_t\in V(D)$ with at least the average amount of weight on its out-edges. By relabelling, we assume $j_t=r$. Similarly to the proof of Lemma~\ref{lm:distillation-case-2}, we again duplicate vertices of $D$ to get weighted vertex sets $A$ and $B$  representing the proportion of each vertex that is in the out- and in-neighbourhood of $j_t$ respectively, and aim to find homomorphisms which allocate plenty of space in $B$. To do this, we find maximally many homomorphisms from $H_{4,1}$ to $D$ satisfying certain conditions (\ref{sub-0-d2}--\ref{sub-3-d2} in the proof), and further optimising over two additional conditions (stated just after \ref{sub-3-d2}), to get homomorphisms $\phi_1,\ldots,\phi_{s_0}$. Assuming we do not have enough homomorphisms to easily find the remaining ones we wanted (see Claim~\ref{clm:J_1-not-too-small}), we make a sequence of deductions based on this assumption and the maximality (Claim~\ref{clm:no0to1} and Claim~\ref{clm:good-types}) about the weight distribution on $D$. Letting $A_0\subseteq A$ and $B_0\subseteq B$ be the sets of vertices with unallocated weight, an example deduction is that almost all of the weight on the edges of $D$ between $A_0$ and $B_0$ is on edges directed from $A_0$ to $B_0$ (for otherwise we would have found at least one more homomorphism). We then choose a new vertex $j_t'\in[r]$, with $a_{j_t'}\in A_0$, to use instead of $j_t=r$ by selecting it so that $a_{j_t'}$ has sufficient out-weight to certain vertices in the image of the homomorphisms we have found (just before Claim~\ref{clm:I1-and-I2}) -- this will allow us later to rearrange these homomorphisms to embed $x$ into the out-neighbourhood of $j_t'$ instead of $j_t$. While it may not necessarily be possible to ensure the desired conditions on $j_t'$ hold with respect to every homomorphism chosen so far, Claim~\ref{clm:I1-and-I2} confirms that there is a choice of $j_t'$ with the required conditions for rearrangement holding with respect to most of them -- by relabelling, these will be the homomorphisms $\phi_1,\ldots,\phi_{s_1}$. We add a dummy vertex $q$ with weight $\beta_1(q)=\max{\{\beta_x,\beta_z\}}$ to $H_{4,1}$, and then extend $\phi_1,\ldots,\phi_{s_1}$ to homomorphisms $\psi_1,\ldots,\psi_{s_1}$ also covering $q$ (see \ref{Q1}, \ref{Q2} and Claim~\ref{clm:s=s}). The role played by $q$ is to reserve additional weight in the out-neighbourhood of $j_t'$, which may be required for the desired rearrangement. Each $\psi_i$ then represents some allocated space, for which we then find homomorphisms $\phi_{i,j}$ from $H_{4,2},H_{4,3}$ or $H_{4,4}$ which use a proportion of this space (in either Claim~\ref{clm-forcase1} or Claim~\ref{clm-forcase2}) -- a proportion matching the space a tree must cover in the tournament following an eventual application of the lemma (that is, the proportion $|T|/|G|$). The homomorphisms $\phi_{i,j}$ potentially leave plenty of weight remaining on vertices in $A_0$ and $B_0$, but, as noted above in the simplified setting, these vertices are easiest to use for new homomorphisms as there is a lot of weight on edges from $j_t'$ to $B_0$ and from $B_0$ to $A_0$ in $D$. We therefore find more homomorphisms, this time from $H_{4,5}$, to use enough of this weight (that is, so that the proportion of the weight used is at least the matching proportion $|T|/|G|$). Finally, then, we have a collection of homomorphisms from which to pick our random homomorphism to complete the proof.

\begin{proof}[Proof of Lemma~\ref{lm:distillation-case-3}]
Let $\gamma=\max\{\beta_x,\beta_z\}$. Let $1/r\ll\epsilon\ll\alpha$. We remark that $\gamma\leqslant 1$, and we may also assume that $\alpha\leqslant 1$, so we have $(1+\gamma+\alpha)\leqslant3$.

Let $D$ be a complete looped digraph on vertex set $[r]$ with $\epsilon$-complete edge weights $d(e)$, $e\in E(D)$. We will find a random $(\phi,i(\phi))$ satisfying \itref{good-plan-1}-\itref{good-plan-3}. By relabelling, we can assume that 
\begin{equation}\label{eqn:jbig}
\sum_{j\in[r-1]}d(r,j)\geqslant(\textstyle\frac{1}{2}-2\epsilon)\cdot r.
\end{equation}

Take two new disjoint vertex sets $A=\{a_1,\ldots,a_{r-1}\}$ and $B=\{b_1,\ldots,b_{r-1}\}$. Let $\bar{D}$ be the weighted complete looped digraph on $A\cup B$ in which the edges $a_ib_j$, $a_ia_j$, $b_ib_j$ and $b_ia_j$ have weight $d(i,j)$.

Let $s$ be such that $1/s\ll1/r$. For each $i\in [r-1]$, let $w_{a_i}=d(r,i)\cdot (1+\gamma+\alpha)\cdot s/r$ and $w_{b_i}=(1-d(r,i))(1+\gamma+\alpha)\cdot s/r$. Note that
\begin{equation}\label{eqn-wvsum}
\sum_{v\in A\cup B}w_v\geqslant(1+\gamma+\alpha/2)\cdot s,
\end{equation}
and
\begin{equation}\label{eqn:wvbig}
\sum_{v\in A}w_v=\sum_{i\in [r-1]}d(r,i)\cdot(1+\gamma+\alpha)\cdot s/r\overset{\eqref{eqn:jbig}}{\geq}
\frac{1}{2}(1+\gamma+\alpha/2)\cdot s.
\end{equation}

Let $s_0\leqslant s$ be the largest integer for which there exist homomorphisms $\phi_1,\ldots,\phi_{s_0}:H_{4,1}\to \bar{D}$ and indicators $j_1,\ldots,j_{s_0}\in [2]$ such that the following properties hold.
\stepcounter{propcounter}
\begin{enumerate}[label = {\bfseries \Alph{propcounter}\arabic{enumi}}]
\item \label{sub-0-d2} For each $i\in [s_0]$, $d(\phi_i(y),\phi_i(z))+d(\phi_i(z),\phi_i(y))>0$.
\item\label{sub-1-d2} For each $i\in[s_0]$ with $j_i=1$, we have 
    $\phi_i(x)\in A$, 
    $\phi_i(y)\in B$, and
    $\phi_i(z)\in B$.
\item\label{sub-2-d2}
For each $i\in[s_0]$ with $j_i=2$, we have
    $\phi_i(x)\in A$, 
    $\phi_i(y)\in A$, and
    $\phi_i(z)\in B$.
    \item\label{sub-3-d2} For each $v\in A\cup B$, $\sum_{i\in [s_0]}\beta_1(\phi_i^{-1}(v))\leq w_v$.
\end{enumerate}
Given $s_0$, take $\phi_1,\ldots\phi_{s_0}$ and $j_1,\ldots,j_{s_0}$ such that the number of $i\in[s_0]$ with $j_i=1$ is maximised. For each $v\in A\cup B$, let $\tilde{w}_v=\sum_{i\in [s_0]}\beta_1(\phi_i^{-1}(v))$, so that, by \ref{sub-3-d2}, we have  $\tilde{w}_v\leqslant w_v$.

Note that
\begin{equation}\label{eqn-sumAB-wbar}
\sum_{v\in A\cup B}\tilde{w}_v=\sum_{v\in A\cup B}\sum_{i\in [s_0]}\beta_1({\phi}_i^{-1}(v))=
\sum_{i\in [s_0]}\beta_1({\phi}_i^{-1}(A\cup B))=s_0,
\end{equation}
and
\begin{equation}\label{eq:sumA-wbar}
    \sum_{v\in A}\tilde{w}_v=\sum_{v\in A}\sum_{i\in[s_0]}\beta_1(\phi_i^{-1}(v))=\sum_{i\in[s_0]}\beta_1(\phi_i^{-1}(A))\overset{\text{\ref{sub-1-d2},\ref{sub-2-d2}}}{\leqslant}\sum_{i\in[s_0]}\beta_1(x,y)=\beta_1(x,y)\cdot s_0.
\end{equation}

Let $A_0$ be the set of $v\in A$ with  $w_v-\tilde{w}_v\geqslant1$. Let $B_0$ be the set of $v\in B$ with  $w_v-\tilde{w}_v\geqslant1$. Subject to the value of $s_0$, and the value of $|\{i\in[s_0]:j_i=1\}|$, assume that
\begin{equation}\label{eq:minimise-this-please}
    |\{(i,v):i\in [s_0],v\in A_0,d(\phi_i(y),v)>0\}|
\end{equation}
is minimised.

Suppose that $s_0<s$, for otherwise we are done by letting $i(\phi)=1$ and picking $\phi$ from $\{\phi_i:i\in [s]\}$ uniformly at random.
Note that, as $1-\gamma\leq 1-\beta_x=\beta_y+\beta_u+\beta_z+\beta_w\leqslant 2(\beta_z+\beta_w)$, we have
\begin{align}
    \frac{1+\gamma}{2}-\beta_1(x,y)&= \frac{1+\gamma}{2}-(1-\beta_1(z))=\min\{\beta_w+\beta_z,1-\beta_z\}-\frac{1-\gamma}{2}\nonumber \\
    &\geq \min{\left\{\frac{1-\gamma}{2},1-\gamma\right\}}-\frac{1-\gamma}{2}= 0.
\label{eqn-g-x-y}
\end{align}
Therefore,
\begin{align}
    \sum_{v\in A_0}(w_v-\tilde{w}_v)&\geqslant \sum_{v\in A}(w_v-\tilde{w}_v)-r\overset{\eqref{eqn:wvbig},\eqref{eq:sumA-wbar}}{\geqslant}
    \frac{1}{2}(1+\gamma+\alpha/4)\cdot s-\beta_1(x,y)\cdot s_0
    \nonumber\\
    &=
    \left(\frac{1+\gamma}{2}\right)\cdot(s-s_0)+\left(\frac{1+\gamma}{2}-\beta_1(x,y)\right)\cdot s_0+(\alpha/8)\cdot s
    \nonumber\\
    &\hspace{-0.15cm}\overset{\eqref{eqn-g-x-y}}{\geqslant}
    \left(\frac{1+\gamma}{2}\right)\cdot(s-s_0)+(\alpha/8)\cdot s, \label{eq:J_0-not-too-small}
\end{align}
and hence, $|A_0|\geqslant(\alpha/32)\cdot r$.
In addition, we have
\begin{align}
\sum_{v\in A_0\cup B_0}(w_v-\tilde{w}_v)&\geqslant \sum_{v\in A\cup B}(w_v-\tilde{w}_v)-2r\overset{\eqref{eqn-wvsum},\eqref{eqn-sumAB-wbar}}{\geqslant}(1+\gamma+\alpha/2)\cdot s-s_0-2r
\nonumber \\
&\geqslant (s-s_0)+(\gamma+\alpha/4)\cdot s. \label{eq:A_0+B_0}
\end{align}

We now show there are no pairs from $B_0$ to $A_0$ with positive weight.

\begin{claim}\label{clm:no0to1}
If $u\in A_0$ and $v\in B_0$, then $d(v,u)=0$. \end{claim}
\begin{proof}[Proof of Claim~\ref{clm:no0to1}]
If $u\in A_0$, $v\in B_0$, are such that $d(v,u)>0$, then we may\linebreak set $\phi_{s_0+1}(x),\phi_{s_0+1}(y)=u$, $\phi_{s_0+1}(z)=v$, and $i_{s_0+1}=2$, contradicting the maximality of $s_0$.
\renewcommand{\qedsymbol}{$\boxdot$}
\end{proof}
\renewcommand{\qedsymbol}{$\square$}

We now show that we are done unless $\sum_{v\in B_0}(w_v-\tilde{w}_v)$ is not too small.

\begin{claim}\label{clm:J_1-not-too-small}
Either there exists a random $(\phi,i(\phi))$ satisfying \itref{good-plan-1}-\itref{good-plan-3}, or
\begin{equation}\label{eq:case3-J_1}
    \sum_{v\in B_0}(w_v-\tilde{w}_v)\geqslant(\gamma+\alpha/16)\cdot s+\beta_1(y)\cdot (s-s_0).
\end{equation}
\end{claim}
\begin{proof}[Proof of Claim~\ref{clm:J_1-not-too-small}]

Suppose that \eqref{eq:case3-J_1} does not hold. If $\sum_{v\in B_0}(w_v-\tilde{w}_v)\leqslant(\gamma+\alpha/8)\cdot s$, then set $s_1=s_0$. Otherwise, let $s_1\in[s]\setminus[s_0]$ be maximal such that $\sum_{v\in B_0}(w_v-\tilde{w}_v)\geqslant(\gamma+\alpha/16)\cdot s+\beta_1(y)\cdot(s_1-s_0)$. Note that $s_1<s$, else \eqref{eq:case3-J_1} holds. By considering the cases $s_1=s_0$ and $s_1>s_0$ separately (using the maximality of $s_1$ in the latter), we deduce that
\begin{equation}\label{eq:remaining-B0}
    \sum_{v\in B_0}(w_v-\tilde{w}_v)\leqslant(\gamma+\alpha/8)\cdot s+\beta_1(y)\cdot(s_1-s_0),
\end{equation}
and hence,
\begin{equation*}
    \sum_{v\in A_0}(w_v-\tilde{w}_v)\overset{\eqref{eq:A_0+B_0},\eqref{eq:remaining-B0}}{\geqslant}\beta_1(x,z)\cdot(s_1-s_0)+(s-s_1)+(\alpha/8)\cdot s.
\end{equation*}
Therefore, using Claim~\ref{clm:no0to1}, we may greedily choose homomorphisms $\phi_{s_0+1},\ldots,\phi_{s}:H_{4,1}\to\bar{D}$ such that the following properties hold.
\stepcounter{propcounter}
\begin{enumerate}[label = {\bfseries \Alph{propcounter}\arabic{enumi}}]
\item\label{sub-1-d3} For each $i\in[s_1]\setminus [s_0]$, we have 
    ${\phi}_i(x)\in A$, 
    ${\phi}_i(y)\in B$, and
    ${\phi}_i(z)\in A$.
\item\label{sub-2-d3}
For each $i\in[s]\setminus[s_1]$ we have
    ${\phi}_i(x)\in A$, 
    ${\phi}_i(y)\in A$, and
    ${\phi}_i(z)\in A$.
\item\label{sub-3-d3} For each $v\in A\cup B$, $\sum_{i\in [s]}\beta({\phi}_i^{-1}(v))\leq w_v$.
\end{enumerate}
Thus, by defining $(\phi,i(\phi))$ by sampling $\phi$ from $\phi_1,\ldots,\phi_s$ uniformly at random (identifying the result as a map $V(H_{4,1})\to V(D)$ in the natural way) and setting $i(\phi)=1$, we obtain a random $(\phi,i(\phi))$ satisfying \itref{good-plan-1}-\itref{good-plan-3}.
\renewcommand{\qedsymbol}{$\boxdot$}
\end{proof}
\renewcommand{\qedsymbol}{$\square$}

Thus, we may now assume that (\ref{eq:case3-J_1}) holds, and hence also $|B_0|\geqslant(\alpha/64)\cdot r$.
Next we show that each $i\in[s_0]$ satisfies (at least) one of two properties, \itref{good-type-1} or \itref{good-type-2}.

\begin{claim}\label{clm:good-types}
For each $i\in [s_0]$, at least one of the following holds.
\stepcounter{propcounter}
\begin{enumerate}[label = {\bfseries \emph{\Alph{propcounter}\arabic{enumi}}}]
\item $d(\phi_i(z),v)=0$ for every $v\in A_0$.\label{good-type-1}
\item $d(v,\phi_i(x))=0$ for every $v\in B_0$, $d(\phi_i(y),v)=0$ for every $v\in A_0$, and $j_i=1$.\label{good-type-2}
\end{enumerate}
\end{claim}
\begin{proof}[Proof of Claim~\ref{clm:good-types}]

Let $i\in[s_0]$. Assume there is some $u\in A_0$ with $d(\phi_i(z),u)>0$, for otherwise \itref{good-type-1} holds.

Now, if there is some $v\in B_0$ with $d(v,\phi_i(x))>0$, set $\phi_{s_0+1}(x),\phi_{s_0+1}(y)=u$ and $\phi_{s_0+1}(z)=\phi_i(z)$, and then switch $\phi_i(z)=v$. This contradicts the maximality of $s_0$.

Thus, $d(v,\phi_i(x))=0$ for every $v\in B_0$. As $|B_0|\geqslant (\alpha/64)\cdot r$, we may now fix some $v\in B_0$ with $d(\phi_i(x),v)>0$ and $d(v,\phi_i(z))+d(\phi_i(z),v)>0$, by \itref{prop:likeregularity}. Then we must have that $j_i=1$, else we could switch $\phi_i(y)=v$ and $j_i=1$, contradicting the maximality of $|\{i\in[s_0]:j_i=1\}|$.

Now, note that, by Claim~\ref{clm:no0to1}, $|\{u'\in A_0:d(v,u')>0\}|=0$. Therefore, if $|\{u'\in A_0:d(\phi_i(y),u')>0\}|>0$, we could switch $\phi_i(y)=v$ to reduce $\sum_{i'\in[s_0]}|\{u'\in A_0:d(\phi_{i'}(y),u')>0\}|$ while leaving $A_0$ unmodified, a contradiction to the minimisation of \eqref{eq:minimise-this-please}. Thus, $|\{u'\in A_0:d(\phi_i(y),u')>0\}|=0$, and hence \itref{good-type-2} holds.
\renewcommand{\qedsymbol}{$\boxdot$}
\end{proof}
\renewcommand{\qedsymbol}{$\square$}

Given $u\in A_0$, let $I_1(u)$ be the set of $i\in[s_0]$ which satisfy \itref{good-type-1} and are such that $d(u,\phi_i(z))=1$ and $\phi_i^{-1}(u)=\emptyset$, and let $I_2(u)$ be the set of $i\in[s_0]\setminus I_1(u)$ which satisfy \itref{good-type-2} and are such that $d(u,\phi_i(y))=1$ and $\phi_i^{-1}(u)=\emptyset$. Pick $j_t'\in\{j\in[r-1]:a_j\in A_0\}$ such that $|I_1(a_{j_t'})|+|I_2(a_{j_t'})|$ is maximised, and set $I_1=I_1(a_{j_t'})$, $I_2=I_2(a_{j_t'})$. By relabelling, we may assume that $I_1\cup I_2=[s_1]$ for some $s_1\leqslant s_0$.

\begin{claim}\label{clm:I1-and-I2}
$s_1\geqslant(1-\sqrt{\epsilon})s_0$.
\end{claim}
\begin{proof}[Proof of Claim~\ref{clm:I1-and-I2}]

If $i\in[s_0]$ satisfies \itref{good-type-1}, then, by \itref{prop:likeregularity}, $d(u,\phi_i(z))=1$ holds for all but at most $\epsilon r$ many $u\in A_0$. Similarly, if $i\in[s_0]$ satisfies \itref{good-type-2}, then $d(u,\phi_i(y))=1$ holds for all but at most $\epsilon r$ many $u\in A_0$. In addition, for each $i\in[s_0]$, we have $\phi_i^{-1}(u)\neq\emptyset$ for at most three $u\in A_0$. Therefore, for every $i\in[s_0]$, we have $i\in I_1(u)\cup I_2(u)$ for all but at most $2\epsilon r+3$ many $u\in A_0$. In particular,
\begin{equation*}
    \sum_{u\in A_0}(|I_1(u)|+|I_2(u)|)\geqslant s_0\cdot(|A_0|-3\epsilon r).
\end{equation*}
Noting that $|I_1|+|I_2|\geqslant\frac{1}{|A_0|}\sum_{u\in A_0}(|I_1(u)|+|I_2(u)|)$ and $|A_0|\geqslant(\alpha/32)\cdot r$, we deduce $s_0-|I_1|-|I_2|\leqslant\sqrt{\epsilon}s_0$, and hence the claim.
\renewcommand{\qedsymbol}{$\boxdot$}
\end{proof}
\renewcommand{\qedsymbol}{$\square$}

Let $B_1\subseteq B_0$ be the subset of vertices $v\in B_0$ with $d(a_{j_t'},v)=1$, and note, by Claim~\ref{clm:no0to1} and \itref{prop:likeregularity}, that $|B_0\setminus B_1|\leqslant\epsilon r$.
Let $q$ be a new vertex disjoint from $V(H_{4,1})$ and add it to $H_{4,1}$ to get $H_{4,1}'$. Let $\beta_1(q)=\gamma$. Let $s_1'\leqslant s_1$ be maximal for which, for each $i\in[s_1']$, we can define $\psi_i(q)\in B_1$ such that the following hold.
\stepcounter{propcounter}
\begin{enumerate}[label = {\bfseries \Alph{propcounter}\arabic{enumi}}]
\item For each $i\in[s_1']$ and $u\in \{x,y,z\}$, $d(\psi_i(q),\phi_i(u))+d(\phi_i(u),\psi_i(q))>0$.\label{Q1}
\item For each $v\in B$, \label{Q2}
\begin{equation}\label{eqn-wvbarwv}
\tilde{w}_v+\gamma\cdot |\{i\in [s_1']:\psi_i(q)=v\}|\leq w_v.
\end{equation}
\end{enumerate}
Let $\psi_i:H_{4,1}'\to \bar{D}$ be defined by this choice of $\psi_i(q)$ and by setting $\psi_i(v)=\phi_i(v)$ for each $v\in V(H_{4,1})$.

\begin{claim}\label{clm:s=s}
$s_1'=s_1$.
\end{claim}
\begin{proof}[Proof of Claim~\ref{clm:s=s}]
Suppose for contradiction that $s_1'<s_1$. Let $B_1'$ be the set of $v\in B_1$ such that $d(v,\phi_{s_1'+1}(u))+d(\phi_{s_1'+1}(u),v)>0$ for every $u\in\{x,y,z\}$. Because the edge weights $d(e)$, $e\in E(D)$, are $\epsilon$-complete and $|B_0\setminus B_1|\leqslant\epsilon r$, we have $|B_0\setminus B_1'|\leqslant 4\epsilon r$. Therefore, we have
\begin{align*}
    \sum_{v\in B_1'}(w_v-\tilde{w}_v&-\gamma\cdot|\{i\in[s_1']:\psi_i(q)=v\}|)\geqslant\sum_{v\in B_0}(w_v-\tilde{w}_v)-\sum_{v\in B_0\setminus B_1'}w_v-\gamma\cdot s_1'\\
    &\overset{\eqref{eq:case3-J_1}}{\geqslant}(\gamma+\alpha/16)\cdot s-12\epsilon\cdot s-\gamma\cdot s_1'\geqslant(\alpha/32)\cdot s\geqslant\gamma\cdot r
\end{align*}
Thus, there is some $v\in B_1'$ such that $\tilde{w}_v+\gamma\cdot|\{i\in[s_1']:\psi_i(q)=v\}|+\gamma\leqslant w_v$. But then setting $\psi_{s_1'+1}(q)=v$ contradicts the maximality of $s_1'$.
\renewcommand{\qedsymbol}{$\boxdot$}
\end{proof}
\renewcommand{\qedsymbol}{$\square$}
Let $m$ satisfy $\epsilon\ll1/m\ll \alpha$.
\begin{claim} \label{clm-forcase1}
For each $i\in I_1$, there are homomorphisms $\phi_{i,j}$ and indicators $k_{i,j}\in \{2,3\}$, $j\in [m-1]$, such that, for each $j\in [m-1]$, $\phi_{i,j}$ is a homomorphism from $H_{4,k_{i,j}}$ to $\bar{D}[\psi_i(V(H_{4,1}'))]$, and the following hold.
\stepcounter{propcounter}
\begin{enumerate}[label = {\emph{\bfseries \Alph{propcounter}\arabic{enumi}}}]
    \item \label{prop-x-okay-case1}$d(a_{j_t'},\phi_{i,j}(x))=1$.
    \item For each $v\in \psi_i(V(H_{4,1}'))$, $\beta_1(\psi^{-1}_i(v))\geq \sum_{j\in [m-1]}\beta_{k_{i,j}}(\phi_{i,j}^{-1}(v))/m$.\label{prop-fitsinside-case1}
\end{enumerate}
\end{claim}
\begin{proof}[Proof of Claim~\ref{clm-forcase1}]
Fixing $i\in I_1$, for each $j'=1,\ldots,m-1$ in turn, choose a\linebreak homomorphism $\phi_{i,j'}$ from $H_{4,2}$ or $H_{4,3}$ to $\bar{D}[\psi_i(V(H_{4,1}'))]$ such that the following properties hold.
\begin{enumerate}[label =(\roman{enumi})]
    \item $\beta_1(\psi_i^{-1}(v))\geqslant\sum_{j\in[j']}\beta_{k_{i,j}}(\phi_{i,j}^{-1}(v))/m$ for each $v\in\psi_i(H_{4,1}')$.
    \item $\phi_{i,j'}(x)\in \{\psi_i(z),\psi_i(q)\}$.
    \item
    \begin{enumerate}[label=(\alph{enumii})]
        \item $\phi_{i,j'}(y)\in\{ \psi_i(x),\psi_i(y)\}$ if $\phi_{i,j'}$ is a homomorphism from $H_{4,2}$.
        \item $\phi_{i,j'}(z)\in\{ \psi_i(x),\psi_i(y)\}$ if $\phi_{i,j'}$ is a homomorphism from $H_{4,3}$.
    \end{enumerate}
\end{enumerate}
Then \itref{prop-x-okay-case1} holds, through the definition of either $I_1$ (if $\phi_{i,j'}(x)=\psi_i(z)$) or $B_1$ (if $\phi_{i,j'}(x)=\psi_i(q)$).

Note that this is possible because, as
\begin{equation*}
\beta_1(z,q)=\min{\{\beta_w+\beta_z,1-\beta_z\}}+\max{\{\beta_x,\beta_z\}}\geq\min{\{\beta_w+\beta_z+\beta_x,1\}}=\beta_2(x)=\beta_3(x),
\end{equation*}
there is enough room in $\{\psi_i(z),\psi_i(q)\}$ for $\phi_{i,j'}(x)$, and, as
\[
\beta_1(x,y)=1-\beta_1(z)
\geq 1-(\beta_w+\beta_z)= \beta_x+\beta_u+\beta_y\geq \beta_2(y)=\beta_3(z).
\]
there is enough room in $\{\psi_i(x),\psi_i(y)\}$ for $\phi_{i,j'}(y)$.
We also use that there is weight in at least one direction between any pair from $\{\psi_i(z),\psi_i(q)\}$ and $\{\psi_i(x),\psi_i(y)\}$, where the direction of the edge gives whether we embed $H_{4,2}$ or $H_{4,3}$.
\renewcommand{\qedsymbol}{$\boxdot$}
\end{proof}
\renewcommand{\qedsymbol}{$\square$}

\begin{claim}\label{clm-forcase2}
For each $i\in I_2$, there are homomorphisms $\phi_{i,j}$, $j\in [m-1]$, such that, for each $j\in [m-1]$, $\phi_{i,j}$ is a homomorphism from  $H_{4,4}$ to $\bar{D}[\psi_i(V(H_{4,1}'))]$, and the following holds.
\stepcounter{propcounter}
\begin{enumerate}[label = {\emph{\bfseries \Alph{propcounter}\arabic{enumi}}}]
    \item \label{prop-x-okay-case2}$d(a_{j_t'},\phi_{i,j}(x))=1$.
    \item For each $v\in\psi_i(V(H_{4,1}'))$, $\beta_1(\psi^{-1}_i(v))\geq \sum_{j\in [m-1]}\beta_4(\phi_{i,j}^{-1}(v))/m$. \label{prop-fitsinside-case2}
\end{enumerate}
\end{claim}
\begin{proof}[Proof of Claim~\ref{clm-forcase2}]
Fixing $i\in I_2$, for each $j'=1,\ldots,m-1$ in turn, choose a homomorphism $\phi_{i,j'}$ from $H_{4,4}$ to $\bar{D}[\psi_i(V(H_{4,1}'))]$ such that 
$\beta_1(\psi^{-1}_i(v))\geq \sum_{j\in [j']}\beta_4(\phi_{i,j}^{-1}(v))/m$ for each $v\in \psi_i(H_{4,1}')$, $\phi_{i,j'}(x)\in \{\psi_i(y),\psi_i(q)\}$, $\phi_{i,j'}(z)=\psi_i(x)$ and $\phi_{i,j'}(w)=\psi_i(z)$. Then \itref{prop-x-okay-case2} holds, through the definition of either $I_2$ (if $\phi_{i,j'}(x)=\psi_i(y)$) or $B_1$ (if $\phi_{i,j'}(x)=\psi_i(q)$).

Note that this is possible using the following. As $\beta_1(y)+\beta_1(q)=\beta_1(y)+\gamma=\max\{\beta_x+\beta_y,\gamma\}=\max\{\beta_x+\beta_y,\beta_z\}\geq \beta_4(x)$,  there is enough room in $\{\psi_i(y),\psi_i(q)\}$ for $\phi_{i,j'}(x)$. As $\beta_1(z)\geq\min\{\beta_w,1-\beta_z\}\geq \beta_w=\beta_4(w)$, there is enough room in $\psi_i(z)$ for $\phi_{i,j'}(w)$. Finally, recall from \eqref{eqn-beta1x-beta4y}, that $\beta_4(z)\leq \beta_1(x)$, so there is enough room in $\psi_i(x)$ for $\phi_{i,j'}(z)$. Because $d(\psi_i(x),\psi_i(y)),d(\psi_i(z),\psi_i(x))>0$ (as $\psi_i:H_{4,1}'\to\bar{D}$ is a homomorphism) and $d(\psi_i(x),\psi_i(q))>0$ (by~\itref{good-type-2} and~\ref{Q1}), we also have that $\phi_{i,j'}:H_{4,4}\to\bar{D}[\psi_i(V(H_{4,1}'))]$ is a homomorphism, as required. 
\renewcommand{\qedsymbol}{$\boxdot$}
\end{proof}
\renewcommand{\qedsymbol}{$\square$}

For each $i\in I_2$ and $j\in [m-1]$, let $k_{i,j}=4$. For each $v\in A\cup B$, let 
\begin{align}
\hat{w}_v&=\frac{1}{m}\sum_{i\in [s_1]}\sum_{j\in [m-1]}\beta_{k_{i,j}}(\phi_{i,j}^{-1}(v))
\overset{\text{\itref{prop-fitsinside-case1},\itref{prop-fitsinside-case2}}}{\leq} \sum_{i\in [s_1]}\beta_1(\psi_i^{-1}(v))\nonumber\\
&=\sum_{i\in [s_1]}\beta_1(\phi_i^{-1}(v))+\gamma\cdot |\{i\in [s_1]:\psi_i(q)=v\}|=\tilde{w}_v+\gamma\cdot |\{i\in [s_1]:\psi_i(q)=v\}|\overset{\eqref{eqn-wvbarwv}}{\leq} w_v.\label{eqn-hatwvwv}
\end{align}
We note that
\begin{align}
\sum_{v\in A_0\cup B_1}(w_v-\hat{w}_v)&\geqslant \sum_{v\in A_0\cup B_0}(w_v-\hat{w}_v)-3\epsilon\cdot s \nonumber\\
&\overset{\eqref{eqn-hatwvwv}}{\geqslant} 
\sum_{v\in A_0\cup B_0}(w_v-\tilde{w}_v)-3\epsilon \cdot s-\sum_{v\in A_0\cup B_0} \gamma\cdot |\{i\in [s_1]:\psi_i(q)=v\}|\nonumber
\\
&\overset{\eqref{eq:A_0+B_0}}{\geq} s-s_0+(\gamma+\alpha/8)\cdot s-\gamma\cdot s_1\overset{\text{Claim~\ref{clm:I1-and-I2}}}{\geq}
(1+\gamma)\cdot(s-s_1)+(\alpha/16)\cdot s. \label{eq:hatJAJB-not-too-small}
\end{align}
Furthermore,
\begin{align}
\sum_{v\in B_1}(w_v-\hat{w}_v)&\geqslant\sum_{v\in B_0}(w_v-\hat{w}_v)-3\epsilon\cdot s\overset{\eqref{eqn-hatwvwv}}{\geqslant} 
\sum_{v\in B_0}(w_v-\tilde{w}_v)-3\epsilon\cdot s-\sum_{v\in B_0} \gamma\cdot |\{i\in [s_1]:\psi_i(q)=v\}|\nonumber
\\
&\overset{\eqref{eq:case3-J_1}}{\geq} \beta_1(y)\cdot(s-s_0)+(\gamma+\alpha/32)\cdot s-\gamma\cdot s_1\overset{\text{Claim~\ref{clm:I1-and-I2}}}{\geq}
(\gamma+\beta_1(y))\cdot(s-s_1)+(\alpha/64) s. \label{eq:hatJB-not-too-small}
\end{align}

Take a maximal set $J\subseteq([s]\times[m])\setminus([s_1]\times[m-1])$ for which there are homomorphisms $\phi_{i,j}:H_{4,5}\to \bar{D}[A_0\cup B_1]$, $(i,j)\in J$ such that the following hold.
\stepcounter{propcounter}
\begin{enumerate}[label = {\bfseries \Alph{propcounter}\arabic{enumi}}]
\item For each $(i,j)\in J$, $\phi_{i,j}(x)\in B_1$.
\item For each $v\in A\cup B$, $\sum_{(i,j)\in J}\beta_5(\phi_{i,j}^{-1}(v))/m\leq w_v-\hat{w}_v$.
\end{enumerate}
Subject this choice of $J$, maximise
\begin{equation}\label{eqn-tomax}
|\{(i,j)\in J: \phi_{i,j}(z)\in A_0\}|.
\end{equation}

\begin{claim}\label{clm:J}
$J=([s]\times[m])\setminus([s_1]\times[m-1])$.
\end{claim}
\begin{proof}[Proof of Claim~\ref{clm:J}]
Suppose, for contradiction, that there is some $(i',j')\in  (([s]\times[m])\setminus([s_1]\times[m-1]))\setminus J$. We must then have, for each $v\in B_1$, that $\sum_{(i,j)\in J}\beta_5(\phi_{i,j}^{-1}(v))/m\geq w_v-\hat{w}_v-1/m$, for otherwise we can take the homomorphism $\phi_{i',j'}$ sending $H_{4,5}$ to $v$.

Therefore, we have
\begin{equation*}
\sum_{v\in A_0}\bigg(w_v-\hat{w}_v-\textstyle\frac{1}{m}\displaystyle-\sum_{(i,j)\in J}\beta_5(\phi_{i,j}^{-1}(v))/m\bigg)\overset{\eqref{eq:hatJAJB-not-too-small}}{\geqslant}\gamma\cdot (s-s_1)+(\alpha/16)\cdot s-|A_0|/m-|B_1|/m\geqslant(\alpha/32)\cdot s.
\end{equation*}
Thus, there must be at least $2\eps r$ vertices $v\in A_0$ with $\sum_{(i,j)\in J}\beta_5(\phi_{i,j}^{-1}(v))/m\leq w_v-\hat{w}_v-1/m$. Therefore, by Claim~\ref{clm:no0to1}, if there is some $(i,j)\in J$ with $\phi_{i,j}(y)\notin A_0$, then we could move $\phi_{i,j}(y)$ into $A_0$, and increase the value of \eqref{eqn-tomax}. Thus, we must have $\phi_{i,j}(y)\in A_0$ for each $(i,j)\in J\setminus([s_1]\times[m-1])$.

Therefore, using that $\beta_5(x)\leqslant\gamma+\beta_1(y)$ and $|J|\leqslant(s-s_1)m+s$,
\begin{align*}
0&\geqslant\sum_{v\in B_1}\bigg(w_v-\hat{w}_v-\textstyle\frac{1}{m}\displaystyle-\sum_{(i,j)\in J}\beta_5(\phi_{i,j}^{-1}(v))/m\bigg)\geqslant\sum_{v\in B_1}(w_v-\hat{w}_v)-|B_1|/m-((s-s_1)m+s)\cdot\beta_5(x)/m\\
&\geqslant\sum_{v\in B_1}(w_v-\hat{w}_v)-(s-s_1)\cdot(\gamma+\beta_1(y))-s/m-|B_1|/m\overset{\eqref{eq:hatJB-not-too-small}}{>}0,
\end{align*}
a contradiction. Therefore, $J=([s]\times[m])\setminus([s_1]\times[m-1])$.
\renewcommand{\qedsymbol}{$\boxdot$}
\end{proof}
\renewcommand{\qedsymbol}{$\square$}

For each $(i,j)\in ([s]\times[m])\setminus([s_1]\times [m-1])$, let $k_{i,j}=5$. 
Select $(\phi,i(\phi))$ uniformly at random from $(\phi_{i,j},k_{i,j})$, $(i,j)\in [s]\times [m]$.
\end{proof}

\subsection{Proof of Theorem~\ref{thm:extending-distillation}}\label{sect:extending-new}

We are now ready to complete this section by proving Theorem~\ref{thm:extending-distillation}. To give a brief overview of this proof, we again turn to our simplified situation: assume we are trying to embed a tree $T$ in a tournament $G$ and suppose we have $t\in V(T)$ so that $T-t$ consists of small components. Unlike in Sections~\ref{sec:lmcase2} and~\ref{sec:lmcase3}, $t$ can have both in- and out-neighbours in $T$. Let $T^+$ and $T^-$ be the trees covering the edges of $T$, intersecting only on $t$, so that $t$ has only out-neighbours in $T^+$ and only in-neighbours in $T^-$. As $G$ is a tournament with distinctly more vertices than $T$, each vertex $v\in V(G)$ either has enough out-neighbours in $G$ that we can embed the components of $T^+-t$ greedily into $N^+_G(v)$ or the components of $T^--t$ greedily into $N^-_G(v)$. If we partition $V(G)=V^+\cup V^-$ so that the former holds for vertices in $V^+$ and the latter holds for vertices in $V^-$, then, from this partition, either $G[V^+]$ will be large enough to embed $T^-$ (using our previous methods) or $G[V^-]$ will be large enough to embed $T^+$. By directional duality, we can assume that the latter case holds. This allows us to find a copy of $T^+$ in $G$ with $t$ embedded to $v_t$, a vertex of $G$ which has enough in-neighbours for us to greedily embed the components of $T^--t$. Of course, some of these in-neighbours may be occupied already by the embedding of $T^+$, but, by embedding $T^+$ in such a way to cover minimally many of these in-neighbours we will have there are enough in-neighbours to embed the components of $T^--t$, and complete the embedding.

For Theorem~\ref{thm:extending-distillation}, we do this in the setting of distillations, random homomorphisms and a weighted looped digraph $D$. We ultimately wish to find a random homorphism from $H$, where $H$ is the fully-looped oriented forest with vertex and edge sets given by
\begin{equation*}
\HDefinitionEquation
\end{equation*}
For each $\diamond\in\{+,-\}$, let $X^\diamond=\{x^\diamond,\bar{x}^\diamond\}$, and let $X=X^+\cup X^-$. 

Note that $H_0\cong H[\{x^+,y^+,z^+,u^+,w^+,\bar{x}^+,\bar{z}^+,\bar{u}^+,\bar{w}^+\}]$. Therefore, we will assume equality here by letting, for example, $x^+=x$. In addition, let $H_0'=H-V(H_0)$, and note that $H_0'$ is itself isomorphic to a copy of $H_0$ with all edges reversed. Here $H_0$ and $H_0'$ correspond to $T^+$ and $T^-$ in the sketch above.

Instead of partitioning $V(G)$ as $V^+\cup V^-$ we partition $V(D)$ as $J^+\cup J^-$ in a similar manner, and assume, by directional duality, that $J^-$ is large enough that we can apply Theorem~\ref{thm:overarching} to get a random homomorphism of $H_0$ into $D[J^-]$ satisfying \ref{sub-out-1}--\ref{sub-out-3} below (comparable to the embedding of $T^+$ into $G[V^-]$ in the sketch above) before minimising a certain property (comparable to the embedding of $T^+$ using minimally many in-neighbours of $v_t$ above). We then use the minimisation of the random homomorphism to extend it to cover $H_0'$, so that we have a random homorphism of $H$ into $D$. Finally, we adjust this random homomorphism to get the additional condition \itref{ext-good-plan-4} which we dropped for Theorem~\ref{thm:overarching}, completing the proof.

\begin{proof}[Proof of Theorem~\ref{thm:extending-distillation}]
For $\diamond\in\{+,-\}$, let \begin{equation*}
    \lambda^\diamond=\beta(x^\diamond,y^\diamond,z^\diamond,u^\diamond,w^\diamond,\bar{x}^\diamond,\bar{z}^\diamond,\bar{u}^\diamond,\bar{w}^\diamond)
\end{equation*}
so that $\lambda^++\lambda^-=1$, and let
\begin{equation*}
    \gamma^\diamond=\max{\{\beta(x^\diamond,\bar{x}^\diamond),\beta(z^\diamond,\bar{z}^\diamond)\}}/\lambda^\diamond.
\end{equation*}
For $\diamond\in\{+,-\}$, define
\begin{equation*}
    r^\diamond=\left\lceil\frac{\lambda^\diamond(1+\gamma^\diamond+\alpha/16)}{1+\gamma+\alpha/4}\cdot r\right\rceil,
\end{equation*}
so that $r^++r^-\leqslant(1-\epsilon)\cdot r$.

Let $K$ be an $\epsilon$-almost tournament with vertex set $[r]$, such that $d(i,j)\geqslant1/2$ whenever $i\rightarrow_Kj$. Partition $[r]$ as $J^+\cup J^-$ such that $d_K^\diamond(j)\geqslant r^\diamond$ whenever $j\in J^\diamond$. Note that we either have $|J^+|\geqslant r^-$ or $|J^-|\geqslant r^+$. By directional duality, we may assume that $|J^-|\geqslant r^+$.

Let $\beta_0:V(H_0)\to[0,1]$ be given by $\beta_0(v)=\beta(v)/\lambda^+$, and note that $\sum_{v\in V(H_0)}\beta_0(v)=1$ and $\beta_0(y^+)\geqslant\beta_0(x^+)$. By Theorem~\ref{thm:overarching}, if $\mathcal{H}_0=(H_0,X^+,\beta_0)$, then $\{\mathcal{H}_0\}$ is $\gamma^+$-good. Therefore, because $\beta=\lambda^+\cdot\beta_0$ and $\lambda^+\cdot\frac{1+\gamma^++\alpha/16}{r^+}\leqslant\frac{1+\gamma+\alpha/4}{r}$, there exists some $j_t\in J^-$ and random $\psi:H_0\to D[J^-]$ such that the following hold.
\stepcounter{propcounter}
\begin{enumerate}[label = {\bfseries{\Alph{propcounter}\arabic{enumi}}}]
    \item\label{sub-out-1} With probability 1, we have that $\psi$ is a homomorphism from $H_0$ to $D$, and that $j_t\notin\psi(X^+)$.
    \item\label{sub-out-2} For each $j\in[r]$, $\E(\beta(\psi^{-1}(j)))\leqslant\frac{1+\gamma+\alpha/4}{r}$.
    \item\label{sub-out-3} For each $j\in[r]$, $\E(\beta(\psi^{-1}(j)\cap X^+))\leqslant d(j_t,j)\cdot\frac{1+\gamma+\alpha/4}{r}$.
\end{enumerate}
Fix such a $j_t\in J^-$. Let $A=N_K^+(j_t)$ and $B=N_K^-(j_t)$. Take a random $\psi:H_0\to D$ satisfying \ref{sub-out-1}-\ref{sub-out-3} so that  $\E(\beta(\psi^{-1}(B)))$ is minimised. (To see that this minimum can be attained, note that with $j_t$ fixed we may consider a random $\psi:H_0\to D[J^-]$ satisfying \ref{sub-out-1} as an element of the probability simplex $\Delta^{m-1}$, where $m$ is the number of homomorphisms from $H_0$ to $D[J^-]$ satisfying \ref{sub-out-1}; then, regarding $\Delta^{m-1}$ as a compact topological space, the random homomorphisms also satisfying \ref{sub-out-2} and \ref{sub-out-3} form a closed subset of $\Delta^{m-1}$, and hence the continuous function $f(\psi)=\mathbb{E}(\beta(\psi^{-1}(B)))$ attains its minimum over this set.)

\begin{claim}\label{clm:in-weight}
$|B|\cdot\frac{1+\gamma+\alpha/4}{r}-\E(\beta(\psi^{-1}(B)))-\E(\beta(\psi^{-1}(B)\cap X^+))\geqslant\lambda^-+\beta(X^-)+\alpha/8$.
\end{claim}
\begin{proof}[Proof of Claim~\ref{clm:in-weight}]
First, if $|A|\cdot\frac{1+\gamma+\alpha/4}{r}-\E(\beta(\psi^{-1}(A))-\E(\beta(\psi^{-1}(A)\cap X^+))\leqslant0$, then
\begin{align*}
    |B|\cdot\textstyle\frac{1+\gamma+\alpha/4}{r}\displaystyle-&\E(\beta(\psi^{-1}(B)))-\E(\beta(\psi^{-1}(B)\cap X^+))\geqslant|D|\cdot\textstyle\frac{1+\gamma+\alpha/4}{r}\displaystyle-3\epsilon-\lambda^+-\beta(X^+)\\
    &\geqslant1+\gamma+\alpha/8-\lambda^+-\beta(X^+)\geqslant\lambda^-+\beta(X^-)+\alpha/8.
\end{align*}
So we may assume that $|A|\cdot\frac{1+\gamma+\alpha/4}{r}-\E(\beta(\psi^{-1}(A)))-\E(\beta(\psi^{-1}(A)\cap X^+))>0$, else the claim is proven. In particular, we may assume there is some $j\in A$ such that, if
\begin{align*}
    p&:=\frac{1}{2}\left(\frac{1+\gamma+\alpha/4}{r}-\E(\beta(\psi^{-1}(j)))-\E(\beta(\psi^{-1}(j)\cap X^+))\right)\\
    &\leqslant d(j_t,j)\cdot\frac{1+\gamma+\alpha/4}{r}-\E(\beta(\psi^{-1}(j)\cap X^+)),
\end{align*}
then $0<p<1$. In addition, we may assume that $\mathbb{E}(\beta(\psi^{-1}(B)))>0$, else the claim follows immediately from the definition of $J^-$. Then, however, if we define $\hat{\psi}:H_0\to D$ by setting $\hat{\psi}(V(H_0))=j$ with probability $p$ and sampling $\psi$ otherwise, we find $\hat{\psi}$ satisfies \ref{sub-out-1}-\ref{sub-out-3}, but $\E(\beta(\hat{\psi}^{-1}(B)))=(1-p)\cdot\E(\beta(\psi^{-1}(B)))$, a contradiction to the minimality of $\E(\beta(\psi^{-1}(B)))$.
\renewcommand{\qedsymbol}{$\boxdot$}
\end{proof}
\renewcommand{\qedsymbol}{$\square$}

Consider the random $\hat{\phi}:H\to D$ defined by sampling $\psi$ to determine $\hat{\phi}|_{V(H_0)}$, and independently choosing $\hat{\phi}(V(H_0'))\in B$ at random so that
\begin{equation*}
    \mathbb{P}(\hat{\phi}(V(H_0'))=j)=p_j:=\frac{\frac{1+\gamma+\alpha/4}{r}-\E(\beta(\psi^{-1}(j)))-\E(\beta(\psi^{-1}(j)\cap X^+))}{|B|\cdot\frac{1+\gamma+\alpha/4}{r}-\E(\beta(\psi^{-1}(B)))-\E(\beta(\psi^{-1}(B)\cap X^+))}.
\end{equation*}
We remark that for every $j\in B$, we have
\begin{equation}\label{eq:in-prob}
    p_j\cdot\max{\{\lambda^-,2\beta(X^-)\}}\leqslant p_j\cdot(\lambda^-+\beta(X^-))\overset{\text{Claim~\ref{clm:in-weight}}}{\leqslant}\textstyle\frac{1+\gamma+\alpha/4}{r}\displaystyle-\E(\beta(\psi^{-1}(j)))-\E(\beta(\psi^{-1}(j)\cap X^+)).
\end{equation}
Therefore, if $j\in B$, we have
\begin{flalign*}
    &\hspace{0.75cm}\E(\beta(\hat{\phi}^{-1}(j)))=\E(\beta(\psi^{-1}(j)))+p_j\cdot\lambda^-\overset{\eqref{eq:in-prob}}{\leqslant}\frac{1+\gamma+\alpha/4}{r},&\\
    &\hspace{0.75cm}\E(\beta(\hat{\phi}^{-1}(j)\cap X^+))\overset{\text{\ref{sub-out-3}}}{\leqslant} d(j_t,j)\cdot\frac{1+\gamma+\alpha/4}{r},\\
    &\hspace{0.75cm}\E(\beta(\hat{\phi}^{-1}(j)\cap X))\leqslant\E(\beta(\psi^{-1}(j)\cap X^+))+p_j\cdot\beta(X^-)\\
    &\hspace{3.75cm}\overset{\eqref{eq:in-prob}}{\leqslant}\frac{1}{2}\left(\frac{1+\gamma+\alpha/4}{r}\right)\leqslant d(j,j_t)\cdot\frac{1+\gamma+\alpha/4}{r},
\end{flalign*}
whereas if $j\in [r]\setminus B$, we have
\begin{flalign*}
    &\hspace{0.75cm}\E(\beta(\hat{\phi}^{-1}(j)))=\E(\beta(\psi^{-1}(j)))\overset{\text{\ref{sub-out-2}}}{\leqslant}\frac{1+\gamma+\alpha/4}{r},&\\
    &\hspace{0.75cm}\E(\beta(\hat{\phi}^{-1}(j)\cap X^-))=0\leqslant d(j,j_t)\cdot\frac{1+\gamma+\alpha/4}{r},\\
    &\hspace{0.75cm}\E(\beta(\hat{\phi}^{-1}(j)\cap X))=\E(\beta(\psi^{-1}(j)\cap X^+))\overset{\text{\ref{sub-out-3}}}{\leqslant} d(j_t,j)\cdot\frac{1+\gamma+\alpha/4}{r}.
\end{flalign*}

Take $\bar{\phi}:H\to D$ with $\sum_{e\in E(H)}\mathbb{P}(|\bar{\phi}(e)|=1)$ minimal, such that the following properties hold.
\stepcounter{propcounter}
\begin{enumerate}[label = {\bfseries{\Alph{propcounter}\arabic{enumi}}}]
\item\label{ext-good-plan-partial-1} With probability 1, we have that $\bar{\phi}$ is a homomorphism from $H$ to $D$, and that $j_t\notin\bar{\phi}(\{x^+,\bar{x}^+,x^-,\bar{x}^-\})$.
\item \label{ext-good-plan-partial-2} For each $j\in[r]$, $\E(\beta(\bar{\phi}^{-1}(j)))\leq \frac{1+\gamma+\alpha/2}{r}-\frac{\alpha^2}{r}\cdot\sum_{e\in E(H)}\mathbb{P}(|\bar{\phi}(e)|=1)$.
\item\label{ext-good-plan-partial-3} For each $j\in[r]$, either
\begin{enumerate}[label = {\bfseries {\Alph{propcounter}\arabic{enumi}.\arabic{enumii}}}]
    \item\label{ext-good-plan-partial-3-1} $\E(\beta(\bar{\phi}^{-1}(j)\cap X^+))\leq d(j_t,j)\cdot \frac{1+\gamma+\alpha/2}{r}$ and $\E(\beta(\bar{\phi}^{-1}(j)\cap X))\leq d(j,j_t)\cdot\frac{1+\gamma+\alpha/2}{r}$, or
    \item\label{ext-good-plan-partial-3-2} $\E(\beta(\bar{\phi}^{-1}(j)\cap X^-))\leq d(j,j_t)\cdot \frac{1+\gamma+\alpha/2}{r}$ and $\E(\beta(\bar{\phi}^{-1}(j)\cap X))\leq d(j_t,j)\cdot\frac{1+\gamma+\alpha/2}{r}$.
\end{enumerate}
\end{enumerate}
$\bar{\phi}$ is well-defined, as we may take $\bar{\phi}=\hat{\phi}$.

We will shortly prove the following claim.

\begin{claim}\label{clm:expected-singletons}
$\mathbb{P}(|\bar{\phi}(e)|=1)\leqslant \epsilon^{1/4}/|E(H)|$ for every non-looped edge $e$ of $H$.
\end{claim}

From Claim~\ref{clm:expected-singletons} it follows that $\mathbb{P}(\text{$|\bar{\phi}(e)|=2$ for every non-looped edge $e$ of $H$})\leqslant\epsilon^{1/4}$. Then, if we take $\phi$ to be $\bar{\phi}$ conditioned on the event that $|\bar{\phi}(e)|=2$ for every non-looped edge $e$ of $H$, we have
\begin{equation*}
    \E(\beta(\phi^{-1}(j)\cap\{v\}))\leqslant(1-\epsilon^{1/4})^{-1}\cdot\E(\beta(\bar{\phi}^{-1}(j)\cap\{v\}))
\end{equation*}
for every $j\in[r],v\in V(H)$, thus $\phi$ satisfies the conclusion of the theorem. So it only remains to prove Claim~\ref{clm:expected-singletons}.

\begin{proof}[Proof of Claim~\ref{clm:expected-singletons}]

Suppose for contradiction that there is some non-looped edge $e$ of $H$ with $\mathbb{P}(|\bar{\phi}(e)|=1)\geqslant \epsilon^{1/4}/|E(H)|$. Let $e=v_1v_2$. For $i\in[2]$, let $H_i$ be the component of $H-e$ containing $v_i$. Assume, by directional duality, that $X\cap V(H_2)=\emptyset$. To prove the claim, we will modify $\bar{\phi}$ by switching the image of $H_2$ in certain occurrences where $e$ is mapped onto a single vertex by $\bar{\phi}$. The result will be a new random homomorphism with a smaller value of $\mathbb{P}(|\bar{\phi}(e)|=1)$. Because the image of $X$ is unchanged by this modification, the new homomorphism will automatically satisfy \ref{ext-good-plan-partial-1} and \ref{ext-good-plan-partial-3}. Additionally, even though this modification may increase $\mathbb{E}(\beta(\bar{\phi}^{-1}(j)))$ for some values of $j$, \ref{ext-good-plan-partial-2} will still hold as this increase is sufficiently offset in the relevant inequality by the reduction in $\mathbb{P}(|\bar{\phi}(e)|=1)$.

Because $\beta(v_1,v_2)\geqslant2\mu$, \ref{ext-good-plan-partial-2} implies that $\mathbb{P}(\bar{\phi}(e)=\{j\})\leqslant\frac{2}{\mu r}$ for every $j\in[r]$. So, if $J$ is the set of $j\in[r]$ for which $\mathbb{P}(\bar{\phi}(e)=\{j\})\geqslant \frac{\sqrt{\epsilon}}{r}$, then $|J|\geqslant\sqrt{\epsilon}\cdot r$, else we find $\mathbb{P}(|\bar{\phi}(e)|=1)=\sum_{j\in[r]}\mathbb{P}(\bar{\phi}(e)=\{j\})\leqslant\frac{2\sqrt{\epsilon}}{\mu}+\sqrt{\epsilon}<\epsilon^{1/4}/|E(H)|$.

Let $m$ be the number of possible homomorphisms $H\to D$. Choose homomorphisms $\phi_j, j\in J$ such that $\phi_j(e)=\{j\}$ and $\mathbb{P}(\bar{\phi}=\phi_j)\geqslant\frac{\sqrt{\epsilon}}{mr}$ for every $j\in J$ (these can be found as the edge weights of $D$ are $\epsilon$-complete).

Set $s=\lceil1/\alpha^3\rceil$. Let $j_1,\ldots,j_{s+1}\in J$ be distinct such that $d(j_i,j_{i+1})>0$ for every $i\in[s]$. Let $k\in[2]$ be random, with distribution coupled with $\bar{\phi}$ such that $\mathbb{P}(k=2)=\frac{(s+1)\sqrt{\epsilon}}{mr}$, and that $\mathbb{P}(\bar{\phi}=\phi_{j_i}\mid k=2)=\frac{\sqrt{\epsilon}}{mr}$ for every $i\in[s+1]$.

Define a random $\psi:H\to D$ as follows. Sample $(\bar{\phi},k)$, choose $i\in[s+1]$ uniformly at random, and set
\begin{equation*}
    \psi(v)=
    \begin{cases*}
        \bar{\phi}(v) & if $k=1$, \\
        \phi_{j_i}(v) & if $k=2$, $i=s+1$,\\
        \phi_{j_i}(v) & if $k=2$, $i\in[s]$, $v\in V(H_1)$,\\
        \phi_{j_{i+1}}(v) & if $k=2$, $i\in[s]$, $v\in V(H_2)$.
    \end{cases*}
\end{equation*}

We then find $\mathbb{P}(|\psi(e)|=1)=\mathbb{P}(|\bar{\phi}(e)|=1)-\frac{s\sqrt{\epsilon}}{mr}$, yet $\E(\beta(\psi^{-1}(j)\cap X))=\E(\beta(\bar{\phi}^{-1}(j)\cap X))$ and $\E(\beta(\psi^{-1}(j)))\leqslant\E(\beta(\bar{\phi}^{-1}(j)))+\frac{\sqrt{\epsilon}}{mr}$ for every $j\in[r]$. Thus, \ref{ext-good-plan-partial-1}-\ref{ext-good-plan-partial-3} also hold for $\psi$, a contradiction to the minimality of $\sum_{e\in E(H)}\mathbb{P}(|\bar{\phi}(e)|=1)$.
\renewcommand{\qedsymbol}{$\boxdot$}
\qedhere
\renewcommand{\qedsymbol}{$\square$}
\qedsymbol
\end{proof}
\renewcommand{\qedsymbol}{}
\end{proof}
\renewcommand{\qedsymbol}{$\square$}

\section*{Acknowledgements}

We thank Peter Allen and Eoin Long for their detailed feedback on this research, which was provided during the examination of the first author's PhD thesis. We also thank the anonymous referees for their work and their many helpful comments.

\bibliographystyle{abbrv}
\bibliography{references}

\begin{thebibliography}{10}

\bibitem{ALO_SHA}
N.~Alon and A.~Shapira.
\newblock Testing subgraphs in directed graphs.
\newblock {\em Journal of Computer and System Sciences}, 69:354--382, 2004.

\bibitem{BEN}
A.~Benford.
\newblock {\em On the appearance of oriented trees in tournaments}.
\newblock PhD thesis, University of Birmingham, 2023.

\bibitem{BEN-MON}
A.~Benford and R.~Montgomery.
\newblock Trees with few leaves in tournaments.
\newblock {\em Journal of Combinatorial Theory, Series B}, 155:141--170, 2022.

\bibitem{DRO-HAV}
F.~Dross and F.~Havet.
\newblock On the unavoidability of oriented trees.
\newblock {\em Journal of Combinatorial Theory, Series B}, 151:83--110, 11 2021.

\bibitem{El_S}
A.~El~Sahili.
\newblock Trees in tournaments.
\newblock {\em Journal of Combinatorial Theory, Series B}, 92:183--187, 2004.

\bibitem{intro-random-graphs}
A.~Frieze and M.~Karoński.
\newblock {\em \textbf{\emph{Introduction to Random Graphs}}}.
\newblock Cambridge University Press, 2015.

\bibitem{HAE-THO}
R.~H{\"a}ggkvist and A.~Thomason.
\newblock Trees in tournaments.
\newblock {\em Combinatorica}, 11:123--130, 1991.

\bibitem{HAV}
F.~Havet.
\newblock Trees in tournaments.
\newblock {\em Discrete Mathematics}, 243(1):121--134, 2002.

\bibitem{HAV2}
F.~Havet.
\newblock On unavoidability of trees with {$k$} leaves.
\newblock {\em Graphs and Combinatorics}, 19:101--110, 2003.

\bibitem{havet2000median}
F.~Havet and S.~Thomass{\'e}.
\newblock Median orders of tournaments: A tool for the second neighborhood problem and {S}umner's conjecture.
\newblock {\em Journal of Graph Theory}, 35(4):244--256, 2000.

\bibitem{HAV-THO}
F.~Havet and S.~Thomass\'e.
\newblock Oriented {H}amiltonian paths in tournaments: A proof of {R}osenfeld's conjecture.
\newblock {\em Journal of Combinatorial Theory, Series B}, 78(2):243--273, 2000.

\bibitem{KAT-MON}
A.~Kathapurkar and R.~Montgomery.
\newblock Spanning trees in dense directed graphs.
\newblock {\em Journal of Combinatorial Theory, Series B}, 156:223--249, 2022.

\bibitem{KUE-MYC-OST}
D.~K\"uhn, R.~Mycroft, and D.~Osthus.
\newblock An approximate version of {S}umner's universal tournament conjecture.
\newblock {\em Journal of Combinatorial Theory, Series B}, 101(6):415–447, 2011.

\bibitem{KUE-MYC-OST-2}
D.~Kühn, R.~Mycroft, and D.~Osthus.
\newblock A proof of {S}umner’s universal tournament conjecture for large tournaments.
\newblock {\em Proceedings of the London Mathematical Society}, 102(4):731–766, 2010.

\bibitem{MON-POK-SUD}
R.~Montgomery, A.~Pokrovskiy, and B.~Sudakov.
\newblock Embedding rainbow trees with applications to graph labelling and decomposition.
\newblock {\em Journal of the European Mathematical Society}, 22(10):3101--3132, 2020.

\bibitem{MYC-NAI}
R.~Mycroft and T.~Naia.
\newblock Unavoidable trees in tournaments.
\newblock {\em Random Structures \& Algorithms}, 53(2):352--385, 2018.

\bibitem{reid1983embedding}
K.~Reid and N.~Wormald.
\newblock Embedding oriented {$n$}-trees in tournaments.
\newblock {\em Studia Sci. Math. Hungar}, 18(2-4):377--387, 1983.

\bibitem{THO}
A.~Thomason.
\newblock Paths and cycles in tournaments.
\newblock {\em Transactions of the American Mathematical Society}, 296:167--180, 1986.

\end{thebibliography}

\end{document}